\documentclass[10pt,a4paper]{article}
\usepackage{typearea}
\usepackage{amsfonts}
\usepackage{amssymb}
\usepackage{amsthm}
\usepackage{amscd}
\usepackage[centertags]{amsmath}
\usepackage{eucal}
\usepackage{epsfig}
\usepackage[matrix,arrow,curve]{xy}
\CompileMatrices
\typearea{11}
\author{Andreas Klein}
\title{Symplectic monodromy, Leray residues and quasi-homogeneous polynomials}

\newcommand{\R}{\mathbb{R}}

\newcommand{\C}{\mathbb{C}}

\newcommand{\leftsc}{\langle}
\newcommand{\rightsc}{\rangle}
\newcommand{\suchthat}{\; : \;}

\renewcommand{\ker}{\mathrm{ker}}

\newtheorem{theorem}{Theorem}[section]
\newtheorem{Def}[theorem]{Definition}
\newtheorem{prop}[theorem]{Proposition}
\newtheorem{lemma}[theorem]{Lemma}
\newtheorem{folg}[theorem]{Corollary}

\newtheorem{ass}[theorem]{Assumption}
\newtheorem{conj}[theorem]{Conjecture}
\begin{document}
\maketitle
\begin{abstract}
We formulate certain sufficient conditions for the symplectic monodromy of an isolated quasihomogeneous singularity to be of infinite order in the relative symplectic mapping class group of the Milnor fibre and give a proof using Maslov classes, stability theory for Lagrangian folds resp. stable Morse theory for generating families as well as algebraic results about relative cohomology of smoothings of isolated singularities. Our conditions being slightly more restrictive than Seidel's, in contrary to Seidel's proof, we do not use Floer theory to derive this result. An alternative approach using bounding disks in fibred Lagrangian families is given and its possible application to generalizations to the non-quasihomogeneous case is discussed.
\end{abstract}
\section{Introduction}
The fundamental question of this work arose as a consequence of Paul Seidel's work on the kernel of the homomorphism
\begin{equation}
\pi_0({\rm Symp}(F,\partial F,\omega))\rightarrow \pi_0({\rm Diff}(F,\partial F)),
\end{equation}
where here $F$ denotes the Milnor fibre of a quasihomogenous polynomial $f:\mathbb{C}^{n+1}\rightarrow \mathbb{C}$ with isolated singularity in $0 \in \mathbb{C}^{n+1}$, that is there are integers $\beta_0, \dots\beta_n,\beta > 0$ such that $f\circ \sigma(t)=t^\beta f$ for any $t\in \mathbb{C}^*$, where $\sigma(t)(z_0,\dots,z_n)=(t^{\beta_0} z_0,\dots,t^{\beta_n} z_n)$ is the weighted $\mathbb{C}^*$-action on $\mathbb{C}^{n+1}$. $F$ can be considered as a symplectic (complex) submanifold of $\mathbb{C}^{n+1}$ with its standard symplectic structure $\Omega_{\mathbb{C}^{n+1}}$ and the groups in question denote the relative isotopy group of symplectomorphisms resp. diffeomorphisms fixing the boundary $\partial F$ pointwise ('relative' here means the isotopy fixes the boundary pointwise). It is known as a consequence of work of Stevens \cite{stevens} resp. Kauffman and Krylov \cite{krylov} that if one represents the bundle $\tilde Y$ as a mapping cylinder
\[
\tilde Y=F\times [0,1]/(x,0)\sim (\rho(x),1)=:F_\rho
\]
for some element $\rho \in {\rm Diff}(F,\partial F)$, then under certain conditions, $\rho$ is of {\it finite} order in $\pi_0({\rm Diff}(F,\partial F))$, namely if $n \geq 4$, $n$ even and $V(\rho^d)=0$, where $V$ is the 'variation mapping' of $\tilde Y$ (cf. \cite{krylov}), then $\rho^{4d}=id$, while for $n=6$ we have $\rho^d=id$. This implies the diffeomorphism periodicity of the branched cyclic covers $k>1$ of $S^{2n+1}, \ n\geq 2$, $n$ even, given by the union
\[
M_k=F_{\rho^k}\cup_r (\partial F\times D^2)
\]
where $r:\partial F_{\rho^k}\rightarrow F\times S^1$ is an appropriate glueing map. In fact in \cite{krylov} it is proven that for any even $k\in \mathbb{N}$, $M_k$ is diffeomorphic to $M_{k+4d}$ if $V(\rho^d)=0$, while for the cases $n=2,6$ we have similarly $M_k\simeq M_{k+d}$. Using Cerf's \cite{cerf} theorem to identify isotopy and pseudoisotopy for $n\geq 3$ and the proof of Corollary 4 in \cite{krylov} one arrives at the above claims on $\pi_0({\rm Diff}(F,\partial F))$. Note that such periodicities do not follow trivially from the fact that $h=\sigma(e^{2\pi i/\beta})$ has order $\beta$ in ${\rm Diff}(F)$, since $h$ does not fix $\partial F$ pointwise. However, $h^\beta=id$ shows $(h^*)^{\beta}=id$ while it is well-known (see e.g. \cite{stevens}), that for links of isolated hypersurface singularities $(h^*)^{\beta}=id$ and $K_f$ being a rational homology sphere are equivalent to $V(\rho^\beta)=0$, if $\rho$ represents $h$ in $\pi_0({\rm Diff}(F,\partial F))$, implying that in such cases (and $n\geq 3$) $\rho$ has finite order in $\pi_0({\rm Diff}(F,\partial F))$. Now Seidel shows (\cite{seidel}) that if $n\geq 2$ and under the condition
\begin{equation}\label{seidelcondition}
m(f)=\sum_{i=1}^{\mu}\beta_i-\beta\neq 0,
\end{equation}
the symplectic monodromy $\rho \in \pi_0({\rm Symp}(F,\partial F,\omega))$ is of infinite order. Note that $Y$ carries the structure of a symplectic fibration $(Y,\Omega)$ and symplectic parallel transport defined by the $\Omega$-orthogonal complement to the fibres defines an action $\rho:\pi_1(S^1)\rightarrow {\rm Symp}(F,\partial F,\omega)$, we will call the image of a generator $1$ of $\pi_1(S^1)$ $\rho:=\rho(1)$ the {\it symplectic monodromy} of $f$, as will be explained in some detail in Section \ref{chapter2sympl}. Now, while $m(f)$ equals the evaluation of some Maslov class on $F$ on a path of Lagrangian subspaces $t \mapsto L_t$ over $(TF|\partial F,\omega)$, where $\omega$ denotes the symplectic form on $F$ induced by restriction of $\Omega$ (cf. \cite{seidel}), it has an interpretation as an element of the spectrum of $f$. Recall that the {\rm spectrum} ${\rm Sp}(f)$ of the quasihomogeneous singularity (see Definition \ref{spectrumdef}) is a set of rational numbers $\{\gamma_i\}_{i=1,\dots,\mu}$ being defined as the normalized logarithm $\gamma_i= (-1/2\pi i){\rm log} \lambda_i$ of the eigenvalues $\lambda_i, i=1\dots,\mu$ of the monodromy. Here, the normalization is determined by the asymptotic Hodge filtration on the 'canonical' Milnor fibre, for that terminology, see for instance Kulikov (\cite{kulikov}) resp. Section \ref{relcohom}. Since $f$ is quasihomogeneous, in terms of a monomial basis $(z^{\alpha(i)})_{i=1}^\mu$, where $\alpha(i) \in \Lambda \subset \mathbb{N}^{n+1}, \ |\Lambda|=\mu, \alpha(1)=0$, of the Milnor algebra 
\begin{equation}\label{mfdecomp}
M(f):=\mathcal{O}_{\mathbb{C}^{n+1},0}/(\frac{\partial f}{\partial z_0},\dots,\frac{\partial f}{\partial z_n})\mathcal{O}_{\mathbb{C}^{n+1},0}=\oplus_l M(f)_l
\end{equation}
one has $\gamma_i= l(\alpha(i))-1$, where $l(\alpha(i))=\sum_k(\alpha(i)_k+1)w_k$, $w_k=\beta_k/\beta$, hence written as a 'divisor', ${\rm Sp}(f)=\sum_{\alpha(i) \in \Lambda} \left(l(\alpha(i))-1\right) \in \mathbb{Z}^{(\mathbb{Q})}$ (\cite{kulikov}) and we have $m(f)=\beta\gamma_1$. On the other hand, each element $\gamma_i$ is associated to a section of the sheaf $\mathcal{H}''=f_*(\Omega^{n+1}_X)/df\wedge d(f_*\Omega^{n-1}_{X/D})$ over a small open disk $D\subset \mathbb{C}$ containing $0$, the so-called Brieskorn lattice. Then the map $\phi$ sending $z^{\alpha(i)}$ to the class given by
\[
\phi(z^{\alpha(i)})=z^{\alpha(i)}dz_0\wedge\dots\wedge dz_n
\]
in $\mathcal{H}''_{0}$, defines a $\mathbb{C}$-isomorphism of vectorspaces $\phi:M(f)\simeq \mathcal{H}''_{0}/f\mathcal{H}''_{0}$, being, since $\mathcal{H}''$ is coherent, even free (\cite{bries}), also an isomorphism of the respective $\mathcal{O}_{D,0}$-modules. Finally recall that $\mathcal{H}^n_{X/D}|D^*\simeq \mathcal{H}''|D^*$, where $\mathcal{H}^n_{X/D}$ is the relative cohomology sheaf on $D$ and the isomorphism is given by the Poincare-Leray-residue (cf. Malgrange \cite{malg2}, Section \ref{relcohom}).
Then, the above facts gave rise to the following question:\\

{\it Is there a 'singularity-theory' proof of the fact that the map $\pi_0({\rm Symp}(F,\partial F,\omega))\rightarrow \pi_0({\rm Diff}(F,\partial F))$, for $F$ being the Milnor fibre of a quasihomogeneous polynomial, has an infinite kernel if $K_f$ is a rational homology sphere, $n\geq 4$, $n$ even and condition (\ref{seidelcondition}) holds true? Does the non-vanishing of {\it any} element of the spectrum of $f$ imply the above result?}\\

Note that the latter conditions are not mutually exclusive, in fact, $K_f$ being a rational homology sphere implies the non-vanishing of all $\gamma_i, i=1, \dots,\mu$ by Lemma \ref{condition}. To connect the number $\gamma_1= l(0)-1$ to the symplectic geometry of $Y$ consider the $\beta$-fold cyclic covering of $Y$ (note that we work with the 'deformed' Milnor bundle as described in (\ref{milnorfibr}) which has a well-defined symplectic parallel transport)
\[
Y^\beta\simeq F\times [0,1]/(x,0)\sim (\rho^\beta(x),1),
\]
where here, $\rho \in {\rm Symp}(F,\partial F,\omega)$ is the symplectic monodromy of $Y$ and we consider $\pi^\beta:Y^\beta\rightarrow Y$ as a {\it symplectic} covering by lifting the corresponding structure on $Y$. Then lift the section $s_0 \in \Gamma(\mathcal{H}^n_{X/D^*}|S^1)$ being associated to $\alpha(1)=0$, so to $\gamma_1$ by the above, to a section $s_0^\beta=(\pi^\beta)^*s_0 \in \Gamma(\mathcal{H}^n_{X^\beta/D^*}|S^1)$, that is to a smooth section of the bundle of fibrewise $n$-th cohomology of $Y^\beta$ and assume that $Q_{z_0}\subset F$ is any smooth Lagrangian cycle so that $[{\rm ev}_F(s_0)]\cup PD[Q_{z_0}](F)\neq 0$ and $[{\rm ev}_F(s_0)]\in {\rm im }\ H^n(F,\partial F, \mathbb{C})\rightarrow  H^n(F,\mathbb{C})$. Then if $Q_\tau=\mathcal{P}_\tau(Q_{z_0})$, where $\mathcal{P}_t$ is symplectic parallel transport along the null direction of $(\pi^\beta)^*(\Omega|Y)$ along the path $\tau\mapsto e^{2\pi i\tau}z_0$ in $D$ then (Lemma \ref{bla347})
\begin{equation}\label{cohwind}
\beta\cdot\gamma_1={\rm wind}({\rm ev}(s_0^\beta)(Q_\tau))_{\tau \in [0,1]}.
\end{equation}
On the other hand, $Y^\beta$ sits as a submanifold in the symplectic covering $X^\beta\rightarrow X$ of $f:X\rightarrow D^*$ and $Q:=\bigcup_{\tau\in [0,1]}Q_\tau$ is (given it defines a closed submanifold) Lagrangian in $X^\beta$. Note that the condition that $Q$ 'closes' under symplectic parallel transport is not necessary, as will be described in Section \ref{relativen0} resp. Section \ref{app3}. Now {\it assuming} $\rho^k \in \pi_0({\rm Symp}(F,\partial F,\omega))$ would be trivial, $X^\beta$ has an extension $X^\beta_e$ to the unpunctured disk as a symplectic fibration which can be assumed to be the trivial fibration $X_0$ over a neighbourhood of $0 \in D$ (Lemma \ref{extension}). Now observe that by definition of $s_0$ and the family $Q_\tau$
\begin{equation}\label{blasection2}
\tilde \alpha_0(\tau):={\rm ev}(s_0^\beta)(Q_\tau)=i_{X_f}(\pi^\beta)^*(dz_0\wedge\dots\wedge dz_n)(Q_\tau)=(e^{i\theta}i_{X_f}{\rm vol_Q})(Q_\tau),
\end{equation}
where $d(f\circ \pi^\beta)_*(X_f)=1$ and $X_f$ is horizontal and $e^{i\theta}:Q\rightarrow S^1$ is the 'phase' of $Q$ in $X^\beta$. Now it is well-known (\cite{cielie}) that $d\theta=\sigma_Q$, where $\sigma_Q$ is the mean curvature form of $Q$ in $X^\beta$. Then  we can deform $Q$ to a family of (Lagrangian) submanifolds $Q^t\subset X^\beta_e,\ t\in [0,1]$ fibred over $S_{\delta(t)}^1\subset D, \ \delta(t)\rightarrow 0, t\rightarrow 1$ into Lagrangian cycles while $Q^0=Q$ and $Q^1\subset X_0$. Further one associates to each $Q^t$ a function $\alpha_t:S^1 \rightarrow\mathbb{C}^*$. To define this note that the mean curvature form of $Q$ in the presence of a disk $D\subset X_e^k$ with boundary in $Q$ as introduced in (\cite{cielie}) decomposes along $\partial D$ as
\[
2i\sigma_Q=\eta_D+i d\vartheta,
\]
where $\eta_D$ is the connection $1$-form associated to a fixed trivialization of the canonical bundle along $D$ and $\vartheta:\partial D\subset Q\rightarrow \mathbb{R}/\mathbb{Z}$ is a function so that $d\vartheta(\partial F)$ equals the Maslov index of $D$ in $Q$. We then want to define $\alpha_t$ for $t \in [0,1]$ by a winding number of a fixed disk $D$ along its intersection with each Lagrangian $Q^t$ so that the definition is consistent at $t=0$ with (\ref{cohwind}). Assuming $\partial D$ would be a closed {\it horizontal} path in $Q$ (horizontal w.r.t. to $(\pi^k)^*(\Omega$)) (this is also not essential, which can be seen by 'pushing' $D$ towards the boundary, see Section \ref{app3}), we first show $\eta_D|\partial D=0$, which replaces the function $e^{i\theta}$ along $\partial D$ in $Q^0$ in (\ref{blasection2}) by the phase $e^{i\vartheta}$. Then the {\it key step} is to show that the winding number of $\tilde \alpha_0:S^1\rightarrow \mathbb{C}^*$ actually coincides with the winding number of the function $e^{i\theta}$ along $\partial D$. For this (Lemma \ref{tildealpha} resp. Proposition \ref{keylemma}) we need a genericity-assumption on the behaviour of a certain representative of the Poincare-dual of the Maslov-class of $Q$, resp. its intersection with a fixed fibre $F$, under a specific symplectic isotopy of $Q\cap F$, this is in detail described in Assumption \ref{ass4} and is proven in Section \ref{app4} using stability theory (for Lagrangian fold singularities), generating families and stable Morse theory. Note also it is necessary to assume that $Q_{z_0}$ is closed and $Q_{z_0}\cap \partial F =\emptyset$ (the validity of this again follows from (\ref{seidelcondition})). Then one can directly calculate (cf. Propositions  \ref{alphatau2} and \ref{keylemma}):
\begin{prop}
Assume $n\geq 2$ and $\rho^\beta= id$ in $\pi_0({\rm Symp}(F,\partial F,\omega))$. Then, using the above construction, one infers ${\rm wind}(\alpha_1)=0$.
\end{prop}
On the other hand ${\rm wind}(\alpha_1)=\beta \cdot \gamma_1=(\sum_{i=1}^{\mu}\beta_i-\beta)$ and the latter is $\neq 0$ since that is already implied by the condition that $[{\rm ev}_F(s_0)]$ represents an element of $H^n(F,\partial F,\mathbb{C})$ which is equivalent to $(\sum_{i=1}^{\mu}\beta_i-\beta)/\beta\notin \mathbb{Z}$ (Lemma \ref{condition}). So we prove, by repeating the above for any covering $X^k\rightarrow X$, $k=m\cdot \beta, m\in \mathbb{N}^+$ (cf. Theorem \ref{theorem34}):
\begin{theorem}\label{symplecticmonodromy}
Let $n\geq 2$ and assume that the evaluation of the section $s \in \Gamma(\mathcal{H}^n_{X/D})$ associated to $1 \in M(f)$ restricted to a fixed fibre is contained in ${\rm im}\ i^*:H^n(F, \partial F, \mathbb{C})\rightarrow H^n(F,\mathbb{C})$, then $\rho$ in $\pi_0({\rm Symp}(F,\partial F,\omega))$ is of infinite order.
\end{theorem}
Note that our condition implies Seidel's condition $m(f)\neq 0$, which in turn implies that $(\rho^*)^k\neq Id \in {\rm Aut}(H^n(F,\mathbb{C}))$ for $k\neq m\cdot\beta, \ m\in \mathbb{N}^+$, which is why we could restrict to an examination of the powers $X^{m\beta}$. Note that the assumption on the existence of a horizontal path in $Q$ respectively the assumption of well-definedness of $Q$ as a closed Lagrangian submanifold in $X^{m\beta}$ can be avoided which is discussed in Section \ref{relativen0} resp. Section \ref{app3}. In Section \ref{relativen0}, we give a modified proof of Theorem \ref{symplecticmonodromy}, which completely dispenses from the use of bounding disks and Lagrangians in $X^{m\beta}$ given in Section \ref{boundingdisks} (note the reversed order of presentation in this introduction) and instead makes use of a family of fibrewise nonvanishing $(n,0)$-forms enabling one to define Maslov-Indices along paths in $Y$ being lifted to the bundle of Lagrangian subspaces $\mathcal{L}^v(Y,\omega)$ of the vertical tangent bundle $T^vY$ of $Y$, i.e. a {\it fibrewise} Lagrangian submanifold $Q\subset Y^{m\beta}$ and a path in $Q$ define such data. On the other hand, Section \ref{app3} combines both approaches and should apply if one considers more general cases than a single isolated singularity at the origin, which is left for future investigation. However, the 'hard' arguments of Section \ref{app4} are essential in all approaches. In general, given a holomorphic function $f:\mathbb{C}^{n+1}\rightarrow \mathbb{C}$ with an isolated singularity at $0$, the monodromy $\rho^*\in {\rm Aut}(H^n(F,\mathbb{C}))$ is not semi-simple one can conjecture that the sum of the spectra plays the role of $\gamma_1$ in our case. In fact we conjecture, coming back to the question we posed at the beginning:
\begin{conj}\label{conjspec}
Let $f:\mathbb{C}^{n+1}\rightarrow \mathbb{C}$, $n\geq 2$, be quasihomogenous and let $\gamma_i \in \mathbb{Q}, i=1, \dots, \mu$ be the set of spectral numbers introduced in Definition \ref{spectrumdef}. If there exists $i\in \{1,\dots, \mu\}$ s.t. $\gamma_i\neq 0$, then $\rho$ in $\pi_0({\rm Symp}(F,\partial F,\omega))$ is of infinite order. If $f$ is a general polynomial with an isolated singularity at $0$ and the set $\{\gamma_i\}_{i=1}^\mu\subset \mathbb{Q}$ presents its spectrum and if $\sum_{i=1}^\mu \gamma_i\neq 0$, then $\rho$ in $\pi_0({\rm Symp}(F,\partial F,\omega))$ is of infinite order.
\end{conj}
Note that in Section \ref{generalspec}, we discuss the related Conjecture \ref{maslovspec} for the quasihomogeneous case, namely the connection between general elements of the spectrum of $f$ and certain Maslov-indizes if $\rho$ is of finite order in ${\rm Symp}(F,\omega)$. We prove an important step in that direction which again involves stability of Lagrangian folds.\\
We finally describe in Section \ref{generalpol} how the 'bounding-disk' method of Section \ref{boundingdisks} leads to the formulation of a necessary condition for the existence of certain Lagrangian submanifolds in $X^k$ in the case of a general polynomial with an isolated singularity at $0$, resp. its 'good representative' $f:X\rightarrow D$ with fibre $F$. Here, suppose $k$ is an integer chosen so that the corresponding power of the symplectic monodromy of $f$ is trivial in $\pi_0({\rm Symp}(F,\partial F,\omega))$. Then we can define the number $\alpha(\tau):={\rm ev}(s_0^k)(Q_\tau)\in \mathbb{C}$ and $s_0^k,\ Q_t$ for $\tau \in [0,1]/\{0,1\}$ exactly as in the quasihomogeneous case in (\ref{blasection2}) and we have the following simple observation (cf. Corollary \ref{generallag}):
\begin{folg}
If the set $Q:=\bigcup_{\tau\in [0,1]}Q_\tau$ defines a Lagrangian submanifold in $X^k$ then any horizontal disk $u:D\rightarrow X^k_e$ with boundary in $Q$ has Maslov Index $\mu(u)=k$. Furthermore, in case such a disk exists, $\alpha(\tau)\neq 0$ for all $\tau \in[0,1]$ and Assumption \ref{ass3} holds, then one has necessarily ${\rm wind}(\alpha)=\mu(u)-k=0$.
\end{folg}
\tableofcontents

\section{Symplectic monodromy and Lagrangian folds}\label{chapter2sympl}
\subsection{Symplectic geometry of Milnor fibrations}
Recall, for $n\geq 1, n \in \mathbb{N}$, the Milnor fibration for a (quasihomogeneous) polynomial $f$ on $\mathbb{C}^{n+1}$, with singular fibre at $0$:
\begin{equation}\label{milnor1}
f:  \hat X=\bigcup_{z \in D_\delta}\hat X_z:=f^{-1}(z) \cap B_1^{2n+2}\rightarrow D^*_\delta,
\end{equation}
where $D_\delta$ is the closed disk in $\mathbb{C}$ with radius $\delta$, $0<\delta$ is sufficiently small, $B_1^{2n+2}$ denotes the closed unit ball in $\mathbb{C}^{n+1}$. Fix a cutoff function $\psi_m:[0,1]\rightarrow [0,1]$ with $\psi_m(t^2)=1$ for $t \leq 1-2/m$ and $\psi(t^2)=0$ for $t \geq 1-1/m$ for some $m>2, m \in \mathbb{N}$ (which will be fixed below). Set $X_z=\{x \in B_1^{2n+2}|f(x) = \psi_m(|x|^2)z\}$. Then set for $0<\delta$ small enough (we will be more precise in a moment)
\begin{equation}\label{milnorfibr}
f: X\ = \ \bigcup_{ z \in D_\delta}\ X_z \times \{z\} \subset \mathbb{C}^{n+2}\longrightarrow D_\delta,
\end{equation}
we will denote the projection $f:X\rightarrow D^*_\delta$ again by $f$, since this will cause no confusion in the following. To be more explicit, (\ref{milnorfibr}) means that if one sets $\hat f(x,z)=f(x)-\psi_m(|x|^2)z$, $(x,z)\in \mathbb{C}^{n+1}\times D^*_\delta$ and
\begin{equation}\label{bigf}
F:\mathbb{C}^{n+1}\times D^*_\delta\rightarrow \C \times D_\delta^*,\quad F(x,z)=(\hat f(x,z),z),
\end{equation}
to define $X=B_1^{2n+2}\cap \hat f^{-1}(0)$ and $f:=pr_2\circ F:X \rightarrow D^*_\delta$. Let now $\theta_{\mathbb{C}^{n+1}}=\frac{i}{4}\sum_j(z_jd \overline z_j-\overline z_jdz_j) \in \Omega^1(\mathbb{C}^{n+1})$ and $\Omega_{\mathbb{C}^{n+1}}=d\theta_{\mathbb{C}^{n+1}}\in \Omega^2(\mathbb{C}^{n+1})$ the standard forms in $\mathbb{C}^{n+1}$.The following two Lemmata are essentially taken from Seidel \cite{seidel}:
\begin{lemma} \label{th:milnor-fibre}
There is an $\delta'>0$ such that for all $0 < |z| < \delta'$, and fixed $m\in \mathbb{N}$, $(X_z,\omega_z:=\Omega_{\mathbb{C}^{n+1}}|TX_z)$ is a smooth symplectic submanifold of $B^{2n+2}$ with boundary $\partial X_z$.
\end{lemma}
\begin{proof} The tangent space of $X_z$ outside $x=0, x \in \mathbb{C}^{n+1}$ is $(TX_z)_x = {\rm ker}\ L(x,z)$ with $L(x,z):\mathbb{C}^{n+1} \longrightarrow \C, z \in D_\delta$, $L(x,z)\xi = df(x)\xi - 2z\,\psi'_m(|x|^2)\leftsc x,\xi \rightsc$. Obviously $L(x,z)$ is surjective if either $z=0$ and $|x|>0$ or if $z \neq 0$ and $0 < |x| \leq 1/1-2m$ or $|x| \geq 1-1/m$, in the following we will denote the intersection of this set with each fibre $X_z$, $z\neq 0$ by $A_z(m)$, note that $A_z(m)$ is closed in $X_z$. On the other hand, the set of those $(x,z)$ for which $L(x,z)$ is onto must be open, since surjectiveness is an open condition. This implies that $L(x,z)$ is surjective for all $(x,z) \in F_z$, provided that $z \neq 0$ is sufficiently small.  Since the fact that $(TX_z)_x$ is a symplectic subspace of $\C^{n+1}$ for a given small $z \neq 0$ is also an open condition, one can argue as in the case of smoothness.
\end{proof}
Furthermore, we note that
\begin{lemma} \label{th:classical-milnor-fibre}
There is an $\epsilon>0$ such that for all $z \in \mathbb{C}, 0 < |z| < \epsilon$, $X_z$ is diffeomorphic to the 'classical' Milnor fibre to $\hat X_z=f^{-1}(z) \cap B^{2n+2}$, that is, there is a diffeomorphism $\Psi_z:X_z\rightarrow \hat X_z$ for any $z \in D^*_\epsilon$.
\end{lemma}
\begin{proof} For $(z,t) \in (\mathbb{C} \setminus \{0\}) \times [0;1]$ consider $G_{(z,t)} = \{ x \in
B^{2n+2} \suchthat f(x) = t\psi(|x|^2)z + (1-t)z\}$. Using the same argument as in the proof of Lemma \ref{th:milnor-fibre} these are smooth manifolds for all sufficiently small $z$. If we fix such a $z$, the $G_{(z,t)}$ form a differentiable fibre bundle over $[0;1]$. Hence $G_{(z,1)} = X_z$ and $G_{(z,0)} = f^{-1}(z) \cap B^{2n+2}$ are diffeomorphic. 
\end{proof}
We can now deduce
\begin{lemma}\label{symplfibres}
Choose $\delta>0$ s.t. $0<\delta< {\rm min} (\delta',\epsilon)$ with $\delta',\epsilon$ as in Lemma \ref{th:classical-milnor-fibre} resp. Lemma \ref{th:milnor-fibre}, then for any two choices $z,z'$ with $0<|z|,|z'|<\delta$ we have that $(X_{z}, \omega_z)$ is symplectomorphic to $(X_{z'},\omega_{z'})$, furthermore any two choices of $m>2, m \in \mathbb{N}$ such that $X_z$ in in (\ref{milnorfibr}) remains smooth lead to symplectomorphic fibres $X_z$.
\end{lemma}
\begin{proof}
As we will show in Lemma \ref{milnorsymplectic}, $X$, provided $0<\delta$ is small enough, is a symplectic fibration with closed $2$-form $\Omega$ on $X$ and the annihilator $T^hX$ of the vertical tangent space $T^vX$ w.r.t. $\Omega$ is tangent to $\partial_h X: =\bigcup_{z\in D^*}\partial X_z$, hence parallel transport along $T^hX$ is well defined and a fibrewise symplectomorphism. To prove the second assertion, fix any $2<m<m' \in \mathbb{N}$ and a $z$ s.t. $0<|z|<\delta$ and define a family of symplectic forms on $X_{z,m}$ by fixing a smooth family of diffeomorphisms s.t. $\tilde \Psi_{(z,s)}(X_{(z,m)})=X_{z,s}, \ s\in  [m,m']$. Here, $X_{z,s}=\{x \in B_1^{2n+2}|f(x) = \psi_s(|x|^2)z\}$. That this family exists follows since the $X_{(z,s)}$ form (for $z$ fixed) a differentiable  fibre bundle over $[m,m']$, we can assume that $\Psi_{(z,m')}(X_{(z,m)})= X_{(z,m')}$ and $\Psi_{(z,m)}=Id_{X_{(z,s)}}$. So define a family of symplectic foms on $X_{(z,m)}$ by
\[
\omega_s=\Psi_{(z,s)}^*\omega_{X_{(z,s)}},
\]
where the latter is the symplectic form on each $X_{(z,s)}$ given by Lemma \ref{th:milnor-fibre}. Now since the boundaries of the $X_{(z,s)}$ coincide, we can assume that $\Psi_{(z,s)}|\partial X_{(z,s)}=id$, hence $[\omega_s-\omega_m]\in H^2(X_{(z,m)},\partial X_{(z,m)},\mathbb{C})$. Now $H^2(X_{(z,m)},\partial X_{(z,m)},\mathbb{C})=0$ since $X_{(z,m)}$ is diffeomorphic to the Milnor fibre $\hat X_z$, so there is a smooth family of $1$-forms $\theta_s$ on $X_{(z,m)}$ vanishing on its boundary s.t.
\[
d\theta_s=\frac{d}{dt}\omega_s,\ s \in [m,m'],
\]
so by Moser's Lemma there is a family of diffeomorphisms $\Phi_s:X_{(z,m)}\rightarrow X_{(s,m)}, s\in [m,m'] $ so that $\Phi_m=id$, $\Phi_s|\partial X_{(z,m)}=id, s \in [m,m']$ and $\omega_m=\Psi_s^*(\omega_s)$, i.e. $\omega_{m'}=\Psi_{(z,m')}^*\omega_{X_{(z,m')}}$ is symplectomorphic to $\omega_m=\Omega_{\mathbb{C}^{n+1}}|X_{(z,s)}$ which is the assertion.
\end{proof}
We will show in the following that $(X, \Omega)$, for any $\delta$ chosen as above and restricted to $D^*_\delta:=D_\delta\setminus \{0\}$ (here, $\Omega$ is the restriction of $\Omega_{\mathbb{C}^{n+2}}$ to $X$), is a symplectic fibration in the following sense (the form of the definition is mainly drawn from \cite{oancea}).
\begin{Def}\label{symplfibr}
Let $\pi:F \hookrightarrow E \rightarrow S$ be a locally trivial fibration over a symplectic manifold $(S,\beta)$ with boundary, so that the fibre $F$ is a compact manifold with boundary. Let $\partial_h E$ denote the union of the boundaries of the fibres, let $\partial_v E:=\pi^{-1}(\partial S)$, so that $\partial E=\partial_h E\cup \partial _v E$ and $E$ is a manifold with corners of codimension $2$ and let for any smooth $\gamma:[0,1]\rightarrow S$ $E_\gamma:=\pi^{-1}({\rm im}\ \gamma)$. We will call $(E,\pi)$ a {\it symplectic fibration} with {\it contact type fibre boundary} if there is a $\Omega \in \Omega^2(E,\mathbb{R})$ and a vertical vector field $Z$ defined on an open neighbourhood $\mathcal{U}\subset E$ of $\partial_h E$ so that 
\begin{enumerate}
\item $\Omega$ is nondegenerate along the fibres and globally closed on $E$,
\item $Z$ is transverse and outward pointing along $\partial_h E$ and satisfies $L_Z\Omega|E_\gamma =\Omega|E_\gamma$,
for any path $\gamma$ as above.
\item For some $z \in S$ consider the trivial fibration $\tilde{\pi} : \tilde{E} = S
\times E_z \rightarrow S$, with $\tilde \Omega=pr_2^*(\Omega|E_z) +\tilde \pi^*(\beta)$ for some closed $\beta \in \Omega^2(S)$ and $\tilde Z$ satisfies $(pr_2)_*(\tilde Z)=Z_z$ using the projection $pr_2:\tilde E\rightarrow E_z$. Then there are neighbourhoods $N\subset \mathcal{U}$ of $\partial_h E$ and $\tilde N$ of $\partial_h\tilde E$ so that there is a diffeomorphism $\Theta:N\rightarrow \tilde N$ s.t. $\Theta^*\tilde \Omega=\Omega$, $\Theta_*Z=\tilde Z$ and
\begin{equation} \label{boundtriv}
\xymatrix{
 {N} \ar[rr]_{\Theta} \ar[dr]_{\pi} && {\tilde{N}} \ar[dl]^{\tilde{\pi}} \\
 & {S} &
}
\end{equation}
\end{enumerate}
For each $z \in S$ let $\omega_x:= \Omega|_{TE_z}$ and $\alpha_z:= i_Z\Omega|_{T\partial E_z}$. Then if  $[\omega_z,\alpha_z]\in H^2(E_z,\partial E_z,\mathbb{R})$ is zero we say $(E,\Omega)$ is an {\it exact} symplectic fibration.
\end{Def}
Note that $\Omega_s = \Omega + s\pi^*\beta$, $s >0$, is symplectic for $s$ big enough if $\beta$ is symplectic. Setting $\Omega_\gamma:=\Omega|E_\gamma$, $\gamma$ any path as above, $d\Omega_\gamma=0$ and $L_Z\Omega_\gamma=\Omega_\gamma$ imply that $\Omega_\gamma$ is exact on $\mathcal{U}_\gamma=\mathcal{U}\cap E_\gamma$, with primitive
\[
\Theta_\gamma=i_Z\Omega_\gamma,
\]
and with $\phi_t$ the flow of $Z$ there is a collar
\begin{equation}\label{utriv}
\begin{split}
\hat \Psi:\partial_h E_\gamma \times [1-\delta;1]\rightarrow \mathcal{U}_\gamma\\
(x,\hat S)\mapsto \phi_{\rm ln\ \hat S}(x)
\end{split}
\end{equation}
for some $\delta >0$. Further one has $\phi_t^*\Omega_\gamma=e^t\Omega_\gamma$ so (using $e^t=\hat S$)
\begin{equation}\label{ftriv}
\hat \Psi^*\Theta_\gamma=e^t\Theta|_{\partial_h E_\gamma},
\end{equation}
note that each fibre $E_z, z \in S$ is a symplectic manifold with contact type boundary. Note that by the condition (\ref{boundtriv}) and the fact that by (\ref{ftriv}) we have $\hat \Psi^*\Omega_\gamma=d(\hat S\Theta_\gamma|_{\partial_h E_\gamma})$ imply that the horizontal distribution $T^hE=({\rm ker}\  \pi_*)^{\perp_\Omega}$ is tangent to $\partial_h E$ and  
\begin{equation}\label{Striv}
d\hat S|T^hE_\gamma=0,
\end{equation}
i.e. that the paralleltransport $\tau_\gamma$ along the path $\gamma:[a,b]\rightarrow S$ is well-defined and 
\[
\tau_\gamma:E_{\gamma(a)}\rightarrow E_{\gamma(b)}
\]
a symplectic diffeomorphism. Note furthermore that there is no reason for $\Omega_\epsilon$ to be exact in $\mathcal{U}$, since $\pi^*\beta$ is not assumed to be exact (although this will be the case in the following).
\begin{ass}\label{compatiblecplx}
For a fixed complex structure $j$ on $S$ compatible with $\beta$ we call an almost complex structure $J$ on $E$ {\it compatible relative to} $j$ if the following conditions are satisfied:
\begin{enumerate}
\item $D\pi \circ J=j\circ D\pi$,
\item $\Omega(\cdot,J\cdot)$ is symmetric on $E$, positive definite along the fibres and $J(TE^h)=TE^h$,
\item In a neighbourhood $\mathcal{U}$ of $\partial_h E$, one has that $Z \perp \{x \in \mathcal{U}: \hat S={\rm const.}\}$, i.e. $Z \perp T(\partial _h E)$ with respect to $g=\Omega(\cdot, J\cdot)$.
\end{enumerate}
\end{ass}
Relative to the splitting $TE_z=TE^v_z\oplus TE^h_z \simeq TE^v_z\oplus TS_z$ where $TE^v={\rm ker\ } \pi_*$, $z=\pi(x)$, the second condition above means that $J$ is symmetric:
\begin{equation} \label{compatible-j}
 J_x = \begin{pmatrix} \pi^*(j)(x) & 0 \\ 0 & J_z^{v}(x) \end{pmatrix}.
\end{equation}
Furthermore, by the symmetry of $J$, the second property from above and since $\beta$ is compatible with $j$, it follows that for sufficiently positive $s>0$ that $\Omega_s = \Omega+s \pi^*\beta$ is compatible with $J$. Note that we have the fibrewise splitting (setting $\hat S=e^t$)
\begin{equation} \label{eq:fine-splitting}
T_xE^v_z \simeq \R Z(x) \oplus \R R(x) \oplus (\ker\, \Theta(x) \cap \ker\, dt(x) \cap
T_xE^v_z),\ x \in E\cap\mathcal{U}, \pi(x)=z,
\end{equation}
where $Z$ is the Liouville vector field of $\theta_z$, that is, $\theta_z=i_Z\Omega_z|E_z$ (here, $z=\pi(x)$) and $R$ is defined by $R=JZ$ on $\mathcal{U}$ (hence $\theta_z(R)=i_Z\omega_z(R)=g(Z,Z)>0$ and $R|\partial_h E$ is proportional to the Reeb vector field of $\alpha$). Then from the above it follows that with respect to (\ref{eq:fine-splitting}),
\begin{equation} \label{eq:boundary-j}
 J_z^{v}(x) = \begin{pmatrix} 0 & -1 & 0 \\ 1 & 0 & 0 \\ 0 & 0 & * \end{pmatrix}.
\end{equation}
From this we conclude that the Liouville vector field $Z$ defines a fibrewise metric collar on $\mathcal{U}$ by using the decomposition for $z=\pi(x) \in S$
\[
T_xE^v_z=\mathbb{R}Z(x)\oplus (\mathbb{R}Z(x))^\perp
\]
where the complement is taken with respect to the metric $g_z=\Omega(\cdot, J\cdot)|_{TE_z}$, by formula (\ref{boundtriv}) setting $g_s=\Omega_s(\cdot, J\cdot)$ these fit together to define an orthogonal splitting relative to coordinates $(t,x)$ on $\mathcal{U}, t \in [1-\delta,1], x \in \partial_h E$
\begin{equation}\label{collarmetric2}
\Psi^*(g_s|_{\mathcal{U}})(t,x)= (d\hat S)^2\oplus g_s|{\partial_h E\times\{t\}}(x,t),
\end{equation}
it is clear that this splitting can be extended along the flow of $Z$ to get a 'metric collar' $\mathcal{U}'=(0,\epsilon]\times\partial_h E$, i.e.
$g_s|{\partial_h E\times\{t\}}(x,t)$ is independent of $t$ on $\mathcal{U}'$.
\begin{lemma}\label{milnorsymplectic}
The Milnor fibration $(X,\Omega)$, given by $f:X\rightarrow D^*_{\delta}$ as in (\ref{milnorfibr}) is an exact symplectic fibration with $\Omega$ given by the restriction of $\Omega:=\Omega_{\mathbb{C}^{n+2}}$ to $X\subset \mathbb{C}^{n+2}$. Furthermore, on each fibre $X_z\subset X,\ z \in D^*_\delta$, there is a nearly complex structure $J^v_z$, which coincides on $A_z(m)\subset X_z$ for any $z \in D^*_\delta$ with the restriction of the canonical complex structure on $\mathbb{C}^{n+1}$ to $X_z$, is compatible to $\omega_z=\Omega|TX_z$, varies smoothly in $z$ and satisfies on a neighbourhood $N\subset X$ of $\partial_h X$ $J^v_{z}=({\rm pr}_2\circ \Theta_{z})^*(J_{z'}), \ z\in D_\delta^*$ for some fixed $z' \in D^*_\delta$. Setting for $\pi(x)=z$
\[
J_x=J^v_z(x) \oplus_{\Omega} f^*(j)(x),
\]
(the splitting is relative to $\Omega$) where $j$ is the canonical complex strucure on $D^*_\delta$, equips $(X,\Omega)$ with an almost complex structure so that $J|\bigcup_z A_z=J_{\mathbb{C}^{n+1}}|\bigcup_z A_z$and that is compatible to $j$ and to $\Omega$.
\end{lemma}
\begin{proof}
That there exists a fibrewise almost complex structure with the asserted properties follows by the standard procedure associating to a symplectic form $\omega_z=\Omega_{\mathbb{C}^{n+1}}|X_z$ and the fibrewise restricted metric $g_z=g_{\mathbb{C}^{n+1}}|X_z$ a (unique) compatible almost complex structure $J_z$ such that $g_z=\omega_z(\cdot, J\cdot)$, it is clear that the construction is smooth in $z$ and that $J_z$ coincides with the induced complex structure on $A_z\subset X_z$. Now define $\Omega$ by restricting $\Omega_{\mathbb{C}^{n+2}}$ to $X$, this defines a closed two-form on $X$ which is fibrewise symplectic by Lemma \ref{th:milnor-fibre}. Furthermore $d(\theta_{\mathbb{C}^{n+1}})|X_z=\omega_z$ follows by definition and it is well-known (see for instance \cite{abe}) that $\alpha_z:=\theta_{\mathbb{C}^{n+1}}|\partial X_z$ defines a family of contact forms on the fibrewise boundaries. To be explicit, the vector field $Z$ defined fibrewise by $i_Z\Omega|TX_z=\theta_{\mathbb{C}^{n+1}}|T(X_z\cap \mathcal{U})$ in a neighbourhood $\mathcal{U}\cap X_z$ of 
$\partial X_z$ satisfies the required $L_Z\Omega_\gamma=\Omega_\gamma$, for any path $\gamma:[0,1]\rightarrow D_\delta^*$, which implies that $\alpha_z$ defines a contact form for any $z \in D^*_\delta$ (see \cite{duff}). That the orthogonality requirement of \ref{compatiblecplx} is satisfied follows from the fact that the gradient $Z_r:={\rm grad}(r)$ of the radius function $r$ on $B_r^{2n+2}, \ 0<r<1$ satisfies by direct calculation $i_{\frac{1}{2}{Z_r}}\Omega_{\mathbb{C}^{n+1}}=\theta_{\mathbb{C}^{n+1}}$, so $Z_z, \ z\in D_\delta$ equals the orthogonal projection of $Z_r/2$ to $X_z$, hence $Z$ is orthogonal to the levels sets of the function $\tilde S_z:=r|X_z$ which is defined on a neighbourhood of $\partial X_z$ for any $z \in D^*_\delta$, from which it follows that, over $\mathcal{U}$, $Z$ is orthogonal to the level sets of $\hat S$.
\end{proof}

\subsection{Relative $(n,0)$-forms and winding numbers}\label{relativen0}
Let now $n\geq 2$ and $\hat s$ be the global section of $\mathcal{H}^n(f_*\Omega^\cdot_{\hat X/D^*_{\delta}})$ defined by the polynomial $1 \in M(f)$ ($M(f)$ denoting the Milnor algebra), that is, having the property that 
\begin{equation}\label{features}
dz_0\wedge\dots\wedge dz_n=d f\wedge \hat s \in \Omega^{n+1}(\hat X,\mathbb{C}).
\end{equation}
Now note that for any $x \in D^*_\delta$, $s_x:=\Psi_x^*(\hat s|\hat X_x)$ defines a closed $n$-form on $X_x$, which is holomorphic on the set 
\begin{equation}\label{tildea}
\tilde A_x(m):=\{z\in X_x|\psi_{m}(|z|^2)=1\}\subset A_x(m)\subset X_x,
\end{equation}
we will denote the resulting global section of ${\bf H}^n(Z,\mathbb{C})$ by $s$. Here, ${\bf H}^n(Z,\mathbb{C})$ denotes the vector bundle whose fibre over $u \in S^1$ is the $n$-th cohomology of the complex $\Omega^*(X_u,\mathbb{C})$ (writing $Z\hookrightarrow X\rightarrow S^1_\delta$ and $Z_u=X_u$). In the following it will be convenient to assume that the family $\Psi_x, x \in D^*_\delta$ from Lemma \ref{th:classical-milnor-fibre} is chosen to be equivariant with respect to the quasihomogeneous circle action which defines a fibre bundle homomorphism on $\hat X$ as well as on $X$ by setting on the latter $\sigma_t(z,x)=(\sigma_t(z),e^{2\pi it}x), \ t \in [0,1]$, so $\Psi_x=\sigma_t^{-1}\Psi_{e^{2\pi it}x}\sigma_t, t \in [0,1]$. Choose any $\epsilon >0$ so that $\epsilon<\delta$ and fix $x \in S^1_\epsilon$ and denote by $M=X_x$ the fibre at $x$, following the discussion above $(M; \omega_x=:\omega)$ is a symplectic manifold with contact boundary. Let $[s_x]=[s|M]\in H^n(M, \mathbb{C})$ be the class in $M$ induced by $s$. Then we will assume that the following holds for any small enough $\epsilon >0$:
\begin{ass}\label{ass2}
There is an oriented closed Lagrangian submanifold $Q_x\subset M$, that is ${\rm dim} Q_x=n$ and $\omega|Q_x=0$, so that $Q_x \cap \partial M=\emptyset$, $Q_x$ represents a non-zero class $[Q_x] \in H_n(M,\mathbb{C})$ and
\begin{enumerate}
\item $H^1(Q_x,\mathbb{C})=H_1(Q_x, \mathbb{Z})=0$,
\item $\int_M[s_x]\wedge PD[Q_x]=c \neq 0$,
\end{enumerate}
\end{ass}
In fact, if an obvious topological condition on $[s_x]$ is satisfied, the two conditions in Assumption \ref{ass2} are provable by means of generically perturbing $f$ and subsequent embedding of a neighbourhood of the zero section of $T^*S^n$ into $M$ (see Appendix A). To be precise we have even further:
\begin{lemma}\label{conditionsproof}
Assume $[s_x]$ lies in the image of $i^*:H^n(M, \partial M, \mathbb{C})\rightarrow H^n(M,\mathbb{C})$. Then there exists an oriented, closed Lagrangian submanifold $Q_x\subset M$ so that 1. and 2. of Assumption \ref{ass2} are satisfied. Further, $Q_x\subset M$ is diffeomorphic to $S^n$ and there exists $s\in\mathbb{R}, x \in D_\delta^*$ so that $Q_x\subset \tilde A_x(s)\subset M:=X_{(x,s)}$ (see (\ref{milnorfibr})). If $[\rho^k]=[Id_M]$, then the corresponding isotopy $\rho^k_t \in {\rm Symp}(M,\partial M,\omega)$ can be chosen so that $\rho^k_t(\tilde A_x(s))\subset \tilde A_x(s)$ for all $t \in [0,1]$.
\end{lemma}
\begin{proof}
Consider the exact sequence
\begin{equation}\label{seq3}
0 \rightarrow H^{n-1}(\partial F,\mathbb{C}) \xrightarrow {\delta}
H^n(M,\partial M,\mathbb{C})
\xrightarrow {b} H^{n}(M,\partial M, \mathbb{C})^* \xrightarrow {r} H^{n}
(\partial M,\mathbb{C})\rightarrow 0.
\end{equation}
where we identified $b$ with the map $i^*:H^n(M,\partial M,\mathbb{C}) \rightarrow H^{n}(M,\mathbb{C})$ induced by the inclusion $i:(M,\emptyset)\rightarrow (M, \partial M)$ using the non-degenerate intersection pairing $(\cdot, \cdot):H^n(M,\mathbb{C})\times H^n(M,\partial M,\mathbb{C})$ identifying $H^n(M;\mathbb{C})$ with $H^n(M,\partial M,\mathbb{C})^*$ via $[\alpha]\mapsto (\alpha,\cdot)$. So we see that for any element $\alpha$ in the image of $i^*$, i.e. for $\alpha=[s_x]$, there is an element $\beta\in H^n(M,\partial M,\mathbb{C})$ so that $(\alpha, \beta)=c\neq 0$. We claim that $\beta$ has a nontrivial image $i^*(\beta) \in H^n(M,\mathbb{C})$.
But this is clear from sequence (\ref{seq3}), since in the other case $b(\beta,\gamma)=0$ for all $\gamma \in H^n(M,\partial M,\mathbb{C})$ which is impossible by the $(-1)^n$-symmetry of the intersection pairing $(\cdot,\cdot)$. By Lemma \ref{lagrangianbasis} in Appendix A., there is a basis $\{\delta_j\}$ of embedded Lagrangian $n$-spheres of $H_n(M,\mathbb{Z})$, so there is at least one element $\delta_k$ with $([s_x],PD[\delta_k])=c\neq 0$, then $i_*(\delta_k)\in H_n(M,\partial M,\mathbb{C})$ will be nontrivial since $PD[\delta_k]\in {\rm im}(i^*)$ and so $Q_x:=\delta_k$ satisfies all requirements.\\
To prove the second assertion, asssume that with $\tilde A_x(m)\subset X_x=M$ for some fixed $x \in D_\delta^*$ as defined in (\ref{tildea}) $Q_x\subset M$ is not contained in  $\tilde A_x(m)$. Choose $s'>m, \ s\in \mathbb{R}$ and set as in the proof of Lemma \ref{symplfibres} $X_{x,s}=\{y \in B_1^{2n+2}|f(y) = \psi_s(|y|^2)x\}$. Then consider the set
\[
\hat Y:=\bigcup_{s \in I\subset \mathbb{R}} X_{\phi(s),s}\times \{\phi(s)\}
\]
where $\phi: I=[m,s']\rightarrow D_\delta^*$ is the path $\phi(s)=\tau(s)x$ where $\tau(s)\in (0,1],\tau(m)=1$ for any $s \in [m, s']$ so that $d_y(f(y)-\psi_s(|y|^2)\phi(s))\neq 0, y \in X_{\phi(s),s}$. Consider $\hat Y$ as an embedding $i:\hat Y\hookrightarrow \mathbb{C}^{n+1}\times \mathbb{C}$, equipped with a closed two-form $\Omega=i^*\Omega_{\mathbb{C}^{n+2}}\in \Omega^2(\hat Y,\mathbb{C})$. Then $(\hat Y,\Omega)$ is an exact symplectic fibration in the sense of Definition \ref{symplfibr} and we consider the embedding $j:\mathbb{R}_0^-\times \partial_h \hat Y\rightarrow \hat Y$ induced by the flow of the Liouville vector field $Z$ given by $i_Z\Omega|TX_{\phi(s),s}=\Theta|TX_{\phi(s),s}$, where $\Theta$ is $\Theta=i^*\Theta_{\mathbb{C}^{n+2}}$. Then, 
since $Q_z\cap\partial M=\emptyset$, we can find a $c<0$ s.t. $\hat Y_c:=j([c,0]\times \partial_h \hat Y)\cap Q_x=\emptyset$. Further, we can assume by (the proof of) Lemma \ref{milnorsymplectic}, applied to $\hat Y$, the radial coordinate of $\hat Y_c$ to be parametrized so that $j(\{c'\}\times \partial \hat Y)=\{x\in \hat Y:\tilde S(x)-1=c'\}, \ c'\in[c,0]$, where $\tilde S(x)=r(x)^2|\hat Y$. Then it is clear that in the definition of $\hat Y_c$ we can choose $s'>0$ so that $-\frac{1}{s'} > c$, then $X_{\phi(s'),s'}\setminus \overline {(\tilde A_{\phi(s')}(s'))}\subset \hat Y_c\cap X_{\phi(s'),s'}$. Finally, by (3.) of Ass. \ref{compatiblecplx} and 
(\ref{Striv}), we see that $Q_{\phi(s')}:=\mathcal{P}^\Omega_{x,\phi(s')}(Q_x)\subset X_{\phi(s'),s'}\cap (\hat Y \setminus \hat Y_c)$. So we see that $Q_{\phi(s')}\subset \tilde A_{\phi(s')}(s')$ and we arrive at the second assertion by setting $M=X_{\phi(s'),s'}$. To prove the last assertion, note that $\rho^k_t, t \in [0,1]$ can be chosen s.t. $\rho^k_t|\hat Y_c\cap X_{\phi(m),m}=Id_{X_x}, \ t \in [0,1]$. But the symplectic monodromy $\rho^k(s')$ of $Y_{s'}:=(f^k)^{-1}(S^1_{|\phi(s')|})$ can be isotoped to the identity by $\rho^k(s')_t:=\tau_{\phi}\circ \rho^k_t\circ \tau_{\phi^{-1}}, \ t\in [0,1]$, where $\tau_{\phi}:X_x=X_{\phi(m),m}\rightarrow X_{\phi(s'),s'}$ is symplectic parallel transport in $\hat Y$, which using (\ref{Striv}) again gives the assertion.
\end{proof}
In the following, by slightly extending the above, we will choose $s(y)\in\mathbb{R}, y \in S^1_\epsilon$ so that $Q_x$, as well as the images of $Q_x$ under symplectic parallel transport along $S^1_\epsilon$ are contained in $\bigcup_{y\in S^1_\epsilon}(\tilde A_y(s)\subset X_y)$. Note also that if we choose $x,s$ as described in the proof of the Lemma, then $\overline {\tilde A_x(s)}\subset M:=X_{(x,s)}$ is a deformation retract of $M$, though we will not use this fact explicitly in the following.\\
The topological condition in Lemma \ref{conditionsproof} can be expressed in terms of the quasihomogeneous weights of $f$:
\begin{lemma}\label{condition}
The condition $[s_x] \in i^*:H^n(M, \partial M, \mathbb{C})\rightarrow H^n(M,\mathbb{C})$ is equivalent to $\sum\beta_i- \beta \neq \mathbb{Z}$.
\end{lemma}
\begin{proof}
Assuming $[s_x] \in {\rm im}\ i^*:H^n(M, \partial M, \mathbb{C})\rightarrow H^n(M,\mathbb{C})$ implies by the exact sequence (\ref{seq3}) above that $[s_x]$ lies in the image of $b:H^n(M, \partial M, \mathbb{C})\rightarrow H^n(M,\partial M, \mathbb{C})^*$. Now since the variation structure of $f$, compare for the notation \cite{nem2} (see also \cite{steenbrink}), is given by
\begin{equation}\label{varstructuref}
\mathcal{V}(f)=\bigoplus_{\alpha \in \Lambda}\mathcal{W}_{exp(2\pi il(\alpha))}((-1)^{[l(\alpha)]+n})
\end{equation}
and since the kernel of $b$ is represented by (Lemma \ref{poly}) monomials $z^{\alpha}$ s.t. $\{\alpha\in \Lambda \subset \mathbb{N}^{n+1}:l(\alpha)\in \mathbb{Z}\}$ the claim follows, since $[s_x]$ is by definition determined by $1\in M(f)$ corresponding to $\alpha=0\in \Lambda$ so $l(0)\notin \mathbb{Z}$.
\end{proof}
Note that concerning the question if the symplectic monodromy $\rho \in  \pi_0({\rm Symp}(M,\partial M,\omega))$ is of finite order in $\pi_0({\rm Symp}(M,\partial M,\omega))$, one can reduce to an examination of the powers $\rho^{\beta \cdot m}$, for $m \in \mathbb{N}, m>1$, since it is well known resp. follows from (\ref{varstructuref}) that $\rho_* \in Aut(H^n(M;\mathbb{C}))$ is of finite order $\beta$ for $f$ quasihomogeneous of weighted order $\beta$ (Steenbrink \cite{steenbrink}), while the condition in Lemma \ref{condition} ensures that $(\rho_*)^k([s_x])\neq [s_x]$ for any $k\neq m\cdot\beta, \ m\in \mathbb{N}^+$. Thus, given Assumptions \ref{ass2} and \ref{ass4} below (which is proven in Section \ref{app4}) we will prove the following
\begin{theorem}\label{theorem34}
Assume $n\geq 2$ and $[s_x]$ lies in the image of $i^*:H^n(M, \partial M, \mathbb{C})\rightarrow H^n(M,\mathbb{C})$. Let $k=\beta m$, $m>0, \ m\in \mathbb{N}$, then $\rho^k \neq Id$ in $\pi_0({\rm Symp}(M,\partial M,\omega))$. This implies that $\rho$ is an element of infinite order in $\pi_0({\rm Symp}(M,\partial M,\omega))$.
\end{theorem}
The proof of the theorem will require a series of lemmata, as already remarked in the introduction we will give a proof of Theorem \ref{theorem34} in Section \ref{boundingdisks} which uses the additional Assumption \ref{ass5} but dispenses from the use of a family of fibrewise $(n,0)$-forms as introduced in Lemma \ref{relativen} below. We begin with
\begin{lemma}\label{mappingcylinder}
Let $\tilde X:=X|D_{[\epsilon,\delta]}$ and consider $M=X_x$ with symplectic form $\omega_x$ and $Z|X_x$ being the Liouville vector field of $M$ for some $x \in S^1_\delta\subset \partial D_{[\epsilon,\delta]}$. Then there is a family $\{\rho_t\}$ for $t \in [\epsilon,\delta]$ so that the $\rho_t \in {\rm Symp}(M,\partial M, \omega)$ define the same element in $\pi_0(Symp(M,\partial M, \omega))$ for different $t, t' \in [\epsilon, \delta]$ and so that there is an equivalence of symplectic fibrations
\begin{equation}\label{symplectictriv}
\begin{split}
\Theta:\tilde X &\simeq \hat X:=\left([\epsilon,\delta]\times [0,1]\times M\right)/\left((t,0,x)\sim (t,1,\rho_{t}(x)\right).
\end{split}
\end{equation}
Furthermore, under this identification, the symplectic form $\Omega$ with horizontal space $H_\Omega$ on $\tilde X$ corresponds to a symplectic form $\Omega_0$ on $[\epsilon,\delta]\times [0,1]\times M$, so that the induced horizontal distribution $H_0$ restricted to the (tangent space of the) hypersurfaces $Y_t:=\left(\{t\}\times [0,1]\times M\right)/\left((t,0,x)\sim (t,1,\rho_{t}(x)\right)$ is the one induced from the trivial horizontal distribution $H_0$ on $\pi_0:\{t\}\times [0,1]\times M \rightarrow \{t\}\times [0,1]$.
\end{lemma}
\begin{proof}
Let for any $t \in [0,\delta-\epsilon]$ $\Theta_t:Y_\delta \rightarrow Y_{\delta-t}$ be the parallel transport with respect to $\Omega$ along radial rays, then $Y_{\delta_t}$ identifies with $M \times [0,1]/(x,0)\sim (\rho_{t_0}(x),1)$ setting $\rho_t:= \Theta^{-1}\circ \rho_\delta \circ \Theta$, where $\rho_\delta$ is the symplectic monodromy of $Y_\delta$, this also gives a differentiable structure on the union $\bigcup_{t \in[\epsilon,\delta]}Y_t$ and the assertion about $H_0$.
\end{proof}
Consider now for $k=m·\beta, m \in \mathbb{N}$ the following base extension of $\tilde X$:
\begin{equation}\label{betacov2}
\begin{CD}
\tilde X^k @>>\pi_k> \tilde X \\
 @VV f^k V    @VV f V \\
D_{[\epsilon,\delta]}  @>>\lambda_k > D_{[\epsilon,\delta]},
\end{CD}
\end{equation}
where $\lambda_k(z)=z^k$. Then $\pi_k:\tilde X^k \rightarrow \tilde X$ is the $k$-fold connected cyclic covering of $\tilde X$, that is $\mathbb{Z}_k$ acts transitively on the set $\pi_k^{-1}(x)$ for any $x \in  \tilde X$. The isomorphism from Lemma \ref{mappingcylinder} lifts to an isomorphism
\begin {equation}\label{mapcyk}
\Theta^k:\tilde X^k\simeq \hat X^k:=\left([\epsilon,\delta]\times [0,1]\times M\right)/\left((t,0,x)\sim (t,1,\rho^k_{t}(x)\right).
\end{equation}
Since $\pi_k:\tilde X^k \rightarrow \tilde X$ is a covering map, we can lift the symplectic form $\Omega$ on $\tilde X$ to a corresponding symplectic form $\Omega^k$ on $\tilde X^k$, analogously we can lift the complex structure $J$ on $\tilde X$ to a compatible complex structure $J^k$ on $\tilde X^k$ (compatible with $\Omega$) and inducing a family of vertical complex structures $J^{k,v}_x, \ x \in D_{[\epsilon,\delta]}$ on the fibres of $\tilde X^k$ (compatible with the fibrewise lifted symplectic forms $\omega^k_x$), we denote the connection associated to $\Omega^k$ by $H_{\Omega_k}$. From the description of $\tilde X^k$ in (\ref{mapcyk}) one sees that the monodromy induced by $H_{\Omega_k}$ around circles $S_t$, $t \in [\epsilon,\delta]$, equals $\rho^k_t$.\\
Denoting $Z^k\hookrightarrow \tilde X^k$ the union of the fibres of $\tilde X^k$ (note that $Z^k_u\simeq Z_u, \ u \in D_{[\epsilon,\delta]}$) let $s^k \in \Gamma({\bf H}^n(Z^k,\mathbb{C}))$ be the section determined by the lift of $s$ over $\tilde X$ to $\tilde X^k$ (lifting the local expressions of $s$ as differential forms). For the following, we need a certain extra-structure on $f^k:\tilde X^k\rightarrow D_{[\epsilon,\delta]}$, namely, the existence of a non-vanishing fibrewise $(n,0)$-form restricting on each $\tilde A_x(m)\subset \tilde X^k_x$ to $s^k$. To be more precise, recall that if $\pi_k:\tilde X^k \rightarrow \tilde X$ denotes the $k=m\cdot\beta, \ m\in \mathbb{N}$-fold covering of $f:\tilde X\rightarrow D_{[\epsilon,\delta]}$, then we have by (\ref{features}) and (\ref{tildea}) 
\begin{equation}\label{features2}
(s^k_0:=i_{X_{f^k}}(\pi^k)^*(dz_0\wedge\dots \wedge dz_n))|\tilde A_x(m)= s^k|\tilde A_x(m),\quad x \in D_{[s,\epsilon]},
\end{equation}
where the equality is understood on vertical tangent vectors and $X_{f^k}$ is defined as the projection to $T^{(1,0)}\tilde X^k\subset T_\mathbb{C}\tilde X^k$ of the vector field being horizontal w.r.t. $\Omega^k$ and satisfying $d(f\circ \pi_k)(X_{f^k})=1$, so $s^k|\tilde A_x(m)\in \Gamma(\Lambda^{n,0}T^*\tilde X^k_x), \ x \in D_{[\epsilon,\delta]}$ (recall $\tilde A_x(m)$ as defined in (\ref{tildea})). Note that by definition of $s$ (see \ref{features} and below), we have $s^k|\tilde X^k_x=\pi_k^*\Psi_x^*(s|X_x)$, where $\Psi_x:X_x\rightarrow \hat X_x$ is the diffeomorphism introduced in Lemma \ref{th:classical-milnor-fibre}, so $s^k \in \Gamma({\bf H}^n(Z^k,\mathbb{C}))$ but in general, $s^k|(\tilde X^k_x \setminus \tilde A_x(m))\notin \Gamma(\Lambda^{n,0}T^*(\tilde X^k_x\setminus \tilde A_x(m)))$. Nervertheless, we have the following Lemma, which is a family version of Lemma 4.12 in \cite{seidel}. For this note that setting for any $x \in S^1_\delta\subset D_{[\epsilon, \delta]}$
\[
B_x(m):=A_x(m)\setminus\tilde A_x(m)= \{z\in X_x|\psi_{m}(|z|^2)=0\}\subset A_x(m)\subset X_x,
\]
there is by Lemma \ref{milnorsymplectic} resp. by (\ref{milnorfibr}) a neighbourhood $N$ of $\partial Y$ where $Y=X\cap f^{-1}(S^1_\delta)$ in $Y$ and a diffeomorphism $\Theta:N \rightarrow S^1_\delta\times B_x(m)$ for some fixed $x \in S^1_\delta$ which is a fibrewise symplectomorphism preserving $\Omega$ and the fibrewise complex structures $J_x,\ x \in S^1_\delta$. Note that $\Theta$ gives rise to a corresponding trivialization $\Theta_k:N^k\rightarrow S^1_\delta\times B_x(m)$ for some neighbourhood $N^k\subset Y^k:=\tilde X^k\cap (f^k)^{-1}(S^1_\delta)$ of $\partial Y^k$ so that $\lambda_k\circ pr_1\circ \Theta_k=pr_1\circ \Theta\circ \pi_k$ (using notation as in \ref{betacov2}).
\begin{lemma}\label{relativen}
There is a family of fibrewise compatible almost complex structures $\tilde J_x, \ x \in S^1_\delta$ on $Y^k$ and a non-vanishing section $\mathbf{s}^k \in \Gamma(\Lambda^{n,0}(T^*Y^k)^v)$ (with respect to $\tilde J_x$) so that 
\[
\mathbf{s}^k|\tilde A_x(m)=s^k|\tilde A_x(m), \ J_x|A_x(m)=\tilde J_x|A_x(m) \ x \in S^1_\delta,
\]
and $(\mathbf{s}, \tilde J)$ is compatible with $\Theta_k$ on $N^k$ in the sense that $(\mathbf{s}^k, \tilde J) |\bigcup_{y\in S^1_\delta}B_y(m)=({\rm pr}_2\circ\Theta_k)^*((\mathbf{s}^k_x, \tilde J_x)|B_x(m))$ for a fixed $x \in S^1_\delta$.
\end{lemma}
{\it Remark.} Actually $\Theta$ resp. $\Theta_k$ are by definition of the Milnor-fibration in (\ref{milnorfibr}) simply the identity on $S^1_\delta\times B_x(m)$, so $N$ resp. $N^k$ are trivialized by the identity map by definition. Furthermore, it actually true that $\mathbf{s}$ can be chosen so that ${\mathbf{s}}^k|A_x(m) = s_0^k|A_x(m), \ x \in S^1_\delta$, where the equality is understood between elements of $\Gamma(\Lambda^{n,0}T^*A_x(m))$, since we will not need this stronger statement in the following, we restrict to the above statement.
\begin{proof}
For the proof, denote by $p:\mathbb{C}^{n+1}\rightarrow \mathbb{C}$ the actual quasihomogeneous polynomial as referred to in (\ref{milnor1}), whereas $f:X\rightarrow D_\delta^*$ denotes the projection in (\ref{milnorfibr}).
Then for any $x \in S^1_\delta$, define a complex vectorbundle $K(x)\rightarrow X_x$ of dimension $2n$ by $K(x)_z={\rm ker}\ dp_z$. The set $K:=\bigcup_{x \in S^1_\delta} K(x)\rightarrow Y$ defines a complex subbundle of $\mathbb{C}^{n+1}\times Y$ of the same dimension over $Y$, let $K^k=\pi_k^*(K)$ be its lift to $Y^k$. $K^k$, as well as $T^vY^k$, (the complexification of) the vertical tangent bundle to $Y^k$, are complex subbundles of $\mathbb{C}^{n+1}\times Y^k$, which by the arguments in Lemma \ref{th:milnor-fibre}, can be made arbitrarily close. This means that the projection $P_K:T^vY^k\rightarrow K^k$ is a bundle isomorphism and that the pullback $J'_x=P_K^*(J^0_x)$ of the canonical complex structure $J^0$ on $\mathbb{C}^{n+1}$, restricted to $K^k(x)$, defines an amooth family of almost complex structures on $Y^k_x, \ x\in S^1_\delta$ which is $\omega_x$-tame and restricts to $J_x$ on $A_x(m)$ for any $x \in S^1_\delta$. Furthermore, $\tilde {\mathbf{s}}^k:=P_K^*(s_0^k)$ restricts on each fibre $Y^k_x$ to a nonvanishing $(n,0)$-form w.r.t. $J'_x$. Now, fix one fibre $Y^k_x$ and observe that $J'_x$ can be homotoped to an almost complex structure $\tilde J_x$ on $Y_x^k$ which is compatible with $\omega_x$ and that this homotopy can be chosen to be constant on $A_x(m)$. Denote by $\phi: TY_x^k \rightarrow TY^k_x$ the corresponding bundle-map, that is, $\phi^*J'_x=J_x$. Observe that the weighted circle action $\sigma(t),\ t \in S^1$ associated to $p$, lifted to $Y^k$, is a fibrewise diffeomorphism preserving the unitary structures $\omega_x,\ J_x$ and that $\sigma^m(t)$ covers $t \mapsto t.\pi_k(x)$ in $Y^k$ if $k=m\beta, \ m\in \mathbb{N}$. Thus define $\mathbf{s}^k_x=\phi^*\tilde {\mathbf{s}}^k_x$ and $\mathbf{s}^k_y=(\sigma^{m}(-t))^* \circ \phi^*\circ (\sigma^m(t))^*\tilde {\mathbf{s}}^k_y$ and finally $J_y=(\sigma^{m}(-t))^* \circ \phi^*\circ (\sigma^m(t))^*J'_y$ if 
$\pi_k(y)=t.\pi_k(x)$. Since $\sigma(t)(A_x(m))=A_y(m)$, the family $\mathbf{s}^k_x, \ x\in S^1_\delta$ satisfies all requirements of the Lemma.
\end{proof}
Now recall the construction of the covering $\pi_k:\tilde X^k\rightarrow \tilde X$, where $f:\tilde X\rightarrow D_{[\epsilon,\delta]}$ is given by (\ref{milnorfibr}), as constructed in (\ref{betacov2}) and the isomorphism 
\begin{equation}\label{trivkbla}
\overline \Theta:\tilde X^k\simeq \frac{M\times [\epsilon,\delta]\times [0,1]}{(x,t,0)\sim (\rho^k(t,x),t,1)}\rightarrow D_{[\epsilon,\delta]},
\end{equation}
where the $\rho^k(t,\cdot), t \in [\epsilon,\delta]$ are conjugated in ${\rm Symp}(M,\partial M,\omega)$ (see Lemma \ref{mappingcylinder}) and $M=\tilde X_x$ for some fixed $x\in D_{[\epsilon,\delta]}$, denote $p^k$ the corresponding quotient map $p^k:M\times [\epsilon,\delta]\times [0,1]\rightarrow \tilde X^k$. 

\begin{ass} Assume from now on that the symplectic monodromy of the 'reference bundle' $f^k:X^k\cap (f^k)^{-1}(S^1_\delta)\rightarrow S^1_\epsilon$, namely $\rho^k=\rho^k_\epsilon \in \pi_0({\rm Symp}(M,\partial M,\omega))$, is trivial and $\sum\beta_i- \beta \neq \mathbb{Z}$.
\end{ass}

We will give a proof of Theorem \ref{theorem34} by leading this assumption to contradiction. Setting $Y^k=(f^k)^{-1}(S^1_\delta)$ as above set $\rho^k(\delta,\cdot)=\rho^k$ as the symplectic monodromy of $Y^k$ and choose an isotopy $\rho^k_{(\cdot)}:[0,1]\times M\rightarrow M$ connecting $\rho^k_1=\rho^k$ to the identity $\rho^k_0 =id$ in ${\rm Symp}(M,\partial M,\omega)$. We now construct a smooth $n+1$-dimensional submanifold $Q\subset Y^k$ by choosing a $0 <r<<1$ and a smooth function $\psi:[1-r,1]\rightarrow [0,1]$ that is zero in some neighbourhood of $1-r$ and equal to $1$ in a neighbourhood of $1$ and defining a subset of $M\times \{\delta\}\times [0,1]$ as \begin{equation}\label{hatq}
\hat Q= Q_x\times \{\delta\} \times [0,1-r]\cup \bigcup_{\tau \in [1-r,1]} (\rho^k_{\psi(\tau) })(Q_x)\times \{\delta\}\times \{\tau\}
\end{equation}
where here, $Q_x\subset \tilde A_x(m)\subset M$ is the Lagrangian cycle satisfying the first two conditions in (\ref{ass2}), i.e. $\int_M[s_x]\wedge PD[Q_x]=c \neq 0$. It is the clear that $\hat Q\subset M\times \{\delta\} \times [0,1]$ factorizes to a well-defined $n+1$-dimensional, closed submanifold $Q\subset Y^k\subset \tilde X^k$ whose intersection $Q^0$ with the image of $\{\delta\}\times M\times [0,1-r]$ is Lagrangian in $\tilde X^k$ (see Lemma \ref{bla456} below). With the notation used in Lemma \ref{bla456} we have, since the Lagrangians $\{\rho^k_{\psi(\tau)}(Q_x)\}_{\tau\in [1-r,1]}$ are mutually isotopic in $M$ have is proven in Lemma \ref{bla456} below (we will explain the modified definition of $\kappa_Q, \theta$ constituting $\alpha$ in a moment)
\begin{equation}\label{windingbla}
[s^k|_{y(\tau)\in S^1_\delta}]=\frac{1}{c}\int_{Q_{y(\tau)}}e^{i\theta}i_{X_{f^k}}\kappa_{Q}\cdot [s^k_x]_{||}(y(\tau))=:\alpha(\tau)\cdot [s^k_x]_{||}(y(\tau))
\end{equation}
for $y(\tau)=xe^{2\pi i\tau}$ and $\tau \in [0,1]$ and $Q_{y(\tau)}:=Q\cap Y^k_{y(\tau)}$. To explain the terms occuring in $\alpha$, we define $\kappa_{Q}\in \Gamma(\Lambda^{n+1,0}(T^*\tilde X^k)|Q)$ so that it coincides over the quotient image $z'$ of each $z \in Q_x\times \{\delta\}\times [0,1]\subset \hat Q$ in $Q \subset X_e^k$ with the element of $\Lambda^{n+1,0}(T^*_{z'}\tilde X^k)|Q$ induced by the Lagrangian subspace 
\begin{equation}\label{lagsubspaces}
T^h_{z'}Y^k \oplus T_{z'}Q_{y(\tau)}\subset T_{z'}\tilde X^k, \ y(\tau)=f^k_e(z'),\ \tau \in [0,1].
\end{equation}
Note that $T^h_{z'}Y^k$ denotes the $\Omega^k$-orthogonal complement of ${\rm ker}(df^k)$ in $TY^k$. Clearly, over the image $z'$ of $z \in Q_x\times \{\delta\}\times [0,1-r]\subset \hat Q$ in $Q$, this is simply the element of  $\Lambda^{n+1,0}(T^*_{z'}\tilde X^k)$ induced by the Lagrangian subspace $T_{z'}Q\subset T\tilde X^k$.
On the other hand, the phase $e^{i\theta}:Q \rightarrow S^1$ is defined by the requirement
\begin{equation}\label{bla987}
e^{i\theta}\kappa_{Q}=(\pi^k)^*(dz_0\wedge\dots \wedge dz_n)|Q.
\end{equation}
Then we have the following:
\begin{lemma}\label{bla456}
Let $Q\subset Y^k\subset \tilde X^k$ be constructed as above, then its intersection $Q^0$ with the image of the canonical projection of $\{\delta\}\times M\times [0,1-r]$ in $Y^k$ is Lagrangian, that is 
\[
\Omega^k|Q^0=0.
\]
Let $dz_0\wedge \dots\wedge dz_n$ be the canonical $(n+1,0)$-form on $\mathbb{C}^{n+1}$, restricted to $\tilde X$ and consider its pullback to $\tilde X^k$ by $\pi_k$. Let $\{e_i\}, i=1,\dots,n$ be an oriented orthonormal basis of $T^h_{z'}Y^k \oplus T_{z'}Q_{y(\tau)}\subset T_{z'}\tilde X^k$, let for each $i$, $u_i=1/2(e_i -iJe_i)$ and let $\{u_i^*\}$ be the associated dual basis. Then write locally 
\begin{equation}\label{bla347}
\pi_k^*(dz_0\wedge\dots\wedge dz_n)|Q=e^{i\theta} (u_0^*\wedge\dots\wedge u_n^*)=:e^{i\theta}\kappa_Q,
\end{equation}
for some (well-defined) function $e^{i\theta}:Q\rightarrow S^1$ (note that since $Q$ is oriented, $\kappa_Q$ is a well-defined $(n+1,0)$-form on $T\tilde X^k|Q$. Then since $H^1(Q_y,\mathbb{C})=0, \ y \in S^1_\epsilon$, $\theta$ lifts to a well-defined function $\theta_y:Q_y\rightarrow \mathbb{R}$, while on $Q$ one has a smooth function $\theta:Q\rightarrow \mathbb{R}/\mathbb{Z}$ satisfying (\ref{bla347}). Let $X_{f^k} \in \Gamma(T^{(1,0)}\tilde X^k)$ s.t. $df^k(X_{f^k})=1$, then one has for any $y=e^{2\pi it} \in S^1_\epsilon$
\begin{equation}\label{winding-s}
[s^k_{y(t)}]:= [s^k|_{y(t)\in S^1_\epsilon}]=\frac{1}{c}\int_{Q_{y(t)}}e^{i\theta}i_{X_{f^k}}\kappa_Q\cdot [s^k_x]_{||}(y(t))=:\frac{1}{c}\alpha(t) [s^k_x]_{||}(y(t))=e^{2\pi i \gamma t}\cdot[s^k_x]_{||}(y(t)).
\end{equation}
where $\alpha:[0,1]/\{0,1\}\rightarrow \mathbb{C}^*$, $[s^k_x]_{||}\in \Gamma({\bf H}^n(Z^k,\mathbb{C}))$ is the parallel section which coincides at $x \in S^1_\epsilon$ with $s^k|_x$, $c \neq 0$ is determined by 2. in Assumption \ref{ass2} and 
\begin{equation}\label{maslovindex36}
{\rm wind}(\alpha)=\gamma =m(\sum_i \beta_i-\beta)\in \mathbb{Z}.
\end{equation}
\end{lemma}
\begin{proof}
That $Q$ is Lagrangian is immediate from the fact $H_{\Omega^k}$ is defined as the annihilator of the vertical bundle and the fact that, by construction, $H_{\Omega^k}\cap TY_\epsilon \subset TQ$, that $Q$ is well-defined as a closed Lagrangian submanifold of $\tilde X^k$ is implied by (\ref{horizontality}) in Assumption \ref{ass2}. To prove equation \ref{winding-s} note that if $\Phi_{X_K}(t), t \in[0,1]$ denotes parallel transport in $\tilde X^k$ along $t \mapsto \epsilon e^{2\pi it}$ using the horizontal distribution defined by the Euler vector field $\pi_k^*(K)=\pi_k^*(2\pi i\sum_iw_iz_i\frac{\partial}{\partial z_i})$ on $\tilde X^k$ then by Lemma 3.1.11 in \cite{klein1}
\begin{equation}\label{eulermulti}
\Phi_{X_K}(t)^*s^k=e^{2\pi i\gamma t}s^k,
\end{equation}
where $\gamma =m(\sum_i \beta_i-\beta) \in \mathbb{Z}$ and $k=m\beta$ as above, so $\gamma \in \mathbb{Z}\setminus \{0\}$ by Assumption \ref{condition}. Note that this equation continues to hold on cohomology classes when replacing $\Phi_{X_K}$ by symplectic parallel transport along $t \mapsto e^{2\pi i t}$. But then for $y(t)=\epsilon e^{2\pi it}x$, the section $t \mapsto e^{-2\pi i \gamma t}[s^k_{y(t)}]=: [s^k_x]_{||}(y(t))$ is parallel. So one gets by the invariance of $b$ under parallel transport
\[
\begin{split}
c \cdot e^{2\pi i\gamma t}&=e^{2\pi i\gamma t}\int_{(Y_\epsilon)_{y(t)}}[s^k_x]_{||}(y(t))\wedge {\rm PD}[Q_{y(t)}]=\int _{(Y_\epsilon)_{y(t)}} s^k_{y(t)}\wedge {\rm PD}[Q_{y(t)}]\\
&=\int_{Q_{y(t)}}s^k(y(t))=\int_{Q_{y(t)}}i_{X_{f^k}}(\pi^k)^*(dz_0\wedge\dots\wedge dz_n)\\
&=\int_{Q_{y(t)}}e^{i\theta}i_{X_{f^k}}\kappa_Q.
\end{split}
\]
Finally inserting in $[s^k_{y(t)}]=e^{2\pi i\gamma t}[s^k_x]_{||}(y(t))$ the last equality one arrives at the assertion.
\end{proof}
{\it Remark.} Since the constant $c\neq 0$ will not be of importance in the following, we will set its value to $c=1$ in all subsequent calculations.\\
For later use, the following will prove useful. Note that the family of bundles $\pi^k_\epsilon:Y^k_\tau\rightarrow S^1, \ \tau \in [s,\epsilon]$ is defined in the proof of Lemma \ref{extension} in Section \ref{boundingdisks}, in this section we will be confined to the case $Y^k_\epsilon=Y^k$ only.
\begin{lemma}\label{metricsubspace}
The horizontal subspace $H^k_f\subset TY_\epsilon^k$ given by the lift of the Euler vector field $K= 2 \pi i\sum_i w_i z_i \frac{\partial}{\partial z_i}, \ x\in Y$ on $Y_\epsilon$ to $Y^k_\epsilon\subset X^k_e$ is mapped by $\Theta^k|Y_\epsilon^k$ (see (\ref{mapcyk})) to the subspace which is induced by 
\begin{equation}\label{circle}
H^k_f(x,t)={\rm span}\ \frac{d}{dt} (\Phi_H(t)(x), t) \subset T(M\times [0,1]), \ x \in M,
\end{equation}
on $(M\times [0,1])/(x,0) \sim (\rho^k(x), 1) \simeq Y_\epsilon^k$. Here, $\Phi_H(\cdot):[0,1] \times M \rightarrow M$ is the Hamiltonian flow associated to the Hamiltonian function $H(k)\in C^\infty([s,\epsilon]\times M\times [0,1],\mathbb{R})$
\begin{equation}\label{circleham}
(\Theta^k)^*H(k)=\pi k\sum_{i=0}^{n}w_i|z_i|^2 \in C^{\infty}(\mathbb{C}^{n+1}),
\end{equation}
considered as a family of fibrewise Hamiltonians on $M\times[0,1]$ and restricted to $\{\epsilon\} \times [0,1]\times M$. Furthermore, let $\Phi_{H,\tau}(t)$ be parallel transport along $t \mapsto e^{2\pi i t}$ in $Y^k_\tau, \ \tau \in [s,\epsilon]$ defined by the family of horizontal subspaces on $T(M\times [0,1])$ given for each $\tau \in [s,\epsilon]$ by 
\begin{equation}\label{horftau}
H^k_{f,\tau}(x,t)={\rm span}\ \frac{d}{dt} \left(((\rho^k)_{\epsilon-t(\epsilon-\tau)}\circ (\rho^k_\epsilon)^{-1}) \circ\Phi_H(t)(x), t \right) \subset T(M\times I), \ x \in M, t \in [0,1]. 
\end{equation}
Then $\Phi_{H,\tau}$ is a family of fibrewise isometries on $Y^k_\tau, \ \tau \in [s,\epsilon]$ with respect to the family of fibrewise metrics on $Y^k_\tau$ introduced in Lemma \ref{extension}.
\end{lemma}
\begin{proof}
Let $\Phi_{\Omega^k}(t)$ resp. $\Phi_{X_K}(t)$ denote the flow of the horizontal lifts $X_{\Omega^k}$ resp. $X_K=k\cdot K$ w.r.t. $H_{\Omega^k}$ resp. $H^k_f$ of the vector field $X=2\pi i z$ on $S^1$. Define a flow $\eta(t)=\Phi_{X_K}(-t)\circ \Phi_{\Omega^k}(t)$. Then $\eta$ maps each fibre of $Y^k$  symplectically to itself and since $\Phi_{X_K}(t)$ preserves $\Omega$ hence $H_\Omega$, $\Phi_{X_K}(t)$ commutes with $\Phi_\Omega(t)$ for any $t$, $\eta(t)$ is the flow generated by $Z:=X_{\Omega^k}-X_K$. Now $(i_{X_K}\Omega^k)|F_{z_0}=d(\pi_k^*H(k)|{F_{z_0}})$, where $F_{z_0}$ is any fibre of $Y^k$. Since $i_{X_{\Omega^k}}\Omega=0$, we have 
\[
(i_{X_{\Omega^k}-X_K}\Omega)|F_{z_0}=d(-\pi_k^*H(k)|{F_{z}}), \ z \in S^1_\epsilon,
\]
so $\eta(t)$ is the Hamiltonian flow of $-\pi_k^*H(k)$ and since by definition $\Phi_{X_K}(t)=\Phi_{\Omega^k}(t)\circ \eta(-t)$, we arrive at the assertion. That parallel transport along the family of horizontal subspaces introduced in (\ref{horftau}) introduces fibrewise isometries, follows directly from the form of the vertical complex structures on $Y^k_\tau, \ \tau \in [s,\epsilon]$ as introduced in Lemma \ref{extension}.
\end{proof}
Before we can proceed we need a basic result about Maslov classes. For this, let $(M, \omega, J)$ be a symplectic manifold with a compatible almost complex structure of dimension $2n$. Let $N\subset M$ be a compact, connected and oriented submanifold of dimension $k$, let $i:N\rightarrow M$ be the inclusion and let $\pi_L:{\rm Lag}(M,\omega)\rightarrow M$ be the fibre bundle of Lagrangian subspaces of $(TM, \omega)$. In the following we will also consider the fibre bundle of {\it oriented} Lagrangian subspaces $\widetilde \pi_L:\tilde {\rm Lag}(M,\omega)\rightarrow M$ which is a $2$-fold covering of ${\rm Lag}(M,\omega)$, that is, $\pi_1(\widetilde {\rm Lag}(M,\omega))$ is an index $2$ subgroup of $\pi_1({\rm Lag}(M,\omega))$ (cf. \cite{audin}), that is, we have for any $x \in M$ a diagram
\begin{equation}\label{laggrcov}
\begin{CD}
\widetilde {\rm Lag}_x(M,\omega) @>>>  {\rm Lag}_x(M,\omega) \\
 @VV det V    @VV det^2 V \\
S^1  @>> z\mapsto z^2 > S^1,
\end{CD}
\end{equation}
where we identified ${\rm Lag}_x(M,\omega)$ with $U(n)/O(n)$ and $\widetilde {\rm Lag}_x(M,\omega)$ with $U(n)/SO(n)$. The vertical arrows above are fibrations with simply connected fibres $SU(n)/O(n)$ and $SU(n)/SO(n)$, respectively, thus $\pi_1(\widetilde {\rm Lag}_x(M,\omega))=2\mathbb{Z}$ as as subgroup of $\pi_1({\rm Lag}_x(M,\omega))=\mathbb{Z}$. For any $x \in M$, we will denote by $\tilde \gamma$ resp. $\gamma$ the associated generating cohomology classes in $H^1(\tilde {\rm Lag}_x(M,\omega), \mathbb{Z})$ resp. $H^1({\rm Lag}_x(M,\omega), \mathbb{Z})$. \\
Assume now that $M$ carries a non-vanishing (not necessarily closed) section $s$ of its canonical bundle, that is $s \in \Gamma(\Lambda^{(n,0)}T^*M)$, $s(x)\neq 0$ for any $x \in M$. Note that given a submanifold $i:N\hookrightarrow M$ and a section $\Lambda_0: N\rightarrow i^*{\rm Lag}(M,\omega)$ we have for any point $x \in N$ a trivialization $i^*{\rm Lag}(M,\omega)\simeq N\times {\rm Lag}_x(M,\omega)$ (analogously in the oriented case). Thus in this situation, $\gamma$ resp. $\tilde \gamma$ induce elements in $H^1(i^*{\rm Lag}(M,\omega), \mathbb{Z})$ resp. $H^1(i^*\widetilde {\rm Lag}(M,\omega), \mathbb{Z})$ which we will call the (oriented) Maslov class associated to $\Lambda_0$.
\begin{lemma}\label{lag5}
Given a section $\Lambda_0: N\rightarrow i^*\widetilde {\rm Lag}(M,\omega)$ (thus the associated subbundle $\Lambda_0\subset i^*TM$ is orientable) there is a unique non-vanishing element $\kappa_N\in \Gamma(i^*\Lambda^{(n,0)}T^*M)$ s.t. $|\kappa_N|_g=1$ and ${\rm ev}(\kappa_N(x))(\Lambda_0(x))={\rm vol}_{\Lambda_0(x)}$, where ${\rm vol}_{\Lambda_0(x)}$ denotes the volume form on $\Lambda_0(x)\subset T_xM$ induced by the orientation and metric. Furthermore, if $g:N\rightarrow \mathbb{C}^*$ is determined by
\begin{equation}\label{maslovgen}
i^*s(x)=g(x)\cdot\kappa_N(x), \ x \in N, 
\end{equation}
then we can associate to $s$ a section $\Lambda_s: N\rightarrow i^*\widetilde {\rm Lag}(M,\omega)$ so that $\Lambda_s^*\tilde \gamma =[g^*\beta] \in H^1(N,\mathbb{Z})$ if $\tilde gamma \in H^1(i^*\widetilde {\rm Lag}(M,\omega), \mathbb{Z})$ is the (oriented) Maslov class associated to $\Lambda_0$, where $\beta \in H^1(C^*, \mathbb{Z})$ is the generator and one has 
\begin{equation}\label{maslovdual3}
PD[g^*\beta]=[N_s] \in H_{k-1}(N,\mathbb{Z}), \  {\rm where} \ N_s=\{x \in N: {\rm Im}({\rm ev}(i^*s(x)(\Lambda_0(x))=0\}.
\end{equation}
We call $[g^*\beta]$ resp. $[N_s]$ the (oriented) Maslov class resp. the Maslov cycle associated to $s$ and $\Lambda_0$ on $N$.
\end{lemma}
\begin{proof}
Choose any oriented, unitary basis of $i^*TM$ locally over the open set $U\subset N$ of the form $(e_1,\dots, e_n, Je_1,\dots, Je_n)$ so that $(e_1,\dots, e_n)$ span $\Lambda_0(x),\ x \in U$. Then
\[
\eta_U:=\bigwedge_{i=1}^{n}(e_i-iJe_i)^*
\]
defines locally an element of $\Gamma(i^*\Lambda^{(n,0)}T^*M)|U$. Covering $N$ by open sets $U_i\subset N$, it is clear that the local forms define a section $\kappa_N\in \Gamma(i^*\Lambda^{(n,0)}T^*M)$ with the required property.
Applying the above construction to arbitrary elements of $\widetilde {\rm Lag}(M,\omega)$, then if $\Delta(J)^*=(\Lambda^{(n,0)}T^*M\setminus (M \times\{0\}))$ we get a fibration with simply connected fibres
\begin{equation}\label{detsection}
{\rm det }:i^*\widetilde {\rm Lag}(M,\omega)\rightarrow i^*\Delta(J)^*.
\end{equation}
Choose any section $u$ of ${\rm det}$ along the image of $i^*s$ in $i^*\Delta(J)^*$, such a section is determined for instance by noting that $i^*s$ defines a trivialization $j:i^*\Delta(J)^*\simeq S^1\times N$ over $N$ and $\Lambda_0$ defines as above a trivialization $i^*\widetilde {\rm Lag}(M,\omega)\simeq N\times U(n)/SO(n)$. Viewed through the trivializations, ${\rm det}$ becomes the map $A\mapsto e^{i\phi(x)}{\rm Det}(A), A \in U(n)/SO(n), x \in N$ for some function $e^{i\phi(x)}:N\rightarrow S^1$. Then, by the fact that the fibres of ${\rm det}$ are isomorphic to $SU(n)$, we can choose over any point $x \in N$ smoothly an element in $e^{-i\phi(x)}\cdot SU(n)\subset U(n)/SO(n)$ (take $x \mapsto e^{-i\phi(x)}Id_{\mathbb{C}^n}$) which is mapped under ${\rm det}$ to $1 \in S^1$, thus lying in the kernel of $j\circ {\rm det }$. This already gives a section of ${\rm det}$ over $i^*s$ in $i^*\Delta(J)^*$. Then $\Lambda_s:=u\circ (i^*s): N\rightarrow i^*\widetilde {\rm Lag}(M,\omega)$ defines the Maslov cycle
\[
\mathcal{M}=\bigcup_{x \in N}\mathcal{M}_x,\ \mathcal{M}_x=\{\Lambda(x) \in i^*\widetilde {\rm Lag}(M,\omega)_x: \Lambda(x)\cap \Lambda_s(x)\neq \{0\}\}.
\]
Now adopting arguments of Arnol'd (\cite{arnold}) one infers that $\mathcal{M}=PD[\tilde g^*\beta]$, where $[\tilde g^*\beta] \in H^1(i^*\widetilde {\rm Lag}(M,\omega),\mathbb{Z})$ is determined by $\tilde g:i^*\widetilde {\rm Lag}(M,\omega) \rightarrow \mathbb{C}^*$ and ${\rm det}(\Lambda_s(x))=\tilde g(\Lambda(x))\cdot{\rm det}(\Lambda(x))$ for any $\Lambda(x) \in i^*\tilde {\rm Lag}(M,\omega)_x, \ x \in N$. On the other hand if $m={\rm dim}(i^*\widetilde {\rm Lag}(M,\omega)_x)$ this implies
\begin{equation}\label{masslovclasseq}
[\mathcal{M}_x]=[\{\Lambda(x)\in i^*\widetilde {\rm Lag}(M,\omega)_x: {\rm Im}({\rm ev}(i^*s(x))(\Lambda(x)))=0\}] \in H_{m-1}(i^*\widetilde{\rm Lag}(M,\omega)_x,\mathbb{Z})
\end{equation}
and by definition (\ref{maslovgen}) we have $g=\Lambda_0^*\tilde g$. Using $\Lambda_0$ as trivializing $i^*\widetilde {\rm Lag}(M,\omega)$, it then follows that $\Lambda_s^*\tilde \gamma=[g^*\beta] \in H^1(N, \mathbb{Z})$ is the pullback by $\Lambda_s$ of the Maslov class on $i^*\widetilde {\rm Lag}(M,\omega)$ defined by $\Lambda_0$ as above this lemma. Finally by the above we have $[N_s]=[\Lambda_0^{-1}(\Lambda_0(N)\cap \mathcal{M})]$ and by the functoriality of the Poincare dual under the mapping $\Lambda_0^*$ we arrive at the assertion.
\end{proof}
{\it Remark.} Note that the above proof explicitly attaches a section $\Lambda_s: N\rightarrow i^*\widetilde {\rm Lag}(M,\omega)$ to a non-vanishing section $s \in \Gamma(\Lambda^{(n,0)}T^*M)$ and an embedding $i:N\rightarrow M$, which will be of some importance in subsequent constructions. Using the notation of the proof of this Lemma, for any $x \in N$, $\mathcal{M}_x\subset i^*\widetilde {\rm Lag}(M,\omega)_x$ is a canonically cooriented cycle of codimension one, to be more precise (\cite{arnold}), $\mathcal{M}_x$ is a real algebraic subvariety with singular set
\[
\overline {\mathcal{M}^2_x}=\bigcup_{k \geq 2}\mathcal{M}^k_x, \ \mathcal{M}^k_x=\{\Lambda(x) \in i^*\widetilde {\rm Lag}(M,\omega)_x: {\rm dim}(\Lambda(x)\cap \Lambda_s(x))=k\},
\]
where ${\rm codim}(\mathcal{M}^k_x)=\frac{1}{2}k(k+1)$ (cf. Arnold \cite{arnold}). Furthermore, for any $x \in N$, the sets $\{\mathcal{M}^k_x\}_{k \in \mathbb{N}^+}$ furnish $\mathcal{M}_x$ with the structure of a stratified space (see Mather \cite{mather}, Whitney \cite{whitney}) with smooth top-stratum $\mathcal{M}^1_x$ of codimension $1$, singular set $\overline {\mathcal{M}^2_x}$ of at least codimension $3$ in $i^*{\rm Lag}(M,\omega)_x$ and strata $\mathcal{M}^k_x$. Set $\mathcal{M}^k=\bigcup_{x\in N} \mathcal{M}^k_x, \ k\in \mathbb{N}^+$. These remarks suggest the following 
\begin{Def}\label{generic}
In the situation and notation of Lemma \ref{lag5}, let $N_s\subset N$ represent $[N_s] \in H_{k-1}(N,\mathbb{Z})$ as defined as in (\ref{maslovdual3}) so that $PD[g^*\beta]=[N_s]$. 
Let $N^{\mathcal{M}}_s:=\Lambda_0^{-1}(\Lambda_0(N)\cap \mathcal{M})$, so that also $PD[g^*\beta]=[N^{\mathcal{M}}_s]$. We will call $N^{\mathcal{M}}_s\subset N$ {\it generic} if it is a Whitney stratified space with smooth cooriented top-stratum of codimension $1$ in $N$ given by $N^{\mathcal{M},top}_s=\Lambda_0^{-1}(\Lambda_0(N)\cap \mathcal{M}^1)$ and with singular set $N^{\mathcal{M},sg}_s=\Lambda_0^{-1}(\Lambda_0(N)\cap \overline {\mathcal{M}^2})$ of at least codimension $3$ in $N$.
\end{Def}
Let now be $x \in S^1_\epsilon$ fixed and $Q_x \subset (Y_\epsilon^k)_x=:M \subset \tilde X^k$ as in Assumption \ref{ass2}. Let $\Phi_H(t), t \in[0,1]$ be the family of Hamiltonian flows on $M$ introduced in (\ref{circleham}) and $\rho^k_t\in {\rm Symp}(M,\partial M,\omega_x), \  t \in [0,1]$ an isotopy so that $\rho^k_0=Id,\rho^k_1=\rho^k$. Then, relative to the representation of $(Y^k_\epsilon, \Omega^k)$ as a symplectic mapping cylinder (see (\ref{mapcyk3}) below) which is induced by symplectic parallel transport $\Phi^{\Omega^k}_{(\cdot)}:Y^k_\epsilon\rightarrow Y^k_\epsilon$, $\Phi_H(\cdot)$ resp. $\rho^k_{(\cdot)}$ define by considering (\ref{triv46}) and (\ref{hatqt}) below a $1$-parameter-family of immersions $i_\tau: Q_x\times [0,1] \rightarrow Y^k_\epsilon$ whose images $Q_\tau:={\rm im }(i_\tau)$ factorize for $\tau=0,1$ into closed $n+1$-dimensional submanifolds $Q_0,\ Q_1\subset Y^k_\epsilon$ so that $Q=Q_1$ and so that any intersection $Q_\tau\cap (Y^k_\epsilon)_u,  \tau \in [0,1], u \in S^1_\epsilon$ is a Lagrangian sphere (resp. a union of Lagrangian spheres for $u=x$) in $(Y^k_\epsilon)_u$. Thus consider the family of Lagrangian spheres in $(Y^k_\epsilon)_{x(t)}$ defined for $\tau,t \in [0,1]$, $x(t)=xe^{2\pi it}$ and fixed $x\in S^1$ by
\begin{equation}\label{qfamily}
Q_{\tau,x(t)}=i_\tau(Q_x\times \{t\}),\ \tau\in [0,1],\ t \in [0,1],\ {\rm  s.t.}\  Q_{\tau,x(1)}=\Phi_{H(x)}(1-\tau)\circ\rho_{\tau}(Q_x) \subset (Y^k_\epsilon)_x,
\end{equation}
for the latter equality compare (\ref{triv46}). For any such $Q_{\tau,x(t)}$ we have a section $\Lambda_{Q_{\tau,x(t)}} \in \Gamma(i_{\tau,t}^*\widetilde {\rm Lag}(\tilde X^k,\Omega^k))$, where $i_{\tau,t}:Q_{\tau,x(t)}\hookrightarrow X_e^k$ is the inclusion, which is given for $z \in Q_{\tau,x(t)}$ by
\begin{equation}\label{lagsubspaces57}
\Lambda_{Q_{\tau,x(t)}}(z)=T^h_{z}Y^k_\epsilon \oplus T_{z}Q_{\tau,x}\subset T_{z}\tilde X^k,\ \tau\in\{0,1\},
\end{equation}
where $T^hY_\epsilon^k$ denotes the $\Omega^k$-orthogonal complement of $T^vY_\epsilon^k$ in $TY_\epsilon^k$. Then by Lemma \ref{lag5}, $\Lambda_{Q_{\tau,x(t)}}$ induces a non-vanishing section $\kappa_{Q_{\tau,x(t)}}\in \Gamma(i_{\tau,t}^*\Lambda^{(n+1,0)}T^*\tilde X^k)$ of unit length for any $\tau \in [0,1], t \in [0,1]$ and a family of functions $g_{\tau,x(t)}:Q_{\tau,x(t)}\rightarrow S^1$ by setting
\begin{equation}\label{maslovfunction}
g_{\tau,x(t)}\kappa_{Q_{\tau,x(t)}}=((\pi^k)^*dz_0\wedge\dots\wedge dz_n)|i_{\tau,t}^*T^*\tilde X^k=X_{f^k}^*\wedge s^k|i_{\tau,t}^*T^*\tilde X^k
\end{equation}
For $\tau,t \in [0,1]$, let $N^{\mathcal{M}}_{\tau,t}\subset Q_{\tau,x(t)}$ be associated to the triple $(Q_{\tau,x(t)}, \Lambda_{Q_{\tau,x(t)}}, X_{f^k}^*\wedge s^k|i_{\tau,t}^*T^*\tilde X^k)$ by Definition \ref{generic} resp. the proof of Lemma \ref{lag5}, more precisely, we assume the following:
\begin{ass}\label{ass4}
With the above notation and definition, one can choose $x=x(0) \in S^1_\epsilon$ and modify the families $Q_{\tau,x(t)}$ for $t \in [0,\delta]\cup [1-\delta,1], \tau\in [0,1]$ for some small $\delta> 0$ and $\Lambda_{Q_{\tau,x(t)}}$ for $t, \tau\in [0,1]$ by 'arbitrarily small amounts' (in a sense to be made precise in Section \ref{app4}) so that the first and at least one of the conditions (2.) and (3.) are satisfied:
\begin{enumerate}
\item For $\tau \in [0,1]$ the set $\hat N^{\mathcal{M}}_\tau:=\bigcup_{t \in [0,1]}\hat N^{\mathcal{M}}_{\tau,t}\subset Q_x \times [0,1]$ where $\hat N^{\mathcal{M}}_{\tau,t}:= i_\tau^{-1}(N^{\mathcal{M}}_{\tau,t})$ is generic outside of a discrete set. Further, each member of the family $N^{\mathcal{M}}_{\tau,t}\subset Q_{\tau,x(t)},\ t \in \{0,1\}, \tau\in [0,1]$ is {\it generic} outside of a discrete subset and non-empty for $\tau=\{0,1\}$.
\item For $t \in \{0,1\}$, the union of top strata $\hat N^{\mathcal{M}, top}_t:=\bigcup_{\tau \in [0,1]}\hat N^{\mathcal{M}, top}_{\tau,t}\subset Q_x \times [0,1]$, as well as each $\hat N^{\mathcal{M}}_{\tau,t}:= i_\tau^{-1}(N^{\mathcal{M},top}_{\tau,t})$ are canonically cooriented (by Definition \ref{generic}) outside of a discrete set. Furthermore, for $t \in \{0,1\}$, there exists an oriented path $\hat c: [0,1]\rightarrow Q_x \times [0,1]$ so that $\hat c(\tau) \in Q_x \times \{\tau\}, \hat c(0)=\hat c(1)$ so that $\hat c$ intersects $\hat N^{\mathcal{M}, top}_t$ transversally in generic points and such that $\hat c\cdot\hat N^{\mathcal{M}, top}_t=0$.
\item For $t \in \{0,1\}$, there is a {\it path-connected} subset $\mathcal{U}_t\subset \bigcup_{\tau \in [0,1]} Q_{\tau,x(t)}\times \{\tau\}\subset (Y^k_\epsilon)_{x(0)}\times [0,1]$ and a family of connected embedded, non-empty $n$-manifolds $\mathcal{U}_{\tau,t}\subset Q_{\tau,x(t)}, \tau \in [0,1]$ so that $\mathcal{U}_t=\bigcup_{\tau \in [0,1]}\mathcal{U}_{\tau,t}\times\{\tau\}$ and one has $\mathcal{U}_{\tau,t}\subset Q_{\tau,x(t)}\setminus  \overline {N_{\tau,t}^{\mathcal{M},top}}$. Further each $\mathcal{U}_{\tau,t}$ is open in $Q_{\tau,x(t)}$  and for any $t \in \{0,1\}, \tau \in [0,1]$ equals a connected component of $Q_{\tau,x(t)}\setminus \overline {N_{\tau,t}^{\mathcal{M},top}}$.
\end{enumerate}
\end{ass}
{\it Remark.} Note that to prove Proposition \ref{keylemma} below, thus Theorem \ref{symplecticmonodromy}, the first and the second assertion are actually sufficient and proven in the present work. The third assertion is proven under a further assumption on the vanishing of 'higher singularities'(cf. Section \ref{app4}, Proposition \ref{keysing}) but the alternative line of reasoning based on it in the proof of Proposition \ref{keylemma} is expected to play a key role in a proof of Conjecture \ref{conjspec}. We will discuss aspects of this in the end of Section \ref{generalspec} for the quasihomogeneous case.\\
By definition resp. Lemma \ref{lag5}, for $\tau\in \{0,1\}$, $N^{\mathcal{M}}_\tau:=i_\tau(\hat N^{\mathcal{M}}_\tau)\in H_n(Q_\tau,\mathbb{Z})$ and $N_\tau:=\bigcup_{t \in [0,1]} N_{\tau,t}\subset Q_\tau$ where $N_{\tau,t}=\{z \in Q_{\tau,x(t)}: g_{\tau,x(t)}(z)\in \mathbb{R}\}$ both represent the Poincare dual of the Maslov class $[g_\tau^*\beta]\in H^1(Q_\tau,\mathbb{Z})$, where $g_\tau:Q_\tau\rightarrow \mathbb{C}^*, \tau \in \{0,1\}$ assembles the family (\ref{maslovfunction}). Note that for $t \in \{0,1\}, \tau\in [0,1]$, the image of the sections $\Lambda_{s^k}: Q_{\tau,x(t)} \rightarrow i_{\tau,t}^*{\rm Lag}(\tilde X^k,\Omega^k)$ referred to in the proof of Lemma \ref{lag5} can be chosen to be $\mathbb{R}^n\times \{0\}\subset T_x\tilde X^k$ for any $x \in Q_{\tau,x(t)}$. Note further that the genericity Assumption \ref{ass4} is formulated here and will be used only for the representatives $N^{\mathcal{M}}_\tau, \hat N^{\mathcal{M}}_\tau$ resp. their intersection with $Q_{\tau,x(t)}$ for $\tau\in [0,1],\  t \in \{0,1\}$, $\hat N^{\mathcal{M}}_{\tau,t}$, and we will in the following drop the upper suffixes $\mathcal{M}$ frequently in the course of the arguments.\\
{\it Remark.} Note that the genericity part of (1.) in Assumption \ref{ass4} is satisfied if $\Lambda_{Q_{\tau,x(t)}} \in \Gamma(i_{\tau,t}^*\widetilde{\rm Lag}(\tilde X^k,\Omega^k))$ intersects the union of the sets $\mathcal{M}_x, \ x \in Q_{\tau,x(t)}$ for all $\tau,t \in [0,1]$ transversally outside of a discrete set of points, which can be achieved by a small perturbation of the sections $\Lambda_{Q_{\tau,x(t)}}$ for $t,\tau =\{0,1\}$ resp. the family $Q_{\tau,x(t)}, \ t\in \{0,1\}$ without affecting the Maslov class of $[N_\tau]\in H_n(Q_\tau,\mathbb{Z})$ for $\tau =\{0,1\}$, this will be proven in Section \ref{app4}. The non-emptyness assumption in (1.) above follows for all $\tau \in [0,1]$ and $t=0$ since the non-vanishing of the class $[N_0] \in H_n(Q_0,\mathbb{Z})$ follows from assuming the non-vanishing of ${\rm wind}(\alpha)= m(\sum_{i=1}^\mu \beta_i-\beta)$ (see Lemma \ref{bla456}) and formula (\ref{bla4567}). The validity of Assumption (3.) will be discussed in Section \ref{app4} by representing a neighbourhood of each $Q_{\tau,x(t)}$ in $Q_\tau$ using generating families and subsequently using stability theory (see Eliashberg and Gromov \cite{eliasgr} resp. Guillemin and Sternberg \cite{guille}) to show that certain connected components of the complement of the family of caustics $N^{\mathcal{M}}_{\tau,t}\subset Q_{\tau,x(t)},\ t \in \{0,1\}, \tau\in [0,1]$ do not vanish along the symplectic isotopy that is underlying the deformation in the parameter $\tau$ (given the non-occurrence of 'higher singularities'). Note that these considerations are connected with a question posed by Arnol'd concerning the persistence of caustics of wavefronts in families of Lagrangian embeddings (see Ferrand and Pushkar \cite{legendrian}, Entov \cite{entov}).\\
We finally define a loop $c:[0,1]/\{0,1\}\rightarrow Q$ by fixing point $z_0 \in Q_x$, $x \in S^1_\delta$ above, and defining a map
\begin{equation}\label{tildec}
\begin{split}
\tilde c&:  [0,1] \rightarrow \hat Q\\
\tilde c(\tau)&= \left\{\begin{matrix}(z_0,\delta,\tau), \ \tau \in [0,1-r]\\  \left((\rho^k_{\psi(\tau)})(z_0),\delta,\tau\right), \ \tau \in [1-r,1].\end{matrix}\right.
\end{split}
\end{equation}
this factorizes to a well-defined smooth map $c:[0,1]/\{0,1\} \rightarrow Q \subset \tilde X$. We then have the following:
\begin{prop}\label{keylemma}
Let $g_1=e^{i\theta}:Q \rightarrow S^1$ be determined as described in the formulation of Lemma \ref{bla456}. Then if ${\rm dim}(Q_x)\geq 2$ and with the above notations, i.e. (\ref{maslovindex36}) and (\ref{windingbla}) we have
\[
{\rm wind}(\alpha)={\rm wind}(e^{i\theta\circ c})-k,
\]
On the other hand, ${\rm wind}(e^{i\theta \circ c})=k$, which thus implies ${\rm wind}(\alpha)=0$.
\end{prop}
\begin{proof}
We will first prove that
\[
{\rm wind}(\alpha)={\rm wind}(g_1\circ \hat c)-k,
\]
where $\hat c:[0,1]\rightarrow Q$ is an arbitrary smooth path lifting $t \mapsto xe^{2\pi it}$ for a fixed $x\in S^1_\epsilon$ and $Q=Q_1$ is as defined in (\ref{hatq}) resp. (\ref{hatqt}) below. Fixing $M:=(Y_\epsilon^k)_x,\ x \in S^1_\epsilon$, choose an isotopy $\rho^k_{(\cdot)}:[0,1]\rightarrow {\rm Symp}(M,\partial M,\omega)$ s.t. $\rho^k_0=Id$, $\rho^k_1=\rho^k$, where $\rho^k$ is the symplectic monodromy of $Y^k_\epsilon$. Recall that symplectic parallel transport in $Y_\epsilon^k$ defines a symplectomorphism
\begin {equation}\label{mapcyk3}
\overline \Theta: Y^k_\epsilon \simeq \overline Y^k:=\left([0,1]\times M\right)/\left((0,z)\sim (1,\rho^k(z)\right),
\end{equation}
as in Lemma \ref{mappingcylinder}. Denote $\pi_0: [0,1]\times M\rightarrow \overline Y^k$ the canonical projection. Set $\overline \Omega^k:= (\overline \Theta^{-1})^*\Omega^k|Y_\epsilon^k$ and $\Phi_t^{\overline \Omega^k}:=\overline \Theta\circ\Phi_t^{\Omega^k}\circ (\Theta^{-1})$, where $\Phi_t^{\Omega^k}$ denotes symplectic parallel transport in $Y_\epsilon^k$ w.r.t. the curve $t \mapsto xe^{2\pi it}$. For any $\tau \in [0,1]$ define a family of diffeomorphisms $\tilde \Phi^{\Omega^k}(\cdot,t,\tau):(Y_\epsilon^k)_x \rightarrow (Y_\epsilon^k)_{xe^{2\pi it}}, t \in [0,1]$ which are w.r.t. (\ref{mapcyk3}) and for any $z \in  (Y_\epsilon^k)_{xe^{2\pi it}}$ given as
\begin{equation}\label{triv46}
\tilde \Phi^{\Omega^k}(z,t, \tau)=\left \{ \begin{matrix} \Phi_{H(t)}((1-\tau)t)\circ\Phi^{\overline \Omega^k}(z,t),\ t \in [0,1-r),\ \tau \in [0,1] \\
\Phi_{H(t)}((1-\tau)t)\circ\rho^k_{\tau\psi(t)}\circ \Phi^{\overline \Omega^k}(z,t),\ t \in [1-r,1]\ \tau \in [0,1]\end{matrix}\right.,
\end{equation}
where $\psi:[1-r,1]\rightarrow [0,1]$ is smooth s.t. $\psi(1-r)=0$, $\psi(1)=1$, $\psi'(1-r)=\psi'(1)=0$. Here we use the notation of Lemma \ref{metricsubspace}, that is $\Phi_H$ denotes the image of the flow of the lift of the Euler vector field $K= \sum_i w_i z_i \frac{\partial}{\partial z_i}, \ z\in Y_\epsilon$ to $Y^k_\epsilon$ under $\overline \Theta$. We see that for each $\tau\in \{0,1\}$ the family $\tilde \Phi^{\Omega^k}(\cdot,t,\tau)$ symplectically trivializes $Y^k_\epsilon$, $\tau=0$ corresponds to the trivialization given by the weighted circle action restricted to $Y^k_\epsilon$, $\tau=1$ to the trivialization induced by the chosen isotopy $\rho^k_t \in {\rm Symp}(M,\partial M,\omega), t \in [0,1]$. Then for $\tau \in [0,1]$ define $\hat Q_\tau\subset M\times [0,1]$ fibering into Lagrangian submanifolds $\hat Q_{\tau, t}:=\hat Q_\tau \cap (M\times \{t\})$ over $[0,1]$ by
\begin{equation}\label{hatqt}
Q_\tau=\pi_0\circ\hat Q_\tau=\bigcup_{t \in [0,1]}\tilde  \Phi^{\Omega^k}(Q_x,t,\tau), \ f^k_e(\pi_0(\hat Q_{\tau,t}))=xe^{2\pi it}, \ t \in [0,1],
\end{equation}
where $Q_x\subset M$ is as in Ass. \ref{ass2}. For $\tau\in \{0,1\}$, $\hat Q_\tau$ factorizes to a well-defined closed submanifold $Q_\tau\subset \overline Y^k\simeq Y^k_\epsilon$ (fibering into Lagrangians $Q_{\tau,t}=\pi_0(\hat Q_{\tau, t})$). Let $i_\tau: \hat Q_\tau\rightarrow \tilde X^k$ be the immersion onto the image of $\pi_0(\hat Q_\tau)$. By Lemma \ref{lag5}, we can associate to any point $z'\in \hat Q_\tau, \ \tau \in [0,1]$ an element $\kappa_{\hat Q_\tau}(z')$ of $i_\tau^*(\Lambda^{n+1,0}(T^*_{z'}\tilde X^k))$ induced by the Lagrangian subspace 
\begin{equation}\label{lagsubspaces56}
T^h_{z'}Y^k_\epsilon \oplus T_{z'}Q_{\tau,t}\subset T_{z'}\tilde X^k, \ y(t)=f^k_e(z'),\ t \in [0,1],
\end{equation}
(recall $Y^k_\epsilon \subset \tilde X^k$). Define a family of functions $g_\tau:=e^{i\theta_\tau}:\hat Q_\tau\rightarrow S^1$ for $\tau\in [0,1]$ by setting
\begin{equation}\label{bla25}
e^{i\theta_\tau}i_\tau^*\kappa_{\hat Q_\tau}=\left(i_\tau^*(\pi^k)^*(dz_0\wedge\dots\wedge dz_n)\right),
\end{equation}
where here as above, we use Lemma \ref{mappingcylinder} and the immersions $i_\tau: \hat Q_\tau\rightarrow \tilde X^k$. Of course, for $\tau=0,1$, $g_\tau$ factorize to functions $g_\tau:Q_\tau\rightarrow S^1$ (using the same symbols).   \\
For the following, fix $x \in S^1_\epsilon$ in (\ref{mapcyk3}) and (\ref{triv46}) so that (1.) and (2.) in Assumption \ref{ass2} and Assumption \ref{ass4} (1.) are satisfied. Then following Lemma \ref{lag5} and the discussion above, we have relative $n$-cycles $\hat N^\mathcal{M}_{\tau}\in H_{n}(\hat Q_\tau,\partial \hat Q_\tau,\mathbb{Z})$ for $\tau \in [0,1]$ so that for $\tau=0,1$ we have after factorizing $PD[g_\tau^*\beta]=[N^\mathcal{M}_{\tau}] \in H_{n}(Q_\tau,\mathbb{Z})$. By Assumption \ref{ass4} (from now on dropping the upper suffix $\mathcal{M}$), $\hat N_\tau \subset \hat Q_\tau, \tau \in [0,1]$, resp. $N_\tau \subset Q_\tau, \tau \in \{0,1\}$ are in fact Whitney stratified spaces with cooriented smooth top-strata $\hat N^{top}_\tau, N^{top}_\tau$ of codimension one and singular sets of at least codimension $3$ in $\hat Q_\tau$ resp. $Q_\tau$. Let $\gamma_\tau:S^1 \rightarrow Q_\tau,\ \tau \in \{0,1\}$ be chosen so that it intersects $N^{top}_\tau$ transversally, generates $H_1(Q_\tau,\mathbb{Z})/{\rm Tor}$ respectively and so that $\gamma_\tau \cap N_{\tau,x}=\emptyset,\ \tau \in \{0,1\}$, where $N_{\tau,x}=N_\tau\cap Q_x$. We then have the following claims:
\begin{enumerate}
\item $N_{\tau}\cdot\gamma_\tau \in \mathbb{Z}$ coincide for $\tau \in \{0,1\}$, where $\cdot$ denotes the geometric intersection number of cycles.
\item We have $N_{0}\cdot\gamma_0-k={\rm wind}(\alpha)$, $N_1\cdot\gamma_1={\rm wind}(e^{i\theta_1\circ c})$ where $c: [0,1]\rightarrow Q_1$ is any closed smooth path generating $H_1(Q_1,\mathbb{Z})/{\rm Tor}$.
\end{enumerate}

{\it Proof of the Claim (1.) using Assumption \ref{ass4} (3.)}\\
To prove the first claim using in addition Assumption \ref{ass4} (3.), note that there is for each $\tau \in [0,1]$ a diffeomorphism $\hat \Psi_\tau:\hat Q_1\rightarrow \hat Q_\tau,$ by assembling the family of fibrewise symplectomorphisms given by 
\begin{equation}\label{triv45}
\hat \Psi_\tau(z,t)=\left \{ \begin{matrix} \Phi_{H(t)}(z,(1-\tau)t), \ z\in \overline Y^k_t,\ t \in [0,1-r),\ \tau \in [0,1]\\
\Phi_{H(t)}((1-\tau)t)\circ \rho_{\tau\psi(t)}\circ (\rho_{\psi(t)})^{-1}(z,t),\ z \in \overline Y^k_t,\ t \in [1-r,1]\ \tau \in [0,1],\end{matrix}\right.
\end{equation}
which by (\ref{triv46}) restrict to mappings $\hat \Psi_\tau(z,t):\hat Q_{1,t}\rightarrow \hat Q_{\tau,t}$. These factorize for $\tau=0,1$ to diffeomorphisms $\Psi_\tau:Q_1\subset \overline Y^k \rightarrow Q_\tau\subset \overline Y^k$, which restrict to the identity on $\overline Y^k_0= (Y^k_\epsilon)_x$. Consider $[\Psi_0^{-1}(N_0)] \in H_{n}(Q_1,\mathbb{Z})$ and $[\Psi_0^{-1}(\gamma_0)] \in H_1(Q_1, \mathbb{Z})$, then by functoriality of the intersection number we have $N_0^1\cdot \tilde \gamma_0 :=\Psi_0^{-1}(N_{0})\cdot\Psi_0^{-1}(\gamma_0) =N_{0}\cdot\gamma_0$. Consider further $\hat N^1_\tau:=\hat \Psi_\tau^{-1}(\hat N_\tau)\subset \hat Q_1,\ \tau \in [0,1]$. Then if we denote for $\tau\in  [0,1]$ $\hat N^{1,top}_\tau\subset Q_1$ the (closure of) smooth topstratum of $\hat N^{1}_\tau$, we have
\[
\partial\hat N^{1,top}_\tau= \hat N^{1,top}_{\tau,0} \sqcup \hat N^{1,top}_{\tau,1} \subset \hat Q_{1,0}\sqcup \hat Q_{1,1}\subset \hat Q_1,
\]
where $\partial$ here means the geometric boundary and for $\tau \in [0,1]$ we set
\[
\hat N^{1,top}_{\tau,t}:=\hat N^{1,top}_\tau\cap (M\times \{t\})\subset \hat Q_1 \cap (M\times \{t\}), \ t \in \{0,1\}.
\]
Note that since $\Psi_0|Q_x=Id_{Q_x}$, we have 
\begin{equation}\label{blaeq}
\hat N^{1,top}_{0,0}=-\hat N^{1,top}_{0,1}=\hat N^{1,top}_{1,0}=-\hat N^{1,top}_{1,1}\subset Q_x.
\end{equation}
Now factorize $N^{1,top}_\tau=\pi_0(\hat N^{1,top}_\tau)\subset Q_1, \tau \in [0,1]$ and $N^{1,top}_{\tau,t}=\pi_0(\hat N^{1,top}_{\tau,t})\subset Q_x, \tau \in [0,1], t \in \{0,1\}$ an consider the union
\[
\mathcal{N}^{1,top} =\bigcup_{\tau \in [0,1]} N^{1,top}_\tau \subset Q_1 \times [0,1]=:\mathcal{Q}_1.
\]
Note that this is an {\it oriented} $n+1$-chain in $\mathcal{Q}_1$, where the orientation is inherited by a relative version of Lemma \ref{lag5} resp. Definition \ref{generic} outside of a discrete set of isolated points (compare Lemma \ref{genericity}). Then if $\mathcal{N}^{1,top}_{t}=\bigcup_{\tau \in [0,1]}N^{1,top}_{\tau,t} \subset \mathcal{Q}_1$ we have that $0=[\mathcal{N}^{1,top}_{0}\cup \mathcal{N}^{1,top}_{1}] \in H_{n+1}(\mathcal{Q}_1, \partial \mathcal{Q}_1, \mathbb{Z})$ with relative primitive $\mathcal{N}^{1,top}$, that is, $\partial_{rel}\mathcal{N}^{1,top}=\mathcal{N}^{1,top}_{0}\cup \mathcal{N}^{1,top}_{1}$, where the boundary is taken relative $\partial \mathcal{Q}_1$. So we write $\partial \mathcal{N}^{1,top}:= \mathcal{N}^{1,top}_{0}\cup \mathcal{N}^{1,top}_{1}\subset \mathcal{Q}_1$. Recall $\tilde \gamma_0=\Psi_0^{-1}(\gamma_0)$ and consider the union $\gamma^\Delta=\tilde \gamma_0 -\gamma_1 \subset Q_1\times \{0\}\cup Q_1\times \{1\} \subset \mathcal{Q}_1$. Since the following arguments will only depend on the homology classes of $\tilde \gamma_0$ and $\gamma_1$ in $Q_1$, since $H_1(Q_x, \mathbb{Z})=0$ and by (\ref{blaeq}) we can assume $\gamma_0:=\tilde \gamma_0=\gamma_1\subset Q_1$. 
Furthermore, using the notation of (3.) in Assumption \ref{ass4}, we can assume that $\gamma_{0,1}\cap Q_x \subset {\rm im} (j_{\tau,t}), \ t\in \{0,1\},  \tau \in \{0,1\}$, i.e. $\gamma_{0,1}\cap N^{1,top}_{\tau,t}=\emptyset$. Thus $[\gamma^\Delta]=0 \in H_1(\mathcal{Q}_1,\mathbb{Z})/{\rm Tor}$ and we have a chain of equalities:
\[
N_{0}\cdot\gamma_0-N_{1}\cdot\gamma_1=N^{1,top}_0\cdot \gamma_0-N^{1,top}_1\cdot\gamma_1=\mathcal{N}^{1,top}\cdot \gamma^\Delta= \partial \mathcal{N}^{1,top} \circledcirc \gamma^\Delta,
\]
where $\circledcirc$ symbolizes the linking pairing. Now let $Q_x^\epsilon:=Q_x\times [-\epsilon, \epsilon]$ for some small positive $\epsilon$, cut $Q_1$ along the separating hypersurface $Q_x$ to obtain $\hat Q_1$ and glue $\partial Q_x^\epsilon$ along $\partial \hat Q_1$ and denote the result by $Q_1^\epsilon$, in the following we will parametrize the 'neck' in $Q_1^\epsilon$ by $Q_x\times [-\epsilon, \epsilon]$. Set $Q^\epsilon_1 \times [0,1]=:\mathcal{Q}^\epsilon_1$. Then $\mathcal{Q}^\epsilon_x:=Q_x^\epsilon\times [0,1]\subset \mathcal{Q}_1^\epsilon$ and $\mathcal{N}^{1,top}_{0}\sqcup \mathcal{N}^{1,top}_{1} \subset \partial Q_x\times \{-\epsilon\} \times [0,1]\sqcup Q_x\times \{\epsilon\}\times [0,1]\subset \mathcal{Q}_1^\epsilon$. By extending $\gamma^\Delta$ constantly along $Q_x^\epsilon$ we can calculate $\partial \mathcal{N}^{1,top} \circledcirc \gamma^\Delta$ in  $\mathcal{Q}_1^\epsilon$ with equal result as in $\mathcal{Q}_1$. As a consequence of 2. in Assumption \ref{ass4}, we can find for each $t\in \{-\epsilon,\epsilon\}$ smooth paths $c_t:[0,1]\rightarrow Q_x\times \{t\} \times [0,1]$ s.t. $c_t(\tau)\in Q_x\times \{t\} \times \{\tau\}, \ \tau \in [0,1]$ so that $c_t(0)=\gamma_0\cap Q_x\times \{t\} \times \{0\}$, $c_t(1)=\gamma_1\cap Q_x\times \{t\}\times \{1\}$ and 
\[
{\rm im} (c_t)\cap (Q_x\times \{t\} \times \{\tau\})\subset \mathcal{U}_{\tau,f(t)},\ t\in \{-\epsilon,\epsilon\}, \ \tau \in [0,1],
\]
where $f(t)=0$ if $t=-\epsilon$ and $1$ otherwise, i.e. $\mathcal{N}^{1,top}_{f(t)}\cap {\rm im} (c_t)=\emptyset$ for $t\in \{-\epsilon, \epsilon\}$. Define smooth paths $\hat \gamma_\tau:[-\epsilon,\epsilon]\rightarrow Q_x^\epsilon\times \{\tau\}$ so that $\hat \gamma_\tau(t)=c_t(\tau)$ for $t \in \{-\epsilon,\epsilon\}$ and $\tau\in [0,1]$ and extend $\hat \gamma_\tau$ for each $\tau \in [0,1]$ smoothly to paths $\gamma_\tau:[0,1]\rightarrow Q_1\times \{\tau\}$ being generators of $H_1(Q_1\times \{\tau\}, \mathbb{Z})/{\rm Tor}$ and coinciding with $\gamma_1=\gamma_0$ for $\tau=0,1$ outside the neck $\mathcal{Q}^\epsilon_x\subset \mathcal{Q}_1$. Thus we have constructed a family $\gamma_\tau\subset Q_1\times\{\tau\}\subset \mathcal{Q}_1, \tau\in [0,1]$ of generators for $H_1(Q_1\times\{\tau\},\mathbb{Z})/{\rm Tor}$ so that $\gamma_{0,1}$ coincide with their previous definitions and further $(\gamma^\Delta_\tau=\gamma_\tau-\gamma_1)\cap \partial \mathcal{N}^{1,top}=\emptyset$ for any $\tau \in [0,1]$. Thus the family of pairs $(\partial \mathcal{N}^{1,top},\gamma^\Delta_\tau)$ defines a link homotopy (see \cite{linking}) so that
\[
\partial \mathcal{N}^{1,top} \circledcirc \gamma^\Delta_0=\partial \mathcal{N}^{1,top} \circledcirc \gamma^\Delta,
\]
and $\partial \mathcal{N}^{1,top} \circledcirc \gamma^\Delta_1=0$, so we arrive at the assertion of Claim (1.)\\

{\it Proof of the Claim (1.) using Assumption \ref{ass4} (2.)}\\
We will explain how to modify the above if (3.) in Assumption \ref{ass4} is replaced by Assumption \ref{ass4} (2.) (in fact we need a little stronger result with regard to coorientation by the kernels of 'vertical Hessians', cf. Proposition \ref{keysing} (3.) and the remark below). Here, we make essential use of the fact that $H^1(Q_x, \mathbb{Z})=0$ for $n\geq 2$. Above we defined a family of submanifolds $\hat Q_\tau\subset M\times [0,1], \ \tau \in [0,1]$, where $M:=(Y^k_\epsilon)_x$ for some fixed $x\in S^1_\epsilon$ so that for $\tau \in [0,1]$ we have immersions $i_\tau: \hat Q_\tau \rightarrow Y^k_\epsilon\subset \tilde X^k$ whose images $Q_\tau:={\rm im }(i_\tau)$ define for $\tau=0,1$ closed $n+1$-dimensional submanifolds $Q_1,\ Q_2\subset Y^k_\epsilon\subset \tilde X^k$. Recall that the $Q_\tau$ are just $\pi_0(\hat Q_\tau)$ where $\pi_0:M\times [0,1]\rightarrow \overline Y^k\simeq Y^k_\epsilon$ is the canonical projection onto $Y^k_\epsilon$, represented as a symplectic mapping cylinder. Following Definition \ref{generic} and using the family of Lagrangian subspaces (\ref{lagsubspaces56}) over $\hat Q_\tau, \tau \in [0,1]$ we defined for $\tau \in [0,1]$ representants $\hat N^\mathcal{M}_{\tau}\subset \hat Q_\tau$ of classes $[\hat N^\mathcal{M}_{\tau}] \in H_{n}(\hat Q_\tau,\partial \hat Q_\tau,\mathbb{Z})$ which factor to representants $N^\mathcal{M}_{\tau}$ of classes $[N^\mathcal{M}_{\tau}] \in H_{n}(Q_\tau,\mathbb{Z})$ for $\tau=0,\ 1$. We choose $\gamma_\tau:S^1 \rightarrow Q_\tau,\ \tau \in \{0,1\}$ so that they intersect the oriented top-strata $N^{top}_\tau$ transversally, generate $H_1(Q_\tau,\mathbb{Z})/{\rm Tor}$ respectively and so that $\gamma_\tau \cap N^\mathcal{M}_{\tau,x}=\emptyset,\ \tau \in \{0,1\}$, where $N^\mathcal{M}_{\tau,x}=N^\mathcal{M}_\tau\cap Q_x$ (note $Q_x=Q_{\tau,x(0)}$ in the above). Now to prove claim (1.), that is that $N^\mathcal{M}_{\tau}\cdot\gamma_\tau \in \mathbb{Z}$ coincide for $\tau \in \{0,1\}$, where $\cdot$ denotes the geometric intersection number of cycles, let 
\[
\mathcal{Q}^0:=\bigcup_{\tau\in [0,1]}\hat Q_\tau\times \{\tau\}\subset (M\times [0,1])\times [0,1],
\]
and let $\hat \gamma_\tau:[0,1] \rightarrow \hat Q_\tau,\ \tau \in \{0,1\}$ be chosen so that $\pi_0(\hat \gamma_\tau(t))=\gamma_\tau(t)$ for all $t \in [0,1]$. Then set
\[
\mathcal{N}^0=\bigcup_{\tau \in [0,1]} \hat N^\mathcal{M}_{\tau}\times \{\tau\}\subset \mathcal{Q}^0,
\]
and note that, by Assumption \ref{ass4} and (3.) of Proposition \ref{keysing}, the top-stratum of $\mathcal{N}^0$ is cooriented outside of a discrete subset and this coorientation coincides outside of a subset $\mathcal{C}^1$ of codimension $1$ in $\bigcup_{\tau \in[0,1]}Q_{\tau,x(1)}\cap \pi_0(\mathcal{N}^0)$ with the coorientations of the individual $\hat N^{1,top}_{\tau,t},\ t = 1,\tau \in [0,1]$ induced by the kernels of the vertical Hessians of the functions generating the $L_\tau, \ \tau \in [0,1]$ with the notation of Section Proposition \ref{keysing}. Note also that we can always assume that the family of immersions $i_\tau: Q_x\times [0,1] \rightarrow Y^k_\epsilon$ defined by (\ref{hatqt}) above is constant in a neighbourhood of $\tau=0,1$. Then the coorientation of the top-stratum of $\mathcal{N}^0$ will induce the given coorientations on the top-strata of $N^\mathcal{M}_{\tau}, \tau =0,1$ (outside of discrete subset). Thus $\mathcal{N}^0$ defines a class $[\mathcal{N}^0] \in H_{n+1}(\mathcal{Q}^0, \partial \mathcal{Q}^0,\mathbb{Z})$. Now assume that we have chosen $\gamma_0, \gamma_1$ so that $\hat \gamma_0(0)=\hat \gamma_1(0)\notin \mathcal{N}^0$ (which is always possible). Then let $\hat \gamma_2:[0,1]\rightarrow (\mathcal{Q}^0 \cap M\times \{0\}\times [0,1])= Q_x\times \{0\}\times [0,1]$ be the path $\hat \gamma_2(\tau)=\hat \gamma_0(0)\times \{0\}\times \{\tau\} \in (\hat Q_\tau\cap M\times \{0\})\times \{\tau\},\ \tau \in [0,1]$. Then we have that $\hat \gamma_0(\tau)\cap \mathcal{N}^0=\emptyset$ for all $\tau \in [0,1]$. On the other hand, let $\hat \gamma_3:[0,1]\rightarrow (\mathcal{Q}^0 \cap M\times \{1\}\times [0,1])$ be the path whose existence is guaranteed by (2.) of Proposition \ref{keysing} above so that $\hat \gamma_3(0)=\hat \gamma_0(1), \hat \gamma_3(1)=\hat \gamma_1(1)$ and $\hat \gamma_3(\tau)\in (\hat Q_\tau\cap M\times \{1\}) \times \{\tau\}$ for all $\tau \in [0,1]$. We can assume that $\pi_0(\hat \gamma_3)$ does not intersect $\mathcal{C}^1\subset \bigcup_{\tau \in[0,1]}Q_{\tau,x(1)}\cap \pi_0(\mathcal{N}^0)$) (thus intersects the latter only in 'fold points'). Then $\hat \gamma_3\cdot \mathcal{N}^0=0$ by (2.) of Proposition \ref{keysing}. Finally note that
\[
\hat \gamma: S^1\rightarrow \mathcal{Q}_0, \ \hat \gamma=(-\hat \gamma_2)\star(-\hat \gamma_1)\star \hat \gamma_3 \star \hat \gamma_0, 
\]
where $\star$ means concatenation and the $-$-sign reverse of orientation, defines a contractible loop in $\mathcal{Q}^0$, hence $0=[{\rm im }\ \hat \gamma] \in H_1(\mathcal{Q}^0, \mathbb{Z})$, hence counting intersection indices along $\hat \gamma$ we arrive at 
\[
0= \hat \gamma \cdot \mathcal{N}^0= -\hat \gamma_1\cdot\mathcal{N}^0 +\hat \gamma_0\cdot\mathcal{N}^0= -\hat N^\mathcal{M}_{1}\cdot\hat \gamma_1+ \hat N^\mathcal{M}_{0}\cdot\hat \gamma_0,
\]
where the latter intersection numbers are determined in $\hat Q_0,\ \hat Q_1$, respectively, which, after factorizing by $\pi_0$, gives the assertion.\\

{\it Proof of the Claim (2.)}\\
To prove the second claim, recall that by (\ref{eulermulti}) we have $\Phi_{X_K}(t)^*s^k=e^{2\pi i\gamma t}s^k$ for $t \in [0,1]$, where $\Phi_{X_K}$ is parallel transport along $t \mapsto xe^{2\pi it}$ induced by the horizontal distribution $H^k_f\subset \tilde X^k$ spanned by the Euler vector field $K$ as defined in the proof of Lemma \ref{metricsubspace}, $\gamma =m(\sum_i \beta_i-\beta) \in \mathbb{Z}$ and $k=m\beta,\ m \in \mathbb{N}^+$. On the other hand $\alpha: S^1\rightarrow \mathbb{C}^*$ is given by (fixing $M=(Y^k_\epsilon)_u,\ u  \in S^1_\epsilon$ as above and setting $Q_{x}=Q_0\cap M$, $Q_0$ as in (\ref{hatqt}))
\begin{equation}\label{blawind}
\alpha(t):={\rm ev}(Q_{x})(\Phi_{X_K}(t)^*s^k)={\rm ev}(Q_{x})(e^{2\pi i\gamma t}s^k)=\int_{Q_{0,x}}e^{i\hat \theta_0(t)}i_{X_{f^k}}\kappa_{Q_{0}}(t),\ t \in [0,1].
\end{equation}
Here, $e^{i\hat \theta_0(t)}i_{X_{f^k}}\kappa_{Q_0}(t)=\Phi_{X_K}(t)^*(e^{i\theta_0(t)}i_{X_{f^k}}\kappa_{Q_{0,t}})$, where $e^{i\theta_0(t)}=e^{i\theta_0}|Q_{0,t}$. Denoting $|X_{f^k}|$ for the norm of $X_{f^k}$ with respect to the metric $g$ on $Q_{0,x}\times [0,1]$ induced by the metric $\Omega^k(\cdot,J\cdot)$ on $Q_0$ we have
\begin{equation}\label{blawind2}
i_{X_{f^k}}\kappa_{Q_{0,t}}=e^{-2\pi i kt}|X_{f^k}| {\rm vol}_{Q_{0,t}} \ \in \Omega^n(Q_{0,t},\mathbb{C}),\ t \in [0,1] 
\end{equation}
where here, ${\rm vol}_{Q_{0,t}}$ denotes the volume-form on $Q_{0,t}=\pi_0(\hat Q_{0,t})$ induced by the restricted metric. To see this, note that for $u \in Y^k_\epsilon$ with $f^k(u)=xe^{2\pi it}$,
\[
i_{X_{f^k}}(e_0^*-iJe_0^*)(u)=e^{-2\pi i kt}|X_{f^k}|(u),
\]
where $e_0$ is a local horizontal (w.r.t. $\Omega^k$) unit vector field of $TQ_0$, and $2\pi kt, \ t\in [0,1]$ measures the angle between $X_{f^k}$ and $e_i$ (which is fibrewise constant since $J|H_{\Omega^k}=(f^k)^*(j)$). Comparing (\ref{blawind}) and (\ref{blawind2}) and using that $\Phi_{X_K}$ acts by fibrewise isometries, we see that ($\gamma \in \mathbb{Z}$ as in (\ref{maslovindex36})) $e^{i\hat \theta_0(t)}=e^{2\pi i\gamma t}e^{i\theta_0(0)}$ and $|X_{f^k}| {\rm vol}_{Q_{0,t}}={\rm const.}$, so we infer
\begin{equation}\label{bla4567}
\begin{split}
{\rm wind}(\alpha)&=\frac{1}{2\pi i}\int_{[0,1]}\alpha^{-1}(t)\partial_t \alpha(t)dt\\
&=\frac{1}{2\pi}\int_{[0,1]}\alpha^{-1}(0)\int_{Q_{x}}\hat \theta_0'(t)dt\wedge e^{i\hat \theta_0(0)}|X_{f^k}|(0) {\rm vol}_{Q_{0,x}}-k\\
&=\frac{1}{2\pi}{\rm ev}(Q_{x}\times [0,1])(d\hat \theta_0\wedge\pi_1^*\delta)-k\\
&= N_{0}\cdot\gamma_0 - k.
\end{split}
\end{equation}
where $\delta \in H^n(Q_{x},\mathbb{C}),\ {\rm ev}(\delta)(Q_{x})=1$ and $\pi_1:Q_{x}\times [0,1]\rightarrow Q_{x}$ is the obvious projection. In the last line, we have used that $PD[\pi_1^*\delta]=\gamma_0$ and $PD[g_0^*\beta]=N_0$, $\beta \in H^1(C^*, \mathbb{Z})$ the generator. The second assertion of Claim 2. is essentially true by the definition of $N_1$ as the Poincare dual of $PD[g_1^*\beta]$.\\

We will now show the second assertion of Proposition \ref{keylemma}, that is that we also have ${\rm wind}(e^{i\theta \circ c})=k$, where $c$ is as in (\ref{tildec}) (note that by definition, $e^{i\theta_1}=e^{i\theta}$ on $Q_1=Q$). For this, choose any smooth path $v:[0,1]\rightarrow M=Y_x^k$ so that $v(0)=c(0)=z_0$ and $v(1) \in \partial Y_x^k$, where $z_0\in M$ is as in (\ref{tildec}) and define the map (recall $Y^k=(f^k)^{-1}(S^1_\delta)$)
\begin{equation}\label{tildeu}
\begin{split}
\tilde u&:  [0,1]\times  [0,1]\rightarrow M\times \{\delta\}\times [0,1]\\
\tilde u(t,\tau)&= \left\{\begin{matrix}(v(t),\delta,\tau), \ \tau \in [0,1-r]\\  \left((\rho^k_{\psi(\tau)})(v(t)),\delta,\tau\right), \ \tau \in [1-r,1].\end{matrix}\right.
\end{split}
\end{equation}
This factorizes to a well-defined map $u:[0,1]\times [0,1]/\{0,1\}\rightarrow Y^k\subset \tilde X^k$, so that with $F:={\rm im}(u)$ we have $\partial F={\rm im}(c) \cup{\rm im}(\hat c)$, where $\hat c(\tau):=u(1,\tau), \ \tau \in [0,1]/\{0,1\}$ is a closed smooth path s.t. ${\rm im}\ \hat c\subset \partial Y^k$. Let now $\mathcal{L}(M)\rightarrow M$ be the fibrebundle whose fibre at $z\in M$ is the Lagrangian Grassmannian $\mathcal{L}(M)_z=\mathcal{L}(T_zM, \omega_z)$, where $\omega$ is the symplectic form on $M$. Then, if $P\rightarrow M$ is the $Sp(2n,\mathbb{R})$ principal bundle of symplectic frames associated to $(TM,\omega)$, one knows that $\mathcal{L}(M)=P\times_{Sp(2n)}\mathcal{L}(2n)$, where $\mathcal{L}(2n)$ is the Lagrangian Grassmannian of $(\mathbb{R}^{2n},\omega_0)$ w.r.t. the standard symplectic structure $\omega_0$. Using this, we see that that any symplectic connection $\nabla^\omega$ on $M$, that is $\nabla^\omega\omega=0$, defines the notion of a parallel transport in $\mathcal{L}(M)$ along paths in $M$ (alternatively, we can take the Levi-Civita-connection on $M$ and the $U(n)$-reduction of $P$ w.r.t. $\tilde J_x$ on $M$). Fix the element $\mathbf{Q}(M)_0\in \mathcal{L}(M)_{z_0}$ which is given by the Lagrangian subspace $TQ_{z_0} \subset T_xM$ and define the path
\[
\mathbf{Q}(M):[0,1]\rightarrow \mathcal{L}(M), \ \mathbf{Q}(M)(t)=\mathcal{P}_{z_0,v(t)}\mathbf{Q}(M)_0, \ t\in [0,1],
\]
where $\mathcal{P}_{z_0,v(t)}:[0,1]\times \mathcal{L}(M)_{z_0}\rightarrow \mathcal{L}(M)_{v(t)}$ is parallel transport in $\mathcal{L}(M)$ along $v$. Let now $\mathcal{L}^v(Y^k)\rightarrow Y^k$ be the fibrebundle whose fibre at one point $z \in Y^k$ with $f^k(z)=y$ is $\mathcal{L}(Y^k_y)=\mathcal{L}(T_zY^k_y, (\omega_y)_z)$, where $\omega_y, y\in S^1_\delta$ denotes the symplectic form given on $Y^k_y$. We define a section $\mathbf{Q}^v$ of $\mathcal{L}^v(Y^k)|F$ by factorizing
\begin{equation}\label{lagpoli}
\begin{split}
\tilde {\mathbf{Q}}^v&:  [0,1]\times  [0,1]\rightarrow \mathcal{L}(M)\times \{\delta\}\times [0,1]\\
\tilde {\mathbf{Q}}^v(t,\tau)&= \left\{\begin{matrix}(\mathbf{Q}(M)(t),\delta,\tau), \ \tau \in [0,1-r]\\  \left((\rho^k_{\psi(\tau)})_*(\mathbf{Q}(M)(t)),\delta,\tau\right), \ \tau \in [1-r,1]\end{matrix}\right.,
\end{split}
\end{equation}
where $(\rho^k_{\psi(\tau)})_*$ denotes the natural action of the symplectomorphism $\rho^k_{\psi(\tau)}$ on $\mathcal{L}(M)$. Now let $\kappa_{\mathbf{Q}}\in \Gamma(\Lambda^{n+1,0}(T^*\tilde X^k)|F)$ be the $(n+1,0)$-form along the image of $u$ associated to the Lagrangian distribution $\mathbf{Q}\in \Gamma(\mathcal{L}(\tilde X^k,\Omega^k)|F)$ along ${\rm im}\ u$ defined for each $z= u(t,\tau)$ by
\begin{equation}\label{lagsubspaces2}
\mathbf{Q}(z)=T^h_{z}Y^k \oplus \mathbf{Q}^v(t,\tau) \subset T_{z}\tilde X^k, \ t,\tau \in [0,1].
\end{equation}
Then we have by construction of $\eta_Q$ and $Q$ above $\eta_Q|{\rm im}\ u =\eta_{\mathbf{Q}}|{\rm im}\ u$, note that $u$ was defined so that $u(0,\tau)=c(\tau), \ \tau \in [0,1]$. Understanding this, using Lemma \ref{relativen} and interior multiplication by $X_{f^k}$ in (\ref{bla987}) it is clear that the function $e^{i\vartheta}:F={\rm im}\ u \rightarrow S^1$ defined by
\begin{equation}\label{bla9876}
e^{i\vartheta}i_{X_{f^k}}\kappa_{\mathbf{Q}}=\mathbf{s}^k|F,
\end{equation}
coincides over ${\rm im}\ c\subset F$ with $e^{i\theta}:{\rm im}\ c\rightarrow S^1$. The logarithmic derivative $\sigma:=d ({\rm log}(e^{i\vartheta}))$ then defines a closed $1$-form $\sigma \in \Omega^1({\rm im}\ u)$ which coincides over $c([0,1-r])\subset Q$ with the mean curvature form, $\sigma_Q$ of $Q^0$, the 'Lagrangian part' of $Q$. Now it is clear since $\sigma$ is closed that 
\[
{\rm wind}(e^{i\theta\circ c})={\rm ev} (\sigma)(c)={\rm ev} (\sigma)(\hat c)={\rm wind}(e^{i\vartheta\circ \hat c})
\] 
since $\partial F={\rm im}\ c \cup{\rm im}\ \hat c$. Now by the triviality condition of $\mathbf{s}^k$ along $\partial Y^k$ and the construction of $\kappa_{\mathbf{Q}}$ it follows for $\tau \in [0,1]$ by using (\ref{insertk})
\[
e^{i\vartheta\circ \hat c}\hat c^*(i_{X_{f^k}}\kappa_{\mathbf{Q}})(\tau)=e^{i\vartheta\circ \hat c}e^{-2\pi ik\tau}|X_{f^k}\circ \hat c|\hat c^*(\kappa_{\mathbf{Q}^v})(\tau)=\hat c^*(\mathbf{s}^k)(\tau),
\]
where $\kappa_{\mathbf{Q}^v}$ is the $(n,0)$-form acting on $(TY^k)^v|F$ associated along $u$ to the section $\mathbf{Q}^v$ of $\mathcal{L}^v(Y^k)|F$ which is by Lemma \ref{relativen} constant along $\hat c$ w.r.t. the trivialization $\Theta^k$, that is, $e^{i\vartheta\circ \hat c(\tau)}e^{-2\pi ik\tau}={\rm const}$ which implies ${\rm wind}(e^{i\theta\circ c})=k$. Thus to summarize, using the fact (Claim 2.) that ${\rm wind}(e^{i\theta\circ c})=N_{1}\cdot\gamma_1$ and using (\ref{bla4567}) in conjunction with the first claim above, which says that $N_{1}\cdot\gamma_1=N_{0}\cdot\gamma_0$, gives ${\rm wind}(\alpha)=0$.
\end{proof}
It is clear that in view of (\ref{maslovindex36}) that Proposition \ref{keylemma} proves Theorem \ref{theorem34}.

\subsection{Proof of the key assumption}\label{app4}
The objective in the following will be to give a proof of Assumption \ref{ass4}. The first part (Lemma \ref{genericity}) deals with (finite dimensional) transversality results for smooth mappings applied to our situation and is fairly standard, except probably for results on 'transversality of families of mappings' that were used. The second part (Prop. \ref{keysing}) is a bit less standard and adopts certain stability concepts for 'fold'-singularities of subgraphical Lagrangian varieties inspired by Eliashberg/Gromov \cite{eliasgr} and Guillemin/Sternberg \cite{guille}.\\
To begin with, recall that in the proof of Proposition \ref{keylemma} we defined for $\tau \in [0,1]$ immersions $i_\tau: Q_x\times [0,1] \rightarrow Y^k_\epsilon\subset \tilde X^k$ whose images $Q_\tau:={\rm im }(i_\tau)$ define for $\tau=0,1$ closed $n+1$-dimensional submanifolds $Q_1,\ Q_2\subset Y^k_\epsilon$ so that any intersection $Q_\tau\cap (Y^k_\epsilon)_u,  \tau \in [0,1], u \in S^1_\epsilon$ is a Lagrangian sphere (resp. a union of two Lagrangian spheres for $u=x$) in the Milnor fibre $(Y^k_\epsilon)_u$, where here $x \in S^1_\epsilon$ is a fixed point. The family $Q_{\tau,x(t)}=Q_{\tau,x(t)}=i_\tau(Q_x\times \{t\}),\ \tau\in [0,1],\ t \in [0,1]$, where $x(t)=xe^{2\pi it}$ comes equipped with two families of sections: $\Lambda_{Q_{\tau,x(t)}} \in \Gamma(i_{\tau,t}^*{\rm Lag}(\tilde X^k,\Omega^k))$ as defined in (\ref{lagsubspaces57}) and $\Lambda_{s^k, \tau,t}:=\pi_k^*(\mathbb{R}^{n+1}\times\{0\})\in \Gamma(i_{\tau,t}^*\tilde {\rm Lag}(\tilde X^k,\Omega^k))$, where $\pi_k:\tilde X^k\rightarrow X^k$ is the covering map and $\pi_k^*(\cdot)$ here means the local lift to $T\tilde X^k$ of the family of (oriented, we will suppress orientation suffixes in the following) Lagrangian planes 
\[
\Lambda^0=\mathbb{R}^{n+1}\times\{0\}\subset T\tilde X\subset T\mathbb{C}^{n+1}
\] 
which span the kernel of the differential of $\pi_0:\mathbb{C}^{n+1} \rightarrow \{0\}\times \mathbb{R}^{n+1}$, the latter being the projection along $\mathbb{R}^{n+1}\times\{0\}$. Recall that $i_{\tau,t}:Q_{\tau,x(t)}\hookrightarrow (Y^k_\epsilon)_{x(t)}\subset X_e^k,\ \tau,t \in [0,1]$ is the inclusion of the respective Lagrangian spheres in the Milnor fibres $(Y^k_\epsilon)_{x(t)}$ as defined in (\ref{qfamily}). Let $\Lambda_{Q_{\tau}}, \Lambda_{s^k, \tau}\in \Gamma(i_{\tau}^*{\rm Lag}(\tilde X^k,\Omega^k))$ be defined by 
\[
\begin{split}
\Lambda_{Q_{\tau}}\circ i_\tau^{-1}|Q_{\tau,x(t)}&=\Lambda_{Q_{\tau,x(t)}}, \ \tau, t \in [0,1],\\
\Lambda_{s^k, \tau}\circ i_\tau^{-1}|Q_{\tau,x(t)}&= \Lambda_{s^k, \tau,t}=\pi_k^*\Lambda^0\circ i_\tau |Q_{\tau,x(t)}, \tau, t \in [0,1].
\end{split}
\]
Now set $(M,\Omega)=(\tilde X^k,\Omega^k)$ and consider the following families of $3$-tuples of the form $\mathcal{G}(i_t:N\rightarrow M, \Lambda^t_0,\Lambda^t_1)_t, \ t \in D$, where $D$ is a compact indexing manifold, $\Lambda_0,\Lambda_1\in  i^*{\rm Lag}(M,\Omega)$ and $i:N\rightarrow M$ is either an immersion or an embedding (remark that $N, i, \Lambda_0, \Lambda_1$ depend (smoothly) on $t\in D$ in the following):
\begin{equation}
\begin{split}
\mathcal{G}_1(i:N\rightarrow M,\Lambda_0,\Lambda_1)_{\tau}&= (i_\tau: Q_x\times[0,1]\rightarrow Q_\tau,   \Lambda_{Q_\tau}, \Lambda_{s^k,\tau}),\ \tau \in D_1=[0,1]\\
\mathcal{G}_2(i:N \hookrightarrow M,\Lambda_0,\Lambda_1)_{\tau,t} &=(i_{\tau,t}:Q_{\tau,x(t)}\hookrightarrow M, \Lambda_{Q_{\tau,x(t)}}, \Lambda_{s^k,\tau,t}), \ (\tau,t) \in D_2=[0,1]\times\{0,1\}.
\end{split}
\end{equation}
Adopting to the above the notation and construction of the proof of Lemma \ref{lag5} recall that there, for any $x \in N$, we defined $\mathcal{M}_x\subset i^*{\rm Lag}(M,\omega)_x$, a canonically cooriented cycle of codimension one, that is, $\mathcal{M}_x$ as a real algebraic subvariety with singular set
\begin{equation}\label{stratm}
\overline {\mathcal{M}^2_x}=\bigcup_{k \geq 2}\mathcal{M}^k_x, \ \mathcal{M}^k_x=\{\Lambda(x) \in i^*{\rm Lag}(M,\omega)_x: {\rm dim}(\Lambda(x)\cap \Lambda_1(x))=k\},\ k \in \mathbb{N}^+,
\end{equation}
where ${\rm codim}(\mathcal{M}^k_x)=\frac{1}{2}k(k+1)$ and a stratification  $\{\mathcal{M}^k_x\}_{k \in \mathbb{N}^+}$ of $\mathcal{M}_x$ with smooth top-stratum $\mathcal{M}^1_x$ of codimension $1$ and singular set $\overline {\mathcal{M}^2_x}$ of  codimension $3$ in $i^*{\rm Lag}(M,\omega)_x$. Recall $\mathcal{M}^k=\bigcup_{x\in N} \mathcal{M}^k_x, \ k\in \mathbb{N}^+$. We then have the following:
\begin{lemma}\label{genericity}
Consider for $i=1,2$ $\mathcal{G}_i(i_{\mathfrak{t}}:N_{\mathfrak{t}}\rightarrow M,\Lambda^{\mathfrak{t}}_0,\Lambda^{\mathfrak{t}}_1)$ for $\mathfrak{t} \in D_i$ as defined above. Then we have
\begin{enumerate}
\item For $i=1$ and $\mathfrak{t}=\tau \in D_1$, there is a a section $\tilde \Lambda^\tau_0 \in \Gamma(i_\tau^*{\rm Lag}(M,\Omega))$ arbitrarily close to $\Lambda^\tau_0$ in the (Whitney)  $C^\infty$-topology so that $\tilde \Lambda^\tau_0$ is transversal to $\mathcal{M}^k \subset i_{\tau}^*{\rm Lag}(M,\Omega)$ as defined in (\ref{stratm}) for each $k \geq 1$ outside of a discrete set of points in $N_\tau$. 
\item For $i=2$ and $\mathfrak{t}=(\tau,t) \in D_2$ there is a family of Lagrangian embeddings $\hat i_{\mathfrak{t}}:N_{\mathfrak{t}}\hookrightarrow (Y^k_\epsilon)_{x(t)}, t \in\{0,1\}, \tau\in[0,1]$ whose images are Lagrangian spheres in $(Y^k_\epsilon)_{x(t)}$,  any $\hat i_{\mathfrak{t}}$ is arbitrarily close to $i_{\mathfrak{t}}$ in the (Whitney) $C^\infty(N, M)$-topology and the sections $\hat \Lambda_0^{\mathfrak{t}} \in \Gamma(\hat i_{\mathfrak{t}}^*{\rm Lag}(M,\Omega))$ being induced along the images of $\hat i_{\mathfrak{t}}$ for each $\mathfrak{t}\in D_{2}$ as in (\ref{lagsubspaces57}) (after eventually slightly rotating the horizontal direction) are transversal to $\mathcal{M}^k \subset \hat i_{\mathfrak{t}}^*{\rm Lag}(M,\Omega)$ outside of a discrete set of points in $N_\mathfrak{t}$. 
\item By modifying $i_\tau, \ \tau \in [0,1]$ and the family of horizontal subspaces given in (\ref{lagsubspaces57}) in a neighborhood of the boundary of $N_\tau$ by an arbitrarily small amount in the $C^\infty$-topology, we can arrange that $\tilde \Lambda^\tau_0, \tau\in [0,1]$ defined in (1.) and the pairs $(\hat i_{\mathfrak{t}}, \hat \Lambda_0^{\mathfrak{t}}), \mathfrak{t}\in D_2$ defined in (2.) are compatible in the sense that $\tilde \Lambda_0^{\tau}|N_{\mathfrak{t}}= \hat \Lambda_0^{\mathfrak{t}}$ for $\mathfrak{t}=(\tau,t) \in D_2$.
\end{enumerate}
\end {lemma}
\begin{proof}
We adopt the proof of Theorem 2.1 in Arnol'd  (\cite{arnold}) to our situation and use a result of Bruce \cite{bruce} to extend the result to the situation of families of mappings.
Let $\mathcal{G}_i,\ i=1,2$ as above. First note that ${\rm Lag}(M,\Omega)\simeq M\times {\rm Lag}(\mathbb{R}^{2n+2},\Omega_0)$, where $\Omega_0$ is the symplectic standard structure since ${\rm Lag}(M,\Omega)$ has a global section (simply take $\pi_k^*(\Lambda^0)$). So for $\mathfrak{t}\in D_{1,2}$, we have that $i_{\mathfrak{t}}^*{\rm Lag}(M,\Omega)$ is diffeomorphic to $N_{\mathfrak{t}}\times {\rm Lag}(\mathbb{R}^{2n+2},\Omega_0)=:N_\mathfrak{t}\times B$. We will denote the projection of the image of $\mathcal{M}^k$ in $N_\mathfrak{t}\times B$ to the second factor under that diffeomorphism by $C(k)$. Note that the group $U(n+1)$ is acting transitively on $B$. By the above identifications we can understand $\Lambda_0^{\mathfrak{t}} \in \Gamma(\tilde i_{\mathfrak{t}}^*{\rm Lag}(M,\Omega))$ as a map $\Lambda_0^{\mathfrak{t}}: N_{\mathfrak{t}} \rightarrow B$, or shortly $f(\mathfrak{t}):N_{\mathfrak{t}}\rightarrow B$, where the group $G=U(n+1)$ is acting smoothly and transitively on the manifold $B$ and we are given a smooth compact submanifold $C(k)\subset B$ for each $k\geq 1$. Note further that as remarked in the proof of Proposition \ref{keylemma}, there is a family of diffeomorphisms (in fact, fibrewise symplectomorphisms) $\Phi_{\mathfrak{t}, \mathfrak{t}'}: N_{\mathfrak{t}}\rightarrow N_{\mathfrak{t}'}$ for all $\mathfrak{t}, \mathfrak{t}' \in D_{1,2}$, so since $i_{\mathfrak{t}'}=i_{\mathfrak{t}}\circ \Phi_{\mathfrak{t}, \mathfrak{t}'}^{-1}$ we will assume in the following that $N$ does not depend on $\mathfrak{t}$. Ignoring for a moment the parameter $\mathfrak{t}$, we are thus in the situation of Lemma 4.1.3 in \cite{arnold}. Extending the result in loc. cit. we claim that, for each $\mathfrak{t}\in D_{1,2}$ and a discrete set $S(\mathfrak{t})\subset N_{\mathfrak{t}}$, the measure of the points $g_{\mathfrak{t}}\in G$ where the mapping
\[
f_{g,\mathfrak{t}}:N_{\mathfrak{t}}\rightarrow B, \  f_{g,\mathfrak{t}}(x)=g_{\mathfrak{t}}f(\mathfrak{t})(x),
\]
restricted to $N_{\mathfrak{t}}\setminus S(\mathfrak{t})$, is not transversal to $C(k)$, is zero. But this follows easily from an inspection of Arnold's proof and Theorem 1.1 in \cite{bruce} which shows that for a residual set of smooth mappings $F:A\times U\rightarrow B$ one has that $F_u:A\rightarrow B, \  u \in U$ is transversal to a given submanifold $C\subset B$ except on a discrete set of points which appear as isolated singular points in $F_u^{-1}(C)$, i.e. $F_u^{-1}(C)$ is smooth outside of isolated points for each $u\in U$. \\
Summarizing, we can find for each $\mathfrak{t}$ an element $g_{\mathfrak{t}}\in G=U(n+1)$ arbitrarily close to the identity so that $\tilde\Lambda_0^{\mathfrak{t}} :=g_{\mathfrak{t}}\Lambda_0^{\mathfrak{t}}$ is transversal to $\mathcal{M}^k\subset i_{\mathfrak{t}}^*{\rm Lag}(M,\Omega)$ outside a discrete set, which already gives the Claim (1.) for the case $i=1$ and  $\mathfrak{t}\in D_1$.\\
To prove Claim (2.), note that we can extend each $N_{\mathfrak{t}}:=i_{\mathfrak{t}}(N), \ \mathfrak{t}\in D_2$ to an $n+1$-dimensional Lagrangian manifold with boundary $N^e_{\mathfrak{t}}$ by using symplectic parallel transport along small arcs $t\mapsto xe^{2\pi it}, t \in (t_0 -\epsilon,t_0 +\epsilon), t_0 \in \{0,1\}$, $\epsilon > 0$ small. Then each $i^e_{\mathfrak{t}}: N^e_{\mathfrak{t}}\hookrightarrow  M$ carries a section $\Lambda_0^{\mathfrak{t}} \in \Gamma((i^e_{\mathfrak{t}})^*{\rm Lag}(M,\Omega))$ as constructed in (\ref{lagsubspaces57}) and by the above, there is for each $\mathfrak{t}\in D_2$ an element $\hat g_{\mathfrak{t}}\in G=U(n+1)$ arbitrarily close to the identity so that $\hat \Lambda_0^{\mathfrak{t}}:= \hat g_{\mathfrak{t}}\Lambda_0^{\mathfrak{t}}\circ i_{\mathfrak{t}}$ is transversal to $\mathcal{M}^k\subset i_{\mathfrak{t}}^*{\rm Lag}(M,\Omega)$ outside a discrete set. But $\hat \Lambda_0^{\mathfrak{t}}$ is just the restriction of the section of $(i^e_{\mathfrak{t}})^*{\rm Lag}(M,\Omega))$ induced by the tangent mapping of $\hat g_{\mathfrak{t}}N^e_{\mathfrak{t}}$ to $N_{\mathfrak{t}}$ and so intersecting $\hat g_{\mathfrak{t}}N^e_{\mathfrak{t}}$ with $(Y^k_\epsilon)_{x(0)}$ giving a family $\tilde N_{\mathfrak{t}}$ we arrive at Claim (2.).\\
Finally Claim (3.) is proven by first extending as in Claim (2.) each $i_{(\tau,t)}:N_{\tau,t}:=Q_{\tau,x(t)}\hookrightarrow (Y^k_\epsilon)_{x(t)}$ for $\tau \in [0,1]$ and $t \in I_\epsilon:=[0,\epsilon]\cup [1-\epsilon, 1]$ for some small $\epsilon>0$ to a Lagrangian submanifold (with boundary) $N^e_{(\tau,t)}$ in $M$ and as in the proof of (2.) we get sections $\Lambda_0^{\mathfrak{t}} \in \Gamma((i^e_{\mathfrak{t}})^*{\rm Lag}(M,\Omega))$ induced by (\ref{lagsubspaces57}) along each $i^e_{\mathfrak{t}}:N^e_{(\tau,t)}\hookrightarrow M$. Then note that we can find smooth paths $\tilde g_0(\tau,t), \tilde g_1(\tau,t), \ t \in I_\epsilon= [0,\epsilon]\cup [1,1-\epsilon], \tau \in [0,1]$ supported in a small neighbourhood of $1 \in U(n+1)$ so that if $g_\tau,\tau \in [0,1]$ is constructed in the proof of (1.) and $\hat g_{\tau,t}, (\tau,t)\in [0,1]\times \{0,1\}$ is as constructed in the proof of (2.) we have
\[
\begin{split}
&\tilde g_0(\tau,1)=\tilde g_0(\tau,0)=Id_{U(n)},\quad \tilde g_0(\tau,\epsilon)=\tilde g_0(\tau,1-\epsilon)=g_\tau,\\
&\tilde g_1(\tau,1)=\hat g_{\tau,1}, \quad \tilde g_1(\tau,0)=\hat g_{\tau,0},\quad \tilde g_1(\tau,\epsilon)=\tilde g_0(\tau,1-\epsilon)=Id_{U(n)} 
\end{split}
\]
and so that $\tilde g_{0}\circ \tilde g_{1}(\tau,t)(\Lambda_0^{(\tau,t)})\circ i_{(\tau,t)}$ are transversal to $\mathcal{M}^k\subset i_{(\tau,t)}^*{\rm Lag}(M,\Omega), k\geq 1$ outside a discrete set in $N_{(\tau,t)}$ for any $(\tau,t)\in [0,1]\times I_\epsilon$. Then replacing $N^e_{\tau,t}$ by $\tilde N^e_{\tau,t}=\tilde g_1(\tau,t) N^e_{\tau,t} \cap(Y^k_\epsilon)_{x(t)}$ for $(\tau,t)\in [0,1]\times I_\epsilon$ and $\tilde\Lambda_0^{\mathfrak{t}}$ constructed in the proof of (1.) by $\tilde g_0(\tau,t)\Lambda_0^{\mathfrak{t}}\circ \tilde g^{-1}_1(\tau,t)$ along the family $\tilde N^e_{\tau,t}\cap (Y^k_\epsilon)_{x(t)}=: \tilde N_{\tau,t}$ and substituting the latter for $N_{\tau,t},\ \tau \in [0,1],  t \in I_\epsilon$, we arrive at Claim (3.).
\end{proof}
Note that by Definition \ref{generic}, this proves the genericity part of Assumption \ref{ass4} (1.) modulo the fact that in a small collar neighbourhood of the (immersed) boundary of each $N_\tau=Q_\tau$ the above proof allows for the presence of a 'non transversality'-set which is discrete in each of the fibres $\tilde N_{\tau}\cap (Y^k_\epsilon)_{x(t)}, t \in I_\epsilon$, using the notation of the proof, while being discrete in $Q_\tau$ outside the collar nghbd. Since that will cause no trouble for our purposes, we will ignore in the following the fact only having proven a slightly weaker result than stated. Note further that the existence part of Assumption \ref{ass4} (1.) will follow for all $\tau \in [0,1]$ by the considerations below, while for $\tau=0,1$ it follows from the remark given above Proposition \ref{keylemma}.\\

The following considerations will be devoted to a proof of part (2.) and a discussion of (3.) of Assumption \ref{ass4}. The latter part will pe proven given the absence of certain 'higher singularities' and its relevance for a proof of Conjecture \ref{conjspec} will be discussed briefly in Section \ref{generalspec}.\\
Let now $N_{\mathfrak{t}}=Q_{\tau,x(t)}=Q_\tau\cap (Y^k_\epsilon)_{x(t)}\subset \tilde X^k, \mathfrak{t}=(\tau, t) \in D_2=[0,1]\times\{0,1\}$ be as above and recall the Lagrangian submanifold with boundary $N^e_{\mathfrak{t}}\subset \tilde X^k$ constructed in the proof of (2.) in the above Lemma. By arguments similar to the above, we can assume that the tangent section of $N^e_{\mathfrak{t}}$, restricted to $N_{(\tau,t)}, t \in \{0,1\}$, induces a section of $(i_{(\tau,t)})^*{\rm Lag}(M, \Omega))$, where $i_{(\tau,t)}:N_{\mathfrak{t}}\hookrightarrow (Y^k_\epsilon)_{x(t)}\subset \tilde X^k=M$ is the inclusion, that is transversal to all submanifolds $\mathcal{M}^k\subset i_{(\tau,t)}^*{\rm Lag}(M,\Omega)$ for all $t \in \{0,1\}$. By choosing $\epsilon >0$ small enough in the definition of $N^e_{(\tau,t)}, t \in \{0,1\}$, we can further assume that the tangent section to $N^e_{(\tau,t)}$ induces sections $\hat \Lambda_0^{(\tau,t)} \in i_{(\tau,t)}^*{\rm Lag}(M,\Omega)$ which have the same transversality property outside of a discrete set in each $N_{(\tau,t)}$ for all $(\tau,t) \in [0,1]\times I_\epsilon$ where here, $i_{(\tau,t)}:N_{(\tau,t)}=Q_{\tau,x(t)}\hookrightarrow (Y^k_\epsilon)_{x(t)},\ (\tau,t) \in [0,1]\times I_\epsilon$. In the following we will set $L_{\tau}:=N^e_{(\tau,t)} \subset M,  t \in \{0,1\}$ for $\tau \in [0,1]$ and $\epsilon >0$ small enough in the above sense (the case $t=1$ will be of course the interesting case). Since the family $L_{\tau}, \tau \in [0,1]$ is contained in the segment $Y^k_{[-\epsilon, \epsilon]}:=\bigcup_{t \in I_\epsilon}(Y^k_\epsilon)_{x(t)}\subset X_e^k$ we can find a $\delta>0$ and an open ball $B^{2n+2}_\delta\subset \mathbb{C}^{n+1}$ centered at the origin so that $\bigcup_{\tau\in [0,1]}L_{\tau}\subset B^{2n+2}_\delta$. Let now $\pi_0:\mathbb{C}^n \rightarrow \{0\}\times \mathbb{R}^{n+1}$ be the projection introduced above and consider $B^{n+1}_\delta=\pi_0(B^{2n+2}_\delta)\subset \mathbb{R}^{n+1}$, then we have a submersion $\pi_0: B^{2n+2}_\delta\rightarrow B^{n+1}_\delta$, for which we will write shortly $\pi_0: B_2\rightarrow B_1$ in the following. Note that $\pi_0$ is a fibration with symplectic total space $(B_2, \Omega=\Omega^k|B_2)$ and Lagrangian leaves integral to ${\rm ker}(\pi_0)$. For each $\tau \in [0,1]$, consider now the smooth map
\[
\alpha_\tau: L_{\tau}\rightarrow B_1, \ \alpha_\tau=\pi_0\circ j_\tau,
\]
where $j_\tau: L_{\tau}\hookrightarrow B_2$ is the inclusion. In the following, we will consider the closures of $L_{\tau} \subset B_2, \ \tau \in [0,1]$ as a family of Lagrangian submanifolds with isotropic boundary in the symplectic manifold $(B_2, \Omega)$ whose intersection $L_{\tau,t}=Q_{\tau,x(t)}$ with each Milnor fibre $(Y^k_\epsilon)_{x(t)},\ t \in I_\epsilon$ is a Lagrangian sphere w.r.t. the restricted form $\Omega|Y^k_\epsilon)_{x(t)}$. We are now interested in representing a suitable modification of each $L_{\tau}\subset B_2$ by means of 'generating families' (see \cite{eliasgr}). For this consider $\beta \in \Omega^1(B_2)$ given by $\beta=pdq$, where $(p,q)$ are the usual canonical coordinates on $B_2 \subset \mathbb{R}^{2n+2}$. Then $d(j_\tau^*\beta)=0$ for each $\tau \in [0,1]$ since the $L_{\tau}$ are Lagrangian. Thus since $H^1(L_\tau, \mathbb{C})=0$, the functions $f_\tau:L_{\tau}\rightarrow \mathbb{R}$ given by $f_\tau(z)=\int_{\gamma_\tau:z_0, z}j_\tau^*\beta$ are well defined for each $\tau \in [0,1]$ up to a constant where we integrate along paths $\gamma_\tau \subset L_{\tau}$ connecting a fixed $z_0\in \mathcal{L}_{\tau}$ to $z \in L_{\tau}$ for each $\tau \in [0,1]$ (and we assume that $z_0$ varies smoothly in $\tau$). Then we have the following definition, which is essentially the construction of subgraphical varieties (re)introduced by Gromov and Eliashberg (\cite{eliasgr}) to symplectic topology.
\begin{Def}\label{subgraphical}
With $f_\tau:L_{\tau}\rightarrow \mathbb{R}$ given as above, define the Lagrangian graph $\tilde L_{\tau}=df_\tau \in \Gamma(T^*L_{\tau})$ for each $\tau \in [0,1]$. Let $(\alpha_{\tau})_*:TL_{\tau}\rightarrow TB_1$ be the induced map and consider the set
\[
\mathcal{L}^C_{\tau}:=({\rm ker} (\alpha_{\tau})_*)^\perp \cap \tilde L_{\tau} \subset T^*L_{\tau},\ \tau \in [0,1]
\]
where $({\rm ker} (\alpha_{\tau})_*)^\perp\subset T^*L_{\tau}$ refers to the set $\{\omega\in T^*L_{\tau}\ |\ \omega({\rm ker} (\alpha_{\tau})_*)=0\}$. Consider the set  
$\mathcal{A}_{\alpha_t}\subset T^*(L_\tau)\oplus \alpha_\tau^*(T^*B_1)$ consisting of pairs $(\kappa_y, \eta_x)\in ({\rm ker} (\alpha_{\tau})_*)_y^\perp\oplus \alpha_\tau^*(T_x^*B_1)$ with $y \in L_{\tau},\ x \in B_1$ so that $\alpha_\tau(y)=x$ and $\eta_x\circ D_y\alpha_\tau=\kappa_y \in T_y^*(L_{\tau})$. Finally define $\underline L_\tau \subset T^*B_1$ by setting 
\[
\underline L_\tau:= {\rm pr}_2\left((\mathcal{L}^C_{\tau}\oplus \alpha_\tau^*(T^*B_1))\cap \mathcal{A}_{\alpha_t}\right),
\]
where ${\rm pr}_2:\mathcal{A}_{\alpha_t}\rightarrow T^*B_1$ is the projection onto the second factor. Understanding this, we will call $\underline L_\tau\subset T^*B_1$ the {\it subgraphical image} of $\tilde L_{\tau}\subset T^*L_{\tau}$.
\end{Def}
Of course, since $B_2\subset \mathbb{R}^{2n+2}$ can be considered as a subset $B_2 \subset T^*B_1$, we can compare our original $L_{\tau}\subset B_2$ and $\underline L_\tau\subset T^*B_1$ constructed above, that is, we have the following.
\begin{lemma}
Assume that $L_{\tau}\subset B_2$ is generic in the sense of Definition \ref{generic} resp. Lemma \ref{genericity}. Consider the set $\underline L^{reg}_{\tau}\subset  \underline L_\tau\subset T^*B_1$ being defined by the restriction of $\alpha_\tau:L_\tau \rightarrow B_1$ and $f_\tau: L_\tau \rightarrow \mathbb{R}$ to the set of points of $L_\tau$ where $(\alpha_\tau)_*$ has maximal rank and subsequent application of the procedure in Definition \ref{subgraphical}. Considering the inclusion $B_2\subset T^*B_1$ we have that the closure of $\underline L^{reg}_{\tau}\subset T^*B_1$ equals $L_\tau\subset B_2$ for all $\tau \in [0,1]$.
\end{lemma}
\begin{proof}
The assertion follows from the fact that if $D_y\alpha$ has maximal rank for some $y \in L_\tau$, then $({\rm ker} (\alpha_{\tau})_*)_y^\perp=T_y^*L_\tau$, hence the set of pairs $(\kappa_y, \eta_x) \in (\mathcal{A}_{\alpha_t})_y$ is isomorphic to $T_y^*L_\tau$ and by the definition of $\tilde L_\tau$, $\eta_x=\kappa_x\circ (D_y\alpha)^{-1} \in L_t\cap T^*_xB_1$.
Then since $L_{\tau}$ was assumed to be generic, its singular set is of codimension $\geq 1$ in $L_\tau$, hence the closure of $\underline L^{reg}_{\tau}$ equals $L_\tau$.
\end{proof}
As announced above, we can now represent $\underline L_\tau$ for each $\tau \in [0,1]$ by 'generating families' in the following sense.
\begin{lemma}\label{genfunction}
For each $\tau \in [0,1]$ and with the notation introduced above, define a function
$\tilde f_\tau: L_\tau\times B_1\times \mathbb{R}^{n+1}\rightarrow \mathbb{R}$ by setting
\begin{equation}\label{genfamily}
\tilde f_\tau(u,y,z)=f_\tau(u)+(y-\alpha_\tau(u), z), \ (u,y,z) \in L_\tau\times B_1\times \mathbb{R}^{n+1}.
\end{equation}
Then we have that $\underline L_\tau\subset T^*B_1, \ \tau \in [0,1]$ is given by
\[
\underline L_\tau=\{(y,y^*) \in T^*B_1\ |\ {\rm there\ exist\ }(u,y,z) \in C_{\tilde f_\tau},  y^*=D_y \tilde f(u,y,z) \in T^*_yB_1\},
\]
where the {\it fibre critical set} $C_{\tilde f_\tau}\subset L_\tau\times B_1\times \mathbb{R}^{n+1}$ is defined as
\[
C_{\tilde f_\tau}=\{(u,y, z) \in L_\tau \times B_1 \times \mathbb{R}^{n+1}\ {\rm s.t.}\ D_{u,z}\tilde f_\tau(u,y,z)=0\}
\]
Furthermore, when considering the smooth fibration $\tilde \alpha_\tau: L_\tau\times B_1\times \mathbb{R}^{n+1}\rightarrow B_1$ given by $(u,y,z) \mapsto y$ the fibre-Hessian $D_{(u,y,z)}^2\tilde f_\tau$ of $\tilde f_\tau$, when restricted to a given intersection point $(u,y,z) \in \tilde \alpha_\tau^{-1}(y) \cap C_{\tilde f_\tau}$ is non-degenerate if and only if the kernel of $(\alpha_\tau)_*$ at $u \in L_\tau$ is zero.
\end{lemma}
{\it Remark.} In especially we have that $\tilde f_\tau, \ \tau \in [0,1]$, restricted to a fibre $\tilde \alpha_\tau^{-1}(y)$ is a Morse function iff $(\alpha_\tau)_*$ is non-singular at all points $u \in L_\tau$ which project to $y \in B_1$ under $\alpha_\tau$. Furthermore, if $L_\tau$ is generic in the sense of Lemma \ref{genericity}, the points $(u,y,z)$ where the fibre Hessian of $\tilde f_\tau$ is degenerate occur with codimension $\geq 1$ in the set of fibre-critical points $C_{\tilde f_\tau}$ of $\tilde f_\tau$ in $L_\tau\times B_1\times \mathbb{R}^{n+1}=:Z_\tau$. Note further that a more invariant definition of $\underline L_\tau$ of course would be to set $\hat L_\tau=d \tilde f_\tau \subset T^*(L_\tau\times B_1\times \mathbb{R}^{n+1})=T^*Z_\tau$ and consider the set $\hat L_\tau \cap H_\tau$, where $H= ({\rm ker} (\tilde \alpha_\tau)_*)^\perp\subset T^*Z_\tau$. Then if $\pi_{Z_\tau}:T^*Z_\tau \rightarrow Z_\tau$ is the canonical projection, we have $C_{\tilde f_\tau}=\pi_{Z_\tau}(\hat L_\tau \cap H_\tau)$ and there is an immersion $\hat L_\tau \cap H_\tau\rightarrow \underline L_\tau$ (see e.g. \cite{guille}), note that of course $\underline L^{reg}_{\tau}$ is a submanifold of $T^*B_1$. We summarize the discussion by the commutative diagram
\begin{equation}\label{generatingfam}
\begin{CD}
C_{\tilde f_\tau}\subset Z_\tau  @<<\pi_{Z_\tau}< \hat L_\tau \cap H_\tau \subset T^*Z_\tau \\
 @VV\tilde \alpha_\tau V @VVV \\
B_1 @<< \pi_{B_1}< \underline L_\tau\subset T^*B_1.\\
\end{CD}
\end{equation}
Note that from the diagram we have (cf. \cite{guille}) that if $\tilde w \in \hat L_\tau \cap H_\tau$, then ${\rm dim}\ {\rm ker}(\tilde \alpha_\tau)_*\cap TC_{\tilde f_\tau}|_{\pi_{Z_\tau}(\tilde w)}={\rm dim}\ {\rm ker}(\pi_{B_1})_*\cap T\underline L_\tau|_w$ if $w$ is the image of $\tilde w$ under the immersion $\hat L_\tau \cap H_\tau\rightarrow \underline L_\tau$. On the other hand, in analogy to the proof below, we have that both dimensions coincide with the dimension of the kernel of the fibre-hessian of $\tilde f_\tau$ at $\pi_{Z_\tau}(\tilde w)$. Note that by its implicit definition in Lemma \ref{genfunction}, the map $\hat L_\tau \cap H_\tau\rightarrow \underline L_\tau$ is in fact an {\it injective} immersion on the set of points $\tilde w \in \hat L_\tau$ so that $(\tilde \alpha_\tau)_*\cap TC_{\tilde f_\tau}|_{\pi_{Z_\tau}(\tilde w)}$ has maximal rank, thus regular points of $\tilde \alpha_\tau$ in $C_{\tilde f_\tau}$ are in smooth bijection to regular points of $\pi_{B_1}$ in $\underline L_\tau$.
\begin{proof}
Let $y \in B_1$ and let $(u,y,z)$ be an element of the fibre-critical set $C_{\tilde f_\tau}$ of $\tilde f_\tau$ over $y$. Then we have for all $\tau \in [0,1]$
\begin{equation}\label{familyeq}
\begin{split}
D_u\tilde f_\tau(u,y,z)&=0= Df_\tau(u)- z^\perp\circ D\alpha_\tau(u)\\
D_z\tilde f_\tau(u,y,z)&=0= y -\alpha(u),
\end{split}
\end{equation}
where here, $z^\perp \in T^*B_1$ denotes the linear form on $\mathbb{R}^{n+1}$ vanishing on all vectors perpendicular to $z$. Then by definition we have
\[
y^* =D_y\tilde f_\tau(u,y,z)=z^\perp.
\]
Now by (\ref{familyeq}), $z^\perp$ satisfies exactly the conditions on the points of $\underline L_\tau\subset T^*B_1$ formulated in Definition \ref{subgraphical}. To prove the second assertion of the Lemma just consider the fibre Hessian of $\tilde f_\tau$ at some fibre critical point $(u,y,z)$ of $\tilde f_\tau$ over $y \in T^*B_1$:
\[
D^2_{(u,y,z)} \tilde f_\tau= \begin{pmatrix}  D^2f(u)-(z^\perp, D^2\alpha_\tau(u))& -D\alpha_\tau(u)\\
-D\alpha_\tau(u) & 0 \end{pmatrix}
\]
Using this we see that for a fibre critical point $(u,y,z)$ the rank of $D^2_{(u,y,z)} \tilde f_\tau$ is maximal on each subspace $V \subset T_{(u,z)}(L_\tau\times \{y\}\times \mathbb{R}^{n+1})$ on which $D\alpha_\tau(u)\circ ({\rm pr}_1)_*$ is maximal, where ${\rm pr}_1:L_\tau\times \{y\}\times \mathbb{R}^{n+1}\rightarrow L_\tau$ is the projection, so we arrive at the second assertion.
\end{proof}
To proceed, we claim that there is a family of diffeomorphisms $\Psi_\tau:L_0\rightarrow L_\tau, \ \tau \in [0,1]$. To define these, recall that $L_{\tau,0}=Q_{\tau,x(0)}= \Phi_\tau(Q_x):=\Phi_{H(x)}(1-\tau)\circ\rho_{\tau}(Q_x) \subset (Y^k_\epsilon)_{x(0)}$ for $\tau \in [0,1], t \in I_\epsilon$. On the other hand we have for each $\tau \in [0,1]$ $L_{\tau,t}=\mathcal{P}^{\Omega}_{c(t)}(L_{\tau,0}), \ t \in I_\epsilon$, where $c(t)=x(0)e^{2 \pi i t}$ and $\mathcal{P}^{\Omega}_{c(\cdot)}$ denotes symplectic parallel transport along $c$. Hence we set for each $(\tau, t) \in [0,1]\times I_\epsilon$
\[
\Psi_\tau|L_{0,t}=\mathcal{P}^{\Omega}_{c(t)}\circ \Phi_\tau \circ (\mathcal{P}^{\Omega}_{c(t)})^{-1}|L_{0,t},
\]
and observe that these maps assemble to a diffeomorphism $\Psi_\tau:L_0\rightarrow L_\tau,\  \tau \in [0,1]$. Using this family , we observe that we can replace the generating family $\tilde f_\tau, \ \tau \in [0,1]$ for $\underline L_\tau\subset T^*B_1$ given by Lemma \ref{genfunction} by a family of generating functions $\hat f_\tau: L_0\times B_1\times \mathbb{R}^{n+1}\rightarrow \mathbb{R}$ defined on the {\it same} manifold by setting
\begin{equation}\label{genfamily2}
\hat f_\tau(u,y,z)=(f_\tau\circ \Psi_\tau)(u)+(y-(\alpha_\tau\circ \Psi_\tau)(u), z), \ (u,y,z) \in L_0\times B_1\times \mathbb{R}^{n+1}.
\end{equation}
Then all the assertions of Lemma \ref{genfunction} and the remark below that Lemma remain valid when replacing $\tilde f_\tau$ by $\hat f_\tau$, $\alpha_\tau$ by $\hat \alpha_\tau:=\alpha_\tau\circ \Psi_\tau:L_0\rightarrow B_1$ and $\tilde \alpha_\tau$ by $\tilde \alpha_0=\tilde \alpha_\tau(\Psi_\tau(\cdot), \cdot, \cdot):Z_0=L_0\times B_1\times \mathbb{R}^{n+1}\rightarrow B_1$. Of course, the fibre critical set $C_{\hat f_\tau}$ corresponding to $\hat f_\tau$ then is a subset of $Z_0$ for each $t \in [0,1]$. Thus we see that the family $\underline L_\tau\subset T^*B_1$ is given by the generating family $(\hat f_\tau, \tilde \alpha_0)_{\tau \in [0,1]}$ on $Z_0$. We will in the following (for the proof of Proposition \ref{keysing} 1.-3. below) assume that the family $\hat f_\tau$ is modified so that $\hat f_\tau$ is constant in $\tau$ in a neighbourhood $W$ of the set $\partial L_0\times B_1\times \mathbb{R}^{n+1}\subset Z_0$ and coincides with the original $\hat f_\tau$ in a nghbd $V$ of $L_{0,0}\times B_1\times \mathbb{R}^{n+1}$. This can be achieved for instance by defining a smooth function $\rho:Z_0\rightarrow [0,1]$ by being identically to $1$ on $W$ and being identically zero on $V$ and constant on the slices $L_{0,t}\times B_1\times \mathbb{R}^{n+1}\subset Z_0$, for $t \in I_\epsilon$. Then set
\begin{equation}\label{modify}
\hat f^0_\tau= (1-\rho)\hat f_\tau+\rho\hat f_0,
\end{equation}
and mildy modify $\rho$ to assure genericity in the sense of Lemma \ref{genericity}. We will in the following denote $\hat f^0_\tau$ again by $\hat f_\tau$.
\\
We are now in a position to prove the second part of Assumption \ref{ass4}. From now on we assume that $L_{\tau}\subset B_2$ is generic in the sense of Definition \ref{generic} resp. Lemma \ref{genericity} outside of a finite set of points for each $\tau \in [0,1]$. Of course the top-strata ${N_{\tau,t}^{\mathcal{M},top}}\subset L_{\tau,0}=L_\tau\cap (Y^k_\epsilon)_{x(0)}, \  \tau \in [0,1]$ appearing in Assumption \ref{ass4} correspond in our setting to the points $S_1(L_\tau,0)$ in $L_{\tau,0}$ where $(\alpha_\tau)_*:TL_\tau \rightarrow TB_1$ has corank $1$ which are by the remark above in one-to-one correspondence to those points $(u,y,z) \in Z_{\tau,0}:=Z_\tau\cap L_{\tau,0}\times B_1\times \mathbb{R}^{n+1}$ where the fibre Hessian of $\tilde f_\tau$ has a kernel of dimension $1$ in the set of fibre-critical points $C_{\tilde f_\tau}$ of $\tilde f_\tau$ in $Z_\tau$. Then, shifting the problem to the family $(\hat f_\tau, \tilde \alpha_0)_{\tau \in [0,1]}$ on $Z_0$ using the family of diffeomorphisms $\Psi_\tau:L_0\rightarrow L_\tau$ as described above, it will suffice to prove the following proposition. For this, let $S_1(C_{\hat f_\tau})\subset C_{\hat f_\tau} \subset Z_0$ be defined for each $\tau \in [0,1]$ by the set of points in $C_{\hat f_\tau}$ where the corank of the fibre Hessian of $\hat f_\tau$ on $Z_0$ is equal to $1$. Then by the genericity of the family $L_\tau \subset B_2$, $S_1(C_{\hat f_\tau})\subset C_{\hat f_\tau}$ is a submanifold of codimension $1$. Furthermore, the set $S_1(C_{\hat f_\tau},0):=Z_{0,0}\cap S_1(C_{\hat f_\tau})\subset Z_{0,0}\cap C_{\hat f_\tau}=:C_{\hat f_\tau,0}$ is a submanifold of codimension $1$ in $C_{\hat f_\tau,0}$ for any $\tau \in [0,1]$ outside of at most finitely many points. Consider the complement $\mathcal{C}_\tau:=C_{\hat f_\tau,0}\setminus \overline {S_1(C_{\hat f_\tau},0)}\subset Z_{0,0}$ for each $\tau \in [0,1]$. Let $x \in S^1_\epsilon$ be the fixed base point chosen in the proof of Proposition \ref{keylemma}. Then since $(Y^k_\epsilon)_x\subset \mathbb{C}^{n+1}$, we can find a $a > 0$ so that if $\mathbf{1}_p=(0,\dots,0,1,\dots, 1)\in \mathbb{R}^{2n+2}$ we have that
\[
(Y^k_\epsilon)^a_x:=\{z+a\cdot\mathbf{1}_p\in \mathbb{C}^{n+1}|\ z \in (Y^k_\epsilon)_x\subset \mathbb{C}^{n+1}\}
\]
does not intersect the set $\mathbb{O}_q:=\mathbb{R}^{n+1}\times \{0\}$, that is the $q$- plane if $(q,p)$ are coordinates in $\mathbb{C}^{n+1}\simeq \mathbb{R}^{2n+2}$. It is clear we can choose $a$ so that if we define $(Y^k_\epsilon)^a_u$ for $u\in S^1_\epsilon$ near $x$ analogously to the above, then $(Y^k_\epsilon)^a_u$ will also be disjoint from $\mathbb{O}_q$. Also it follows that all of the above constructions that concern a small neighbourhood of the fibre $(Y^k_\epsilon)_x$, i.e. those of Definition \ref{subgraphical} and below are invariant under the translation above so that we will assume in the following that $(Y^k_\epsilon)^a_w$ is disjoint from $\mathbb{O}_q$ for $w$ near $x \in S^1_\epsilon$.
\begin{prop}\label{keysing}
Assume $(Y^k_\epsilon)^a_w$, $w\in S^1_\epsilon$ near $x$ satisfies the above. Then with the notation introduced above, there is a smooth family of Morse functions $\mathfrak{f}_\tau, \ \tau\in [0,1]$ on generic fibres $(Z_0)_{y(\tau)}, \ \in B_1$ of $(\tilde \alpha_0, Z_0)$ so that $\mathcal{C}_\tau$ has at least as many connected components as $\mathfrak{f}_\tau$ has critical points for each $\tau \in [0,1]$ and the latter is greater or equal to the number ${\rm stabMor}(\mathfrak{f}_\tau)_{dbd}$ (using the notation of \cite{eliasgr}, Chapter 1.4). Furthermore one has:
\begin{enumerate}
\item Assume first that the set of points on $C_{\hat f_\tau,0}$ where ${\rm ker}(\tilde \alpha_\tau)_*|C_{\hat f_\tau,0}\neq 0$ solely consists of 'fold'-points (see the proof below), for all $\tau \in [0,1]$. Then there is a {\it path-connected} subset $\mathcal{U}\subset \mathcal{C}:=\bigcup_{\tau\in [0,1]}\mathcal{C}_\tau\times \{\tau\}\subset Z_{0,0}\times[0,1]$ so that $\mathcal{U}\cap \mathcal{C}_\tau\times \{\tau\}=:\mathcal{U}_\tau\subset C_{\hat f_\tau,0}\times \{\tau\}$ is a smooth non-empty $n$-manifold and equals exactly one connected component of $\mathcal{C}_\tau$. 
\item In the general case, there is for any given $w_0\in \mathcal{C}_0$ a smooth embedded path $c:[0,1]\rightarrow \mathcal{W}:=\bigcup_{\tau \in[0,1]} C_{\hat f_\tau,0}\times\{\tau\}$ so that $c(0)=w_0=c(1) \in \mathcal{C}_0=\mathcal{C}_1$, $c(\tau)\in C_{\hat f_\tau,0}\times\{\tau\},\ \tau \in [0,1]$, $c$ intersects $S_1:=\bigcup_{\tau \in[0,1]}{S_1(C_{\hat f_\tau},0)}\times  \{\tau\}\subset Z_{0,0}\times[0,1]$ transversally and we have ${\rm im\ c}\cdot S_1=0$, where $\cdot$ means oriented intersection number and $S_1$ carries the orientation induced by the family of fibre-hessians $D_{(\cdot)}^2\tilde f_\tau$ along $S_1(C_{\hat f_\tau},0)\subset C_{\hat f_\tau,0}$.
\item The orientation induced by the family of fibre-hessians $D_{(\cdot)}^2\tilde f_\tau$ along $S_1(C_{\hat f_\tau},0)\subset C_{\hat f_\tau,0}$ on $S_1$ by (\ref{morseorient}) below coincides with the opposite of the Maslov coorientation on $S_1$ as defined by Definition \ref{generic} (cf. \cite{arnold}), (\ref{generatingfam}) and the triple $(\bigcup_{\tau \in [0,1]}L_\tau\times \{\tau\}, \bigcup_{\tau \in [0,1]}\Lambda_{Q_\tau}|L_\tau\oplus \mathbb{R},\bigcup_{\tau \in [0,1]}\Lambda_{s^k, \tau}|L_\tau\oplus i\mathbb{R})$ considered in $(M\times \mathbb{C}, \Omega^k\oplus\omega_{\mathbb{C}})$, where $\omega_{\mathbb{C}}$ is the canonical symplectic form on $\mathbb{R}^2\simeq \mathbb{C}$.
\end{enumerate}
\end{prop}
{\it Remark.} While the first instance above proves the third part of Assumption \ref{ass4} in the absence of 'higher singularities', it was indicated in the proof of Proposition \ref{keylemma} that the second and third part of the above prove  actually a slightly stronger version than Assumption \ref{ass4} (2.). Note further that in 3., to define a Maslov coorientation as in Definition \ref{generic}, strictly speaking a relative version of that definition resp. parts of the preceeding Lemma would be necessary, but since non-transversality occurs at most at a discrete set by Lemma \ref{genericity}, we refrain from formulating this. We will indicate in Section \ref{generalspec} how a modification of (1.) can be used as a key step in proving the quasihomogeneous part of Conjecture \ref{conjspec}.
\begin{proof}
Let $z \in \mathcal{C}_\tau$ for some $\tau \in [0,1]$ and $y=\tilde \alpha_0(z)$. To prove the very first claim, we will show that $\hat f_\tau|(Z_0)_y$, where $(Z_0)_y= \hat f_\tau^{-1}(y)=L_0\times \{y\}\times \mathbb{R}^{n+1}$ has at least ${\rm stabMor}(\mathfrak{f}_\tau)_{dbd}$ critical points lying on $(Z_{0,0})_y:= (Z_0)_y\cap Z_{0,0}$ if all elements of $(C_{\hat f_\tau})_y:=(Z_0)_y\cap C_{\hat f_\tau}$ are non-degenerate critical points (such $y\in B_1$ exist by Lemma \ref{genericity}). For this, we will modify the function $\hat f_\tau|(Z_0)_y$ outside of a small neighbourhood of $(C_{\hat f_\tau,0})_y:=(Z_{0,0})_y\cap C_{\hat f_\tau}$, so that results of \cite{eliasgr} on 'stable' Morse theory become applicable. Using the notation of (\ref{genfamily2}), set $f^0_\tau:=f_\tau\circ \Psi_\tau:L_0\rightarrow \mathbb{R}$ and note that since $(Y^k_\epsilon)_y\cap \mathbb{O}_q=\emptyset$, we deduce that $f^0_\tau$ has no critical points on $L_0$. Let now $(C_{\hat f_\tau})_y=\{z_i,\ i\in K\}$, $K$ and denote by $z_i^0,\ i\in K^0\subset K$ the elements of the subset $(C_{\hat f_\tau,0})_y$. Let $u_i, u_j^0, i\in K, j\in K^0$ be the (obvious) projections of the $z_i\in (Z_0)_y$ to $L_0$.
Let $B^0$ be the union of a set of small (geodesic w.r.t. the standard metric) balls $B_i^0(\epsilon_i)\subset L_0$ of radii $\epsilon_i$ for $i \in \in K^0$ containing the $u_i^0$ and analogously $B^-$ the set $\{B_i(\epsilon_i)\subset L_0\}_i$ of balls containing the $u_i$ for $i \in K^-=K\setminus K_0$ so that the closures of all $B_i^0(\epsilon_i), B_i(\epsilon_i),\ i \in K$ are mutually disjoint. Then choose a smooth function $g:L_0\rightarrow \mathbb{R}^+$ so that $g$ equals $1$ on $B^0$ and is equal to a (small) constant $1 >c>0$ on $B^-$ which will be fixed below. Assume ${\rm dist}(u_i,\partial L_0) >0$ for all $i \in K^-$ (this can be always achieved by genericity). Then we can choose $g$ to be zero on the closure of a small nghbd $U_0$ of $\partial L_0$ in $L_0$ which is disjoint from all $B_i(\epsilon_i), \ i\in K^-$ while we can also choose a smooth function $h$ on $U_0$ which is zero outside $U_0$ and {\it expanding} on $U_0$ in the sense that $h \rightarrow \infty$ and $|dh|(x)\rightarrow \infty$ for $x \rightarrow x_0$, $x_0$ an arbitrary point of $\partial L_0$. Then setting $f^0_\tau=f_\tau\circ \Psi_\tau:L_0\rightarrow \mathbb{R}$ we set
\[
f^{c, c_1}_\tau(u)=g\cdot f^0_\tau+h, \ {\rm for\ all}\ u\in L_0\setminus \partial L_0.
\]
Further, choose a constant $a_1>0$ (again, to be fixed below) and let $z_{a_1}:\mathbb{R}^{n+1}\rightarrow \mathbb{R}^{n+1}$ be the smooth function whose $i$-th component coincides with $z_i$ for $-a_1<z_i$ and with $e^{z_i+a_1}-a_1-1$ for $-a_1\geq z_i$. Choose a point $z_0 \in L_0$ and finally set for the fixed $y\in B_1$ and for the chosen $\tau \in [0,1]$
\[
\hat f^{c,c_1,a_1}_\tau(u,y,z)=f^{c, c_1}_\tau(u)+\left(y-\hat \alpha_\tau(u), z_{a_1}+z_0\right), \ (u,y,z) \in L_0\times \{y\}\times \mathbb{R}^{n+1}.
\]
Now choosing $a_1>{\rm max}_{z \in L_0}(|z_0-z|)>0$ and $a$ so that $a>>a_1$ we see that if we choose $c,c_1$ sufficiently small than the $u$-coordinates of all critical points of $\hat f^{c,c_1,a_1}_\tau$ lie in $L_{0,0}$ and consist exactly of the $u$-cordinates of the $z_i^0\in (Z_{0,0})_y,\ i \in K^0$ introduced above while the $z$-coordinates of the critical points of $\hat f^{c,c_1,a_1}_\tau$ are translated by a constant relative to those of the $z_i^0\in (Z_{0,0})_y,\ i \in K^0$. Note that $\mathfrak{f}_\tau:=f^{c,c_1,a_1}_\tau: (Z_0)_y\rightarrow \mathbb{R}$ is of the form $\mathfrak{f}_\tau(u,z)=\mathfrak{g}_\tau(u)+\mathfrak{e}_\tau(u,z),\ u \in L_0,\ z \in \mathbb{R}^{n+1}$ with $\mathfrak{e}_\tau(u,z)$ being $d$-bounded in the sense of Eliashberg-Gromov (\cite{eliasgr}, Chapter 1.4) and $\mathfrak{g}_\tau(u)$ being a 'fibration at infinity' near $\partial L_0$ in the sense of (\cite{eliasgr}, Chapter 0.2.1). Then by results on (stable) Morse theory (cf. \cite{eliasgr}, Theorem 1.4.1), the number of (non-degenerate) critical points of $f^{c,c_1,a_1}_\tau$ for fixed $y\in B_1$ (and hence of those critical points of $\hat f_\tau$ lying on $(Z_{0,0})_y$) is bounded below by ${\rm stabMor}(\mathfrak{f}_\tau)_{dbd}$, using the notation of \cite{eliasgr}. Now extending the above to any $y \in B_1$ lying in the image of $\mathcal{C}_\tau$ under $\tilde \alpha_0$ we observe that any connected component of $\mathcal{C}_\tau$ consists of a critical point of $\hat f_\tau$ on $Z_{0,0}$ of constant index and crossing $S_1(C_{\hat f_\tau},0)$ tranversally means shifting the index of the corresponding critical point by $\pm 1$. We deduce that $\mathcal{C}_\tau$ consists of at least ${\rm stabMor}(\mathfrak{f}_\tau)_{dbd}$ connected components, which was the assertion.\\
The key ingredient of the proof of (1.) is a stability result for generating functions in a neighbourhood of fold singularities, as it is discussed for instance in Guillemin and Sternberg (\cite{guille}, Chapter VII). Recall that a fold point of $\underline L_\tau\subset T^*B_1$ is a point where the restriction of $(\pi_{B_1})_*:T\underline L_\tau \rightarrow TB_1$ to the tangent space of the set of points $S_1(L_\tau) \subset \underline L_\tau$ where $(\pi_{B_1})_*|T\underline L_\tau$ has corank one, has zero kernel, we will denote this set henceforth by $S_{1,0}(\underline L_\tau) \subset \underline L_\tau$. Generically this is again a submanifold of codimension $1$ in $\underline L_\tau$. Note that at non-transversality points in the sense of Lemma \ref{genericity} (1.), the kernel of $(\pi_{B_1})_*|T\underline L_\tau$ is necessarily tangent to $S_1(L_\tau)$, which is why they do not appear if we assume $S_1(L_\tau)\setminus S_{1,0}(\underline L_\tau)=\emptyset$. Recall also that by the diagram (\ref{generatingfam}) the dimension of the kernel $(\tilde \alpha_\tau)_*:TC_{\tilde f_\tau}\rightarrow TB_1$ equals the dimension of the kernel of $(\pi_{B_1})_*:T\underline L_\tau \rightarrow TB_1$ on points corresponding under the immersion $\hat L_\tau \cap H_\tau\rightarrow \underline L_\tau$ (using the injective immersion $\hat L_\tau \cap H_\tau \rightarrow C_{\tilde f_\tau}$). Now it is proven in \cite{guille} (Ch. VII, Lemma 6.1), that near a point $z \in S_{1,0}(C_{\hat f_\tau})$ (where the latter is defined in analogy to $S_{1,0}(\underline L_\tau)$) in $C_{\hat f_\tau}$ (for a fixed $\tau \in [0,1]$) resp. a neighbourhood of the corresponding point $\lambda \in \underline L_\tau\subset T^*B_1$ we can parametrize the latter by a function of one 'auxilliary variable' $\theta$, namely there is a $n+2$-dimensional submanifold $U\subset Z_0$ which projects to an open ngbhd of $\tilde \alpha_\tau(\lambda)$ in $B_1$ (take the common zero set of $2n+1$ functions $\frac{\partial}{\partial \theta_i} \hat f_{\tau}=0, \ i=1,\dots,2n+1$, where the $\theta_i$ are the fibre variables in some coordinate neighbourhood of $z$ in $Z_0$) so that $z \in U$ and $\hat f_\tau: Z_0\rightarrow \mathbb{R}$ restricted to $U$ is given by 
\begin{equation}\label{fold}
\hat f_{\tau,U}(x, \theta)=\mu(x)+\rho(x)\theta-\frac{\theta^3}{3},\ d\rho\neq 0,
\end{equation}
where $\rho, \mu: V:=\tilde \alpha_0(U)\rightarrow \mathbb{R}$, then the fibre critical set $C_{\hat f_\tau}\cap U \subset Z_0$ is given by $\{(x,\theta): \rho(x)=\theta^2\}$, so $\frac{\partial}{\partial \theta} \hat f_{\tau,U}=0$, the caustic $S_{1,0}(C_{\hat f_\tau})$ corresponds to the set $\rho=0$ on $U$ (which is the set $\frac{\partial^2}{\partial\theta^2}\hat f_{\tau,U}=0$ and $\frac{\partial^2 }{\partial\theta \partial x_1} \hat f_{\tau,U}\neq 0$ for some coordinate function $x_1$ on $V$). Note that while $V\subset B_1$ is an open neighbourhood of the point $y=\tilde \alpha_0(z)=\alpha_\tau(\lambda)\in B_1$, we have $U\simeq V\times J\subset V\times \mathbb{R}$, where $J$ is some open interval centered at $0$, then $C_{\hat f_\tau}\cap U$ is relatively open in $C_{\hat f_\tau}$. Now we claim that we can find an open neighbourhood of $\tau$, called $I_\tau$ and open sets $U_{\tau'}\subset U$ s.t. $z \in U_{\tau'}$ for any $\tau' \in I_\tau$. Then with $V_{\tau'} =\tilde \alpha_0(U_{\tau'})$ the family $\{\hat f_{\tau'}|U_{\tau'}\}_{\tau' \in I_\tau}$ is given up to addition of a family of functions constant along the fibres of ${\rm pr}_1: V_{\tau'}\times J\rightarrow V_{\tau'}$ by composing a smooth family of diffeomorphisms $g_{\tau'}: U_{\tau} \rightarrow U_{\tau'},\ \tau' \in I_\tau, g_\tau=Id_{U_{\tau'}}$ with $\hat f_{\tau,U_{\tau'}}$, that is 
\begin{equation}\label{stability1}
(\hat f_{\tau'}|U_{\tau'}+\mu'_{\tau'})\circ g_{\tau'}= \hat f_{\tau,U_\tau},
\end{equation}
for some smooth family $\mu'_{\tau'}:V_{\tau'}\rightarrow \mathbb{R}, \ \tau' \in I_\tau$. Furthermore there is another smooth family $h_{\tau'}: V_\tau \rightarrow V_{\tau'}, h_\tau=Id_{V_{\tau}}$ so that for all $\tau'\in I_\tau$ the diagram
\begin{equation}\label{stability2}
\begin{CD}
V_\tau\times J @>>g_{\tau'}> V_{\tau'}\times J \\
 @VV{\rm pr}_1 V @VV{\rm pr}_1 V \\
V_\tau @>> h_{\tau'}> V_{\tau'}\\
\end{CD}
\end{equation}
commutes. Following \cite{guille} (Ch. VII, § 8), (see also Guillemin and Golubitsky \cite{guilgo}, Ch. V, Theorem 4.2), this already follows if we can show that $\hat f_{\tau,U}$ is 'infinitesimally stable' in the appropriate sense. We will sketch a proof of this below (Lemma \ref{infstab}) and assume for the moment that for certain families $g_{\tau'}, \ h_{\tau'}, \ \tau'\in I_\tau$, (\ref{stability1}) and (\ref{stability2}) holds. Now the latter equations imply that for any $\tau'\in I_\tau$ we have that $\hat f_{\tau'}|{U_\tau'}$ is of the form (\ref{fold}) with $\mu,\ \rho,\ \theta$ replaced by $\tilde \mu=\mu\circ g_{\tau'}^{-1}+\mu'_{\tau'},\ \tilde \rho=\rho\circ g_{\tau'}^{-1},\ \tilde \theta=\theta\circ g_{\tau'}^{-1}$. But this implies that if $z'$ is a fold point of $\hat f_{\tau}|U_\tau$, then $\hat f_{\tau'}|U_{\tau'}$ has a fold point at $g_{\tau'}(z')$ for all $\tau'\in I_\tau$. By genericity, this means that the fold locus $S_{\tau'}^U:=S_{1,0}(C_{\hat f_{\tau'}})\cap U_{\tau'}$ is a codimension one submanifold for all $\tau'\in I_\tau$. By shrinking $U_\tau$ if necessary, we can assume that $S_{\tau}^U$ is connected, then $S_{\tau'}^U=g_{\tau'}(S_{\tau}^U)$ is connected and will divide $U_{\tau'}$ into exactly two connected components for all $\tau'\in I_\tau$.\\
We can repeat the above procedure for all $z \in S_{1,0}(C_{\hat f_{\tau}},0):= S_{1,0}(C_{\hat f_{\tau}})\cap Z_ {0,0}$ and all $\tau  \in [0,1]$ which gives in analogy to $(I_\tau, U, U_{\tau'})$ above a family $(I_{z,\tau}, U_{z,\tau}, U_{z, \tau,\tau'}),\ \tau'\in I_{z,\tau}$. Consider the pair  $(I_{z,\tau}, \hat U_{z,\tau}:=U_{z,\tau}\cap Z_{0,0})\subset ([0,1],C_{\hat f_{\tau},0})$ and the codimension one submanifolds $\hat S_{\tau'}^{\hat U_{z,\tau}}:=S_{\tau'}^{U_{z,\tau}}\cap Z_{0,0}, \tau' \in I_{z,\tau}$ dividing each $U_{z,\tau,\tau'}\cap \cap Z_{0,0}$ into exactly two connected components. Consider the covering $\mathcal{U}=\bigcup_{\tau \in [0,1]}\bigcup_{\tau' \in I_\tau}\bigcup_{z \in  S_{1,0}(C_{\hat f_\tau},0)} I_{z,\tau} \times \hat U_{z,\tau,\tau'}$ of $\bigcup_{\tau \in [0,1]}\{\tau\}\times S_{1,0}(C_{\hat f_\tau},0)$ in $\bigcup_{\tau \in [0,1]}\{\tau\}\times C_{\hat f_\tau,0}\subset [0,1]\times Z_{0,0}$. By the assumption of (1.) each $S_{1,0}(C_{\hat f_\tau},0)$ coincides with its closure $\overline {S_{1,0}(C_{\hat f_\tau},0)}$ and hence is compact in $C_{\hat f_\tau,0}$. Since $\bigcup_{\tau \in [0,1]}\{\tau\}\times \overline {S_{1,0}(C_{\hat f_\tau},0)}$ is given by the zero set of the determinant of the vertical Hessian, ${\rm det}(D_{(u,y,z)}^2\hat f_\tau)$ on $\bigcup_{\tau \in [0,1]} Z_{\tau,0}$, it is also compact, thus we can chose a finite subcover $\mathcal{U}^0=\bigcup_{\tau \in I^0}\bigcup_{\tau' \in I_\tau}\bigcup_{z \in S^0_\tau} \hat U_{z,\tau,\tau'}$ of $\bigcup_{\tau \in [0,1]}\{\tau\}\times \overline {S_{1,0}(C_{\hat f_\tau},0)}$ where $I^0,\ S^0_\tau$ are some finite indexing subsets of $[0,1],\ S_{1,0}(C_{\hat f_\tau},0)$, respectively. \\
Then let $\mathcal{C}^0:=\bigcup_{\tau \in [0,1]} \mathcal{C}^0_\tau$ where $\mathcal{C}^0_\tau:=C_{\hat f_\tau,0}\setminus (\mathcal{U}^0\cap C_{\hat f_\tau,0}),\  \tau \in [0,1]$. Let $z_0 \in \mathcal{C}^0_0\subset Z_{0,0}$. We want to connect $z_0$ to a point $z_e$ in $\mathcal{C}_1$ by a smooth path whose image is entirely contained in $\mathcal{C}$. Assume that for some time $t_1 \geq 0$ and for any small $\epsilon >0$ we have $z_0\notin \mathcal{C}^0_{\tau}$ for $t_1+\epsilon> \tau  >t_1$ while for $0<\tau \leq t_1$ we have $z_0 \in \mathcal{C}^0_\tau$. If there is no such $t_1$ we are done, since we can set $c(t)=z_0, \ t\in [0,1]$. So $z_0 \in \hat U_{z,\tau_1, t_1}$ for some $\tau_1 \in I^0,\ t_1\in I_{z,\tau_1}$ and some $z \in S^0_{\tau_1}$. By slightly moving $z_0$, we can assume that $z_0 \notin \hat S_{t_1}^{\hat U_{z,\tau_1}}$. Thus since $I_{z,\tau_1}$ is open, there is a smooth path $c:(t_1,t_2)\subset I_{z,\tau_1} \rightarrow \bigcup_{\tau'\in I_{z,\tau_1}} U_{z,\tau_1, \tau'}$ for some $t_2>\tau_1>t_1$ so that $c(\tau')\notin \hat S_{\tau'}^{\hat U_{z,\tau_1}}$ for all $\tau' \in (t_1,t_2)$ and its limit point $z_1$ for $\tau'\rightarrow t_2$ is either an element of $\mathcal{C}^0_{t_2}$ or lies in some $\hat U_{\tilde z,\tau_2, t_2}$ for some $\tau_2> t_2$. In the first case we start our arguments from the beginning, in the second we continue our path $c$ through $I_{\tilde z,\tau_2}$ in $\bigcup_{\tau'\in I_{\tilde z,\tau_2}} (\hat U_{\tilde z,\tau_2,\tau'}\setminus \hat S_{\tau'}^{\hat U_{\tilde z,\tau_2}})$ as in the second part of the previous step. Proceeding in the above way we arrive after a finite number of steps at a point $z_e:=c(1)\in \mathcal{C}$. Certainly we can chose for each $\tau \in [0,1]$ the full connected component $\mathcal{U}_\tau$ of $\mathcal{C}_\tau$ for which $c(\tau)\in \mathcal{U}_\tau$ and we arrive at the second assertion of the Proposition.\\
Consider now the general case (2.), i.e. $\bigcup_{\tau \in [0,1]}\{\tau\}\times  {S_{1,0}(C_{\hat f_\tau},0)}$ is not compact. If we set $S_{1,0}(\bigcup_{\tau\in [0,1]} C_{\hat f_\tau},0):=\bigcup_{\tau \in [0,1]}\{\tau\}\times {S_{1,0}(C_{\hat f_\tau},0)}$ we certainly have a dense inclusion $(A\subset S_1=B \subset C):=S_{1,0}(\bigcup_{\tau\in [0,1]} C_{\hat f_\tau},0)\subset \bigcup_{\tau \in [0,1]}\{\tau\}\times {S_{1}(C_{\hat f_\tau},0)} \subset \overline {S_{1,0}(\bigcup_{\tau\in [0,1]} C_{\hat f_\tau},0)}$, where the latter is just the zero set of the determinant of the vertical Hessian, ${\rm det}(D_{(u,y,z)}^2\hat f_\tau)$ on $\bigcup_{\tau \in [0,1]} Z_{\tau,0}$. Now by the stability poperty (\ref{stability2}) the dense inclusion $A\subset B=S_1$ is also open, that is the union of fold points $A=S_{1,0}(\bigcup_{\tau\in [0,1]} C_{\hat f_\tau},0)$ in $\bigcup_{\tau\in [0,1]} C_{\hat f_\tau}\times \{\tau\}\subset Z_{0,0}\times[0,1]$ is open and dense in $S_1$ and furthermore locally over $[0,1]$ the graph of the family of mappings $g_\tau': U_\tau \cap S_{1,0}(C_{\hat f_\tau},0)\rightarrow U_{\tau'}\cap S_{1,0}(C_{\hat f_{\tau'}},0)$ for any $\tau\in [0,1],\ \tau' \in I_\tau$ as above. From this it already follows that along $A$, the kernel of the vertical hessians $D_{(u,y,z)}^2\hat f_\tau$ defines a smooth one-dimensional distribution transversal to $A$ which is for any $\tau \in [0,1]$ and at any point $(u,y,z) \in A$ oriented by the requirement that
\begin{equation}\label{morseorient}
\frac{d}{dt}_{t=0}(c'(t), D_{(c(t))}^2\hat f_\tau c'(t))>0,
\end{equation}
where $c:(-\delta, \delta)\rightarrow C_{\tilde f_\tau}$ is any path s.t. $c(0)=(u,y,z)$ and $c'(0)$ spans the kernel of $D_{(u,y,z)}^2\hat f_\tau$. To prove (3.), it remains to show that this defines a coorientation of $A$ that coincides with the opposite of the Maslov coorientation as defined by (\ref{stratm}), Lemma \ref{genericity} and \cite{arnold}. To see this, we choose as above for a fixed $\tau \in [0,1]$ and a point $z \in S_{1,0}(C_{\hat f_\tau},0)$ a $n+2$-dimensional submanifold $U\simeq V\times J \subset Z_0$ so that $\hat f_\tau|U: U\subset Z_0\rightarrow \mathbb{R}$ is of the form (\ref{fold}). On $C_{U,\tau}:=C_{\hat f_\tau}\cap U \subset Z_0$ and since $d\rho \neq 0$, $\rho(x)=\theta^2$, we can fix coordinates $(\theta, x_2, \dots, x_n)=:(\theta,x)$, so that we have (compare (\ref{fold}))
\begin{equation}\label{Crestr}
\varphi(x, \theta):=\hat f_{\tau,U}|C_{U,\tau}(x, \theta)=\mu(x, \theta^2)-\frac{2\theta^3}{3}.
\end{equation}
Then by (\ref{generatingfam}) $(d\hat f_{\tau,U})|C_{U,\tau}\subset T^*U$ is diffeomorphic to a ngbhd of a fold point $w \in \underline L_\tau$ so that ${\rm pr}_1(z)=\pi_{B_1}(w)$.
Note that the vertical Hessian, $D_{(u,y,z)}^2\hat f_\tau$, restricted to $C_{U,\tau}$, is just $2\theta$, so the derivative of $c:(-\delta, \delta)\rightarrow C_{U,\tau},\ c(\theta)=\theta$ orients ${\rm ker} D_{(u,y,z)}^2\hat f_\tau$ positively by (\ref{morseorient}). Note that since $1/2(\varphi(x, \theta)+\varphi(x, -\theta))= \mu(x, \theta^2)$ we can assume that in our coordinate system, $\frac{d}{d\theta}\mu(x, \theta^2)=0$. From this and (\ref{Crestr}) it follows that
\[
\frac{d}{d\theta}\left((d\hat f_{\tau,U})_{c(\theta)}(\frac{d}{d\theta})\right)=-4\theta.
\]
So if we set $\tilde c:(-\delta, \delta)\rightarrow U(n+1)/SO(n+1), \tilde c(0)=Id$ the path associated to the family of Lagrangians subspaces $t \mapsto T_c(t)\underline L_\tau \subset T^*B_1$, we see that ${\rm det^2}(\tilde c(t)) \in \mathbb{C}^*$ crosses $+1$ clockwise, so by \cite{arnold} we arrive at our assertion (3.).\\
Finally to show (2.) using the above, just note that (with the above notation) and by Lemma \ref{genericity}, outside of non-transversality points $C$ is a stratified set in $\mathcal{W}=\bigcup_{\tau \in[0,1]} C_{\hat f_\tau,0}\times\{\tau\}$ with smooth topstratum $S_1=\bigcup_{\tau \in [0,1]}\{\tau\}\times {S_{1}(C_{\hat f_\tau},0)}$ of codimension one while the lower $S_{i,0}(\underline L_\tau), i\geq 2$ are of codimension at least $\geq 3$ in $\mathcal{W}$ resp. $S_{1,1, \dots}(\underline L_\tau)$ are of codimension at least $\geq 2$ in $\mathcal{W}$ (using the usual notation for Thom-Boardman singularities). Thus the set of fold points $A\subset S_1$ is open and dense in $S_1$ and cooriented by the family of vertical Hessians as decribed above. Now, the set of non-transversality points are isolated in each $C\cap C_{\hat f_\tau,0}\times\{\tau\}$ by Lemma \ref{genericity} (and appear as isolated singular points in $S_1$) and thus at worst of dimension $1$ in $\mathcal{W}$. Hence we can find for any $w_0\in \mathcal{C}_0=\mathcal{C}_1$ a path $c:[0,1] \rightarrow \mathcal{W}$ s.t. $c(0)=c(1)$ that intersects the set of 'caustic' points $C$ at most in fold points $A\subset S_1 \subset C$ transversally and we have for any $\tau \in [0,1]$ so that for $c(\tau)\in \mathcal{C}_\tau\times \{\tau\}$ 
\[
{\rm im}\ c|_{[0,\tau]}\cdot S_1= [{\rm index}(D_{(c(0))}^2\hat f_0)-{\rm index}(D_{(c(\tau))}^2\hat f_\tau)]_2,
\]
where the left hand side denotes oriented intersection index, ${\rm index}(D_{z}^2\hat f_\tau)$ is the number of negative eigenvalues of the vertical Hessian of $\hat f_\tau$ at a non-caustic point $z \in \mathcal{C}_\tau\times \{\tau\}$ and $[\cdot]_2:\mathbb{N}\rightarrow \mathbb{N}$ is the function so that $[k]_2=\frac{k}{2}$ for $n$ even and $[k]_2=\frac{k-1}{2}$ for $k$ odd (we use the oriented Maslov-index). Since $c(0)=c(1)$, we arrive at the assertion.
\end{proof}
It remains to prove the following.
\begin{lemma}\label{infstab}
Let $\tau \in [0,1]$. Let $z \in S_{1,0}(C_{\hat f_\tau})$, that is, $z$ is a fold point of $C_{\hat f_\tau}\subset Z_0$. Then the function (\ref{fold}) defining $C_{\hat f_\tau}\cap U$ resp. $S_{1,0}(C_{\hat f_\tau})\cap U$ on some $n+2$-dimensional submanifold $U\subset Z_0$ so that $z\in U\simeq V\times J$ where $V \subset B_1$ is open and $J\subset \mathbb{R}$ is an open interval as described in the proof of Proposition \ref{keysing} is stable in the sense of (\ref{stability1}) and (\ref{stability2}).
\end{lemma}
\begin{proof}
The sketched proof will be a slight modification of the arguments of \cite{guille} (Ch. VII, § 8, Thm. 8.2). We assume that $z=0 \in \mathbb{R}^{n+1}$, $\tau=0 \in \mathbb{R}$, then $U\simeq V \times J \subset \mathbb{R}^{n+1}$ will be some neighbourhood of the origin. Set $\hat f_{\tau', U}:=\hat f_{\tau'}|U$ for any $\tau' \in I_\tau$ (using the notation of the previous proof). Let $\mathfrak{e}(x,\theta,\tau')=\frac{d}{dt}|_{t=\tau'}(\hat f_{t,U}-\hat f_{\tau,U})$ for any $\tau' \in I_\tau$. Write $g_i(x,\theta,\tau')=h_i(x,\tau'),\ i=1,\dots, n+1, \ g_{\theta}(x,\theta,\tau')$ for the coordinate-functions of $g_{\cdot}, h_{\cdot}$ on $V\times J\times  I_\tau$ resp. $V\times I_\tau$. Then if $g_{\tau'}, h_{\tau'}$ satisfy (\ref{stability1}) for some $\tau'\in I_\tau$, we necessarily have
\begin{equation}\label{infstable}
-\mathfrak{e}(x,\theta,\tau')=a_0(x,\tau')+\sum_{i=1}^{n+1}\frac{\partial}{\partial x_i} \hat f_{\tau', U} a_i(x,\tau')+ \frac{\partial}{\partial \theta} \hat f_{\tau', U} b_\theta(x,\theta, \tau')
\end{equation}
where we have set
\begin{equation}\label{components}
\begin{split}
a_i(x,\tau')&=\frac{d}{d\tau'}h_i(h_{\tau'}^{-1}(x),\tau'), \quad i=1,\dots, n+1,\\
b_\theta(x,\theta, \tau')&=\frac{d}{d\tau'}g_\theta(g_{\tau'}^{-1}(x, \theta),\tau'),\\
a_0(x,\tau')&=\frac{d}{d\tau'}\mu_{\tau'}(x)+\sum_{i=1}^{n+1} \frac{\partial}{\partial x_i}a_i(x,\tau').
\end{split}
\end{equation}
Now if $\mathcal{E}_{n+2}$ is the ring of germs of smooth functions at the origin in  $\mathbb{R}^{n+2}$ and $\mathcal{R}=\mathcal{E}_{n+2}/I_{\hat f_{\tau', U}}$ its quotient by the ideal $I_{\hat f_{\tau', U}}$ generated by $\frac{\partial}{\partial \theta} \hat f_{\tau', U}$ in $\mathcal{E}_{n+2}$, then $\alpha_1,\dots, \alpha_k$ generate $\mathcal{R}$ as a module over $\mathcal{E}_{n+1}$, where the latter acts on $\mathcal{R}$ by means of pullback by the obvious projection $p: \mathbb{R}^{n+2}\rightarrow \mathbb{R}^{n+1}, \ (x, \theta, \tau')\mapsto (x,\tau')$, if and only if $\mathcal{R}/\mathcal{M}_{n+1}\mathcal{R}$, where $\mathcal{M}_{n+1}$ is the maximal ideal of $\mathcal{E}_{n+1}$ of functions vanishing at $0$, is generated by the images of the $\alpha_i$ under the canonical projection: this is the content of Malgrange's preparation theorem. Thus (\ref{infstable}) will be satisfied in a neighbourhood of the origin in $\mathbb{R}^{n+2}$ if and only if it will be satisfied if all functions in (\ref{infstable}) are evaluated at $(x,\theta,\tau')=(0,\theta, 0)$. Thus we can solve the latter equation for smooth functions $a_i, b_\theta$ for small $(x,\tau')$ if we can do it for $x=0, \tau'=0$. But the latter in turn follows from the proof of Ch. VII, Theorem 8.3 in \cite{guille}. To be precise let $\Psi(\theta)=\hat f_{\tau'=0, U}|\{0\}\times J$ and let $I_{\Psi}$ be the ideal generated by $\frac{\partial}{\partial \theta} \Psi(\theta)$ in $\mathcal{E}_{1}$ and consider a basis $\Psi_1,\dots, \Psi_j$ of $\mathcal{R}=
\mathcal{E}_{1}/I_\Psi$. Now the codimension of the ideal generated by the first differentials of $\Psi$ in $\mathcal{E}_{1}$ is $2$ (having $\rho(0)=0$ in (\ref{fold})), so we have $j=2$. Thus we have by the arguments in loc. cit.
\[
\hat f_{\tau'=0, U}(x, \theta)=f_0(x)+
f_1(x)\Psi_1(\theta)+\Psi(\theta)+\mathfrak{e}(x,\theta)
\]
where $\mathfrak{e}(x,\theta)$ vanishes to infinite order at $x=0$ and $d f_1|_{x=0}\neq 0$. Differentiation in the $x$-variables at $x=0$ then implies that $\Psi_1$ is a linear combination with $\mathbb{R}$-coefficients of the partial differentials $\frac{\partial}{\partial x_i}\hat f_{0, U}|\{0\}\times J$ for $i=1,\dots, n$. Thus we have shown that $\hat f_{\tau', U}$ satisfies (\ref{infstable}) at $(x,\theta,\tau')=(0,\theta, 0)$.\\
Now it remains to solve the equations (\ref{components}) having smooth left-hand sides for small $\tau'$. But that follows from standard methods of ordinary differential equations in the case of the first two equations and by linear Hamilton-Jacobi theory in the third case by eventually again slightly shrinking $U$ and noting that we have modified $\hat f_\tau$ in (\ref{modify}) so that all flows are tangential to $\partial Z_0$ on $\partial Z_0$ and thus do exist locally. \end{proof}

\subsection{General elements of the spectrum}\label{generalspec}
We will discuss in this section aspects of a possible proof of Conjecture \ref{conjspec} in the case of a quasihomogeneous polynomial $f$ with an isolated singularity at $0$. According to this conjecture, it should be possible to replace the 'exponent' of $f$, that is the element of the spectrum corresponding to the monomial $z^{\alpha(1)}=1$ in the Milnor algebra $M(f)$ (thus $\alpha(1)=0$) by any spectral number corresponding to $z^{\alpha(i)}$, where $\alpha(i) \in \Lambda,\ \Lambda \subset \mathbb{N}^{n+1}$ to establish (by its non-vanishing) an obstruction for $\rho$ being of finite order in ${\rm Symp}(F,\partial F,\omega)$. One step in this will be the identification of general elements of the spectrum and certain Maslov-type indizes if $\rho$ is of {\it finite} order in ${\rm Symp}(F,\omega)$. We will discuss a preliminary result in that direction under the assumption that certain higher singularities vanish (namely an analogue of the result (1.) of Proposition \ref{keysing}). We will discuss at the end of this section how the assumption on the vanishing of 'higher singularities' can be weakened on the basis of the discussion in the proof of Proposition \ref{keysing} above to achieve the same result.  \\
Let as above be $\phi$ the map sending $z^{\alpha(i)}$ to the class given by
\[
\phi(z^{\alpha(i)})=z^{\alpha(i)}dz_0\wedge\dots\wedge dz_n
\]
in $\mathcal{H}''_{0}$, thus defining a $\mathbb{C}$-isomorphism of vectorspaces $\phi:M(f)\simeq \mathcal{H}''_{0}/f\mathcal{H}''_{0}$, and let, for each $i \in \{1, \dots, \mu\}$ be $\hat s_i$ a global section of $\mathcal{H}^n(f_*\Omega^\cdot_{\hat X/D^*_{\delta}})$ defined by the polynomial $\alpha(i) \in M(f)$, that is, having the property that 
\begin{equation}\label{features3}
z^{\alpha(i)}dz_0\wedge\dots\wedge dz_n=d f\wedge \hat s_i \in \Omega^{n+1}(\hat X,\mathbb{C}).
\end{equation}
where we use here and in the following throughout the notation of Section \ref{relativen0}. Assume now that $\rho \in {\rm Symp}(F,\omega)$ is of finite order $k=m \cdot \beta, m \in \mathbb{Z}$ in ${\rm Symp}(F,\omega)$ and we have chosen an arbitrary fixed path $\rho^k_\tau , \tau \in [0,1]$ connecting $\rho^k$ to the identity in ${\rm Symp}(F,\omega)$. Analogously to Section \ref{relativen0} above, we denote by 
\[
s_i^k:=i_{X_{f^k}}(\pi^k)^*(z^{\alpha(i)}dz_0\wedge\dots \wedge dz_n))
\]
a representative  $s_i^k \in \Gamma({\bf H}^n(Z^k,\mathbb{C}))$ satisfying (\ref{features3}). Note that since in the constructions of this Section, we will only restrict ourselves to (the union of) the sets $\tilde A_u\subset \tilde X^k, \ u \in S^1_\delta$ of (\ref{tildea}) of Section \ref{relativen0}, we will not elaborate on how to modify $s_i^k$ in a neighbourhood of the boundary of the fibration $\tilde X_k$ to give actually a well-defined element of $\Gamma({\bf H}^n(Z^k,\mathbb{C}))$, so $s_i^k$ will be considered here more precisely as element of $\Gamma({\bf H}^n(Z^k\cap \bigcup_{u \in S^1_\delta}\tilde A_u ,\mathbb{C}))$. Repeating the constructions above Assumption \ref{ass4} resp. in the proof of Propostion \ref{keylemma} we construct a $1$-parameter-family of immersions $i_\tau: Q_x \times [0,1] \rightarrow Y^k_\epsilon$ whose images $Q_\tau:={\rm im }(i_\tau)$ factorize for $\tau=0,1$ into closed $n+1$-dimensional submanifolds $Q_0,\ Q_1\subset \tilde A_u\subset  Y^k_\epsilon$ so that $Q_1=Q$ with $Q$ as defined by (\ref{hatq}) and so that any intersection $Q_\tau\cap (Y^k_\epsilon)_u,  \tau \in [0,1], u \in S^1_\epsilon$ is a Lagrangian sphere (resp. a union of two Lagrangian spheres for $u=x$) in $(Y^k_\epsilon)_u$. Let now for any $i \in \{1, \dots, \mu\}, \ \tau \in [0,1]$ be $\mathcal{Z}_{i,\tau}\subset Q_\tau$ be the 'divisor' defined by the set $Q_\tau\cap \{z \in \mathbb{C}^{n+1}: z^{\alpha(i)}=0\}$ and set in the following $Q^i_\tau=Q_\tau\setminus \mathcal{Z}_{i,\tau}$ and $Q^i=Q^i_1$. Let $\hat Q^i_\tau=i_\tau^{-1}(Q_\tau^i)$ and set $\hat Q^i_{\tau,x(t)}=\hat Q^i_\tau\cap(Q_x\times \{t\})$ for any $t \in [0,1]$. Thus we have for any $i \in \{1, \dots, \mu\}$ a family of Lagrangian spheres minus 'divisors' in $(Y^k_\epsilon)_{x(t)}$ given for $\tau,t \in [0,1]$, $x(t)=xe^{2\pi it}$ and fixed $x\in S^1$ by
\begin{equation}\label{qfamily3}
Q^i_{\tau,x(t)}=i_\tau(\hat Q^i_{\tau,x(t)}),\ \tau\in [0,1],\ t \in [0,1],\ {\rm  s.t.}\  Q_{\tau,x(1)}=\Phi_{H(x)}(1-\tau)\circ\rho_{\tau}(Q_x) \subset (Y^k_\epsilon)_x,\ i \in \{1, \dots, \mu\},
\end{equation}
for the latter equality again compare (\ref{triv46}). For any such $Q^i_{\tau,x(t)}$ we have a section $\Lambda_{Q^i_{\tau,x(t)}} \in \Gamma(i_{\tau,t}^*\widetilde {\rm Lag}(\tilde X^k,\Omega^k))$, where $i_{\tau,t}:Q^i_{\tau,x(t)}\hookrightarrow X_e^k$ is the inclusion, which is given for $z \in Q^i_{\tau,x(t)}$ by
\begin{equation}\label{lagsubspaces576}
\Lambda_{Q^i_{\tau,x(t)}}(z)=T^h_{z}Y^k_\epsilon \oplus T_{z}Q^i_{\tau,x}\subset T_{z}\tilde X^k,\ \tau\in\{0,1\},
\end{equation}
where $T^hY_\epsilon^k$ denotes the $\Omega^k$-orthogonal complement of $T^vY_\epsilon^k$ in $TY_\epsilon^k$. Then by Lemma \ref{lag5}, $\Lambda_{Q^i_{\tau,x(t)}}$ induces a non-vanishing section $\kappa_{Q_{\tau,x(t)}}\in \Gamma(i_{\tau,t}^*\Lambda^{(n+1,0)}T^*\tilde X^k)$ of unit length for any $\tau \in [0,1], t \in [0,1]$ and a family of functions $g^i_{\tau,x(t)}:Q^i_{\tau,x(t)}\rightarrow \mathbb{C}^*$ by setting
\begin{equation}\label{maslovfunction3}
g^i_{\tau,x(t)}\kappa_{Q^i_{\tau,x(t)}}=((\pi^k)^*(z^{\alpha(i)}dz_0\wedge\dots\wedge dz_n))|i_{\tau,t}^*T^*\tilde X^k=X_{f^k}^*\wedge s_i^k|i_{\tau,t}^*T^*\tilde X^k
\end{equation}
For $\tau,t \in [0,1]$, let $N^{\mathcal{M},i}_{\tau,t}\subset Q_{\tau,x(t)}$ be associated to the triple $(Q^i_{\tau,x(t)}, \Lambda_{Q^i_{\tau,x(t)}}, X_{f^k}^*\wedge s^k_i|i_{\tau,t}^*T^*\tilde X^k)$ by Definition \ref{generic} resp. Lemma \ref{lag5}, then for $\tau\in \{0,1\},\ i=1, \dots, \mu$, $N^{\mathcal{M},i}_\tau:=i_\tau(\hat N^{i,\mathcal{M}}_\tau)\in H_n(Q^i_\tau,\mathbb{Z})$ and $N^i_\tau:=\bigcup_{t \in [0,1]} N^i_{\tau,t}\subset Q^i_\tau$ where $N^i_{\tau,t}=\{z \in Q^i_{\tau,x(t)}: g^i_{\tau,x(t)}(z)\in \mathbb{R}\}$ both represent the Poincare dual of the Maslov class $[(g^i_\tau)^*\beta]\in H^1(Q^i_\tau,\mathbb{Z})$, where $g^i_\tau:Q^i_\tau\rightarrow \mathbb{C}^*, \tau \in \{0,1\}$ assembles the family (\ref{maslovfunction3}).
We conjecture the following:
\begin{conj}\label{maslovspec}
Assume $\rho$ is of finite order $k =m\beta, m \in \mathbb{Z}$ in ${\rm Symp}(F,\omega)$ and that the family $N^{\mathcal{M},i}_{\tau,1}, \tau \in [0,1]$ consists only of fold-type singularities (in a sense to be made precise in the proof of Proposition \ref{foldstab} below). Then if $\gamma_i, i=1, \dots, \mu$ are the elements of the spectrum of $f$ and $c^i$ is any closed path $c^i:[0,1] \rightarrow Q^i$ that satisfies $\pi^k:(c(t))=\delta\cdot e^{it}$ we have
\begin{equation}\label{specmaslov}
k\cdot(\gamma_i+1)= [c^i]\cdot N^{\mathcal{M},i}_1= {\rm wind}(g^i_{1,x(\cdot)}(c(\cdot))).
\end{equation}
where $[c^i]$ is the element of $H_1(Q^i, \mathbb{Z})$ represented by $c^i$.
\end{conj}
{\it Remark.} Note that the hypothesis that $\rho$ is of finite order in ${\rm Symp}(F,\omega)$ is always satisfied since $f$ is quasihomogeneous (by use of the flow  $\Phi_{H(x)}$). Thus the conjecture says that (\ref{specmaslov}), which is valid for the Maslov index along orbits of the weighted circle action (see the comments below) is also valid in $Q_1$, the latter defined by symplectic parallel transport in $\tilde X^k$ and the choice of {\it any} fixed path joining $\rho^k$ to $Id$ in ${\rm Symp}(F,\omega)$ for any $k=m\beta$. We expect that $\rho$ being of finite order in ${\rm Symp}(F,\partial F,\omega)$ actually forces all Maslov indizes as above resp. elements of the spectrum of $f$ to be zero (as it is the case for $i=1$ by Theorem \ref{symplecticmonodromy}.)\\
The conjecture is actually the 'analogue' of Claim 1. and Claim 2. in the proof of Proposition \ref{keylemma} which is valid for the case $i=1$. Note that the first Claim is the actual hard part in the proof of Proposition \ref{keylemma} in the case $i=1$ and we will only prove 'one half' for the case of general elements of the spectrum here, namely, an extension of (1.) of Proposition \ref{keysing}.
\begin{prop}\label{foldstab}
Let $i \in \{1,\dots, \mu\}$. Each member of the family $N^{\mathcal{M},i}_{\tau,1}\subset Q^i_{\tau,x(1)}, \tau\in [0,1]$ is {\it generic} outside of a discrete subset. Assume further that the family $N^{\mathcal{M},i}_{\tau,1}, \tau \in [0,1]$ consists only of fold-type singularities (as precised in the proof below), then (3.) of Assumption \ref{ass4} is valid for the family of pairs $(Q^i_{\tau,x(1)}, N^{\mathcal{M},i}_{\tau,1}), \tau \in [0,1]$.
\end{prop}
\begin{proof}
We have to prove that, given genericity and the hypothesis, for each $i \in \{1,\dots, \mu\}$,  there is a {\it path-connected} subset $\mathcal{U}^i\subset \bigcup_{\tau \in [0,1]} Q^i_{\tau,x(1)}\times \{\tau\}\subset (Y^k_\epsilon)_{x(0)}\times [0,1]$ of the form $\mathcal{U}^i=\bigcup_{\tau \in [0,1]}\mathcal{U}^i_{\tau}\times\{\tau\}$ where $\mathcal{U}^i_{\tau}\subset Q_{\tau,x(1)}, \tau \in [0,1]$ is a family of connected embedded, non-empty $n$-manifolds and one has $\mathcal{U}^i_{\tau}\subset Q^i_{\tau,x(1)}\setminus  \overline {N_{\tau,1}^{\mathcal{M},i}}$. Further each $\mathcal{U}^i_{\tau}$ is open in $Q^i_{\tau,x(1)}$  and for any $\tau \in [0,1]$ equals a connected component of $Q^i_{\tau,x(1)}\setminus \overline {N_{\tau,1}^{\mathcal{M},i}}$.\\
Thus let $L_\tau$ be the $S^1$-family of Lagrangian submanifolds (with boundary) in $\tilde X^k$ defined in Section \ref{app4} (we only consider the case $t=1$ here, for $t=0$ this family is constant), so that $L_\tau \cap (Y^k_\epsilon)_x=Q_{\tau,x(1)}$. For each $\tau \in [0,1]$, we have a section $\Lambda_{s^k_i}$ of $\Gamma(i_{\tau}^*\widetilde {\rm Lag}(\tilde X^k,\Omega^k))$ along $L^i_\tau$, where $i_\tau:L^i_\tau \hookrightarrow \tilde X^k$ denotes the inclusion of $L_\tau \cap \{z \in \mathbb{C}^{n+1}: z^{\alpha(i)}=0\}$. Analogously as in Section \ref{genericity}, one sees that one can find families of mappings $\kappa_i:\bigcup_\tau L^i_\tau \times \{\tau\}\rightarrow U(n+1)$ whose image lies in a small nghbd of $Id$ so that $\{\kappa_i\cdot (\Lambda_{s^k_i}, \tau), \tau \in [0,1]\}$ satisfies the required transversality property as a mapping into the pair $(\bigcup_\tau (i_\tau^*(\widetilde {\rm Lag}(\tilde X^k,\Omega^k)), \tau), \bigcup_{\tau}{\rm im}(\bigcup_\tau \Lambda_{L^i_\tau}\times \{\tau\})$. Here $\Lambda_{L^i_\tau}:L_\tau \rightarrow i_\tau^*(\widetilde {\rm Lag}(\tilde X^k,\Omega^k))$ is the obvious mapping assembling the mappings $\Lambda_{Q^i_{\tau,x(1\pm t)}}$ as discussed above for some small $0<t<\epsilon$ for each $\tau \in [0,1]$ (it is just the tangential mapping of $L^i_\tau$).\\
For a given $x \in L^i_\tau$ and by our assumption on the vanishing of 'higher singularities', one can extend $\Lambda_{s^k_i}$ smoothly to a neighbourhood $U$ of $x$ in $\tilde X^k$ by Lemma 2.5 of Entov \cite{entov} (our assumption implies the property of $\Sigma_2$-nonsingularity in loc. cit.). Further by \cite{morin} (cf. also \cite{entov}), there is a diffeomorphism $h^i_\tau:U\rightarrow \mathbb{R}^{2(n+1)}$ and an integer-number $0\leq s\leq n+1$, so that $h^i_\tau(U\cap L^i_\tau)=\mu_{s}(\mathbb{R}^{n+1})$ where $\mu_s:\mathbb{R}^{n+1}\rightarrow \mathbb{R}^{2(n+1)}$ is given by
\[
\mu_s(y_1,\dots,y_{n},t)=(y_1,\dots, y_n, \sum_{i=1}^{s-1}y_it^i +t^{s+1},t ,0, \dots, 0).
\]
Then our assumption on the vanishing of 'higher singularities' corresponds precisely to the situation that for any $x \in L_\tau^i$, $s$ is either $0$ or $1$, the latter case corresponding to the set of fold points $S_{1,0}(L^i_\tau)$. Note that $s>1$ corresponds to the singularities of type $S_{1,1,\dots}(L^i_\tau)$ in the Thom-Boardman classification. Let $x \in S_{1,0}(L^i_\tau)$, then we can describe the image of $h^i_\tau(U\cap L^i_\tau)$ by the generating function $f:\mathbb{R}^{n+1+N}\rightarrow \mathbb{R}$ given by
\[
f(y_1, \dots, y_n,t)=y_1\cdot t-\frac{t^3}{3},
\]
in complete analogy to the setting in Lemma \ref{genfunction}. Note that the pair $(L_\tau^i, \Lambda_{s^k_i})$ is only, using the above image of $L_\tau^i$ in $\mathbb{R}^{2n+2}$, locally in $L^i_\tau$ around each $x \in L^i_\tau$ given by a generating function in the above sense, nevertheless, we can apply the same arguments using stability of fold-singularities $S_{1,0}(L^i_\tau)$ as in the proof of (1.) of Proposition \ref{keysing} to construct open sets $U_{x,\tau}, x \in S_{1,0}(L^i_\tau), \tau \in [0,1]$ as in the same proof whose union over $x, \tau$ covers $S_{1,0}(L^i_\tau)$ and so that $U(z, \tau)$ is diffeomorphic to $U(\Phi_{\tau,z}(z), \tau')$ for $|\tau-\tau'|<\epsilon'(\tau,x)$ with $0<\epsilon'(\tau,x)$ sufficiently small and an appropriate family of diffeomorphisms $\Phi_{\tau,z}:U(z, \tau)\rightarrow U(\Phi_{\tau,z}(z), \tau')$. Then one arrives at paths $c_i:[0,1]\rightarrow \bigcup_{\tau \in [0,1]}(L^i_\tau\setminus S_{1,0}(L^i_\tau))\cap Q^i_{\tau,x(1)}$, so that $c_i(\tau)\in L^i_\tau\setminus S_{1,0}(L^i_\tau)$ and $c(0)=c(1)\in L_0^i$, we leave the details to the reader.\end{proof}
Trying to use the above Proposition to prove the analogue of Claim (1.) in the proof of Proposition \ref{keylemma} for the situation of the family of pairs $(L_\tau^i, \Lambda_{s^k_i})$ instead of $(L_\tau^i, \Lambda_{s^k_i})$ (where we used (3.) of Assumption \ref{ass4} resp. (1.) of Proposition \ref{keysing}) one is confronted with the problem that $L_\tau^i$ is not simply connected anymore for $n\geq 2$ for $i\neq 1$. Instead, one has to deal with cohomologies of the type $H^1(L_\tau^i, \mathbb{Z})$ which are in general non-trivial. To remedy this situation, we conjecture that the latter can be expressed as a relative cohomology of an appropriate complex with $log$-type singularities along $\mathcal{Z}_{i,\tau}$, the problem here being of course that $L_\tau^i$ is no complex variety. In any case we point out to the reader that the above Proposition \ref{foldstab} cannot be proven using global generating functions (for subgraphical varieties) as are used in the proof of Claim (1.) in the proof of Propostion \ref{keylemma} using (2.) in Assumption \ref{ass4}, which is why we introduced already in the proof of Proposition \ref{keylemma} the alternative approach using the more restricted Assumption \ref{ass4}. In general, we expect to be able to eventually relax the condition on 'vanishing of higher singularities' in the situation for general spectral elements by sharpening the result of Proposition \ref{keysing} on the minimal number of connected components of $L^i_\tau\setminus S_{1,0}(L^i_\tau)$, for instance if there are for one $\tau\in [0,1]$ $3$ connected components of $L^i_\tau\setminus S_{1,0}(L^i_\tau)$ with consecutive 'self-indexing' Morse indizes, then it is sufficient to assume $\Sigma_2 $-nonsingularity in the sense of \cite{entov}.

\section{Mean curvature form and bounding disks}
\subsection{Lagrangians and bounding disks}\label{boundingdisks}
In this Section, we will introduce the additional Assumption \ref{ass5} which allows to define a Lagrangian $Q\subset X^k$, where $X^k$ is the $k$-fold cyclic covering of $X$ (see (\ref{betacov2})) by symplectically parallel transporting $Q_x$ in $X^k$ around $S^1_\epsilon$, furthermore it allows for the existence of a closed horizontal curve in $Q$ by the second condition in (\ref{horizontality}) below. This will allow for a proof of Theorem \ref{theorem34} without the introduction of the family of $(n,0)$-forms as in Lemma \ref{relativen} of the previous section. Furthermore the technique introduced here (using Assumption \ref{ass5}) will be used in the Sections \ref{generalpol}. We will also describe in the next Section \ref{app3} how a combination of these techniques and those of Section \ref{relativen0} can be used to prove Theorem \ref{theorem34} when condition \ref{ass5} does not hold {\it and} the triviality property on the family of $(n,0)$-forms introduced in Lemma \ref{relativen} is relaxed in a specific sense. In any case, we assume for the following while keeping the notation from Section \ref{relativen0}:
\begin{ass}\label{ass5}
There is an oriented closed Lagrangian submanifold $Q_x\subset M$, that is ${\rm dim} Q_x=n$ and $\omega|TQ_x=0$, so that $Q_x$ satisfies the conditions of Assumption \ref{ass2} and in addition
\begin{enumerate}
\item Let $\rho \in {\rm Symp}(M,\partial M,\omega)$ be the symplectic monodromy of $Y_\epsilon:=X|_{S^1_\epsilon}$. If $[\rho^k]=[Id_M] \in \pi_0({\rm Symp}(M,\partial M,\omega))$ for some $k >0, \ k \in \mathbb{N}$, then the isotopy $\rho^k_t \in {\rm Symp}(M,\partial M,\omega), \ \rho_1^k=\rho^k, \rho_0^k=Id_M$ and $Q$ can be chosen so that
\begin{equation}\label{horizontality}
\begin{split}
&{\rm (a)}\quad \rho_t^k(Q_x)\subset Q_x,\\
&{\rm (b)}\quad {\rm there \ is\ }z_0\in Q_x{\rm \ s.t.\ }\ \rho_t^k(z_0)=z_0,
\end{split}
\end{equation}
for any $t \in [0,1]$.
\end{enumerate}
\end{ass}
As in the previous section, assume from now on that the symplectic monodromy of the bundle $X^k|S^1_\epsilon$, namely $\rho^k=\rho^k_\epsilon \in \pi_0({\rm Symp}(M,\partial M,\omega))$, is trivial, then from Assumption \ref{ass5} one infers that there is a Lagrangian $Q_x\subset M$ so that $\rho^k(Q_x)\subset Q_x$. Assume furthermore that we have chosen the 'reference fibre' $M=\tilde X_x$ in $Y_\epsilon$, representing a fixed fibre $M$ in $\tilde X^k|S^1_\epsilon=:Y^k_\epsilon$. Then we have in analogy to Lemma \ref{bla456}
\begin{lemma}\label{bla4562}
Let $Q_y:=\mathcal{P}^{\Omega^k}_{x,y}(Q_x)$ be the Lagrangian submanifold of $\tilde X^k_y$ induced by parallel transport of $Q_x\subset M$ determined by Assumption \ref{ass2} resp. \ref{ass5} along a circle segment in $Y_\epsilon \rightarrow S^1_\epsilon$ connecting $x,y \in S^1_\epsilon$. Then the union $Q:=\bigcup_y Q_y,\ y\in S^1_\epsilon$ is a $n+1$-dimensional submanifold of $\tilde X^k_\epsilon$ and one has
\[
\Omega^k|TQ=0,
\]
i.e. $Q$ is Lagrangian. Let $dz_0\wedge \dots\wedge dz_n$ be the canonical $(n+1,0)$-form on $\mathbb{C}^{n+1}$, restricted to $\tilde X$ and consider its pullback to $\tilde X^k$ by $\pi_k$. Let $\{e_i\}, i=1,\dots,n$ be an oriented orthonormal bais of $Q$, let for each $i$, $u_i=1/2(e_i -iJe_i)$ and let $\{u_i^*\}$ be the associated dual basis. Write locally 
\begin{equation}\label{bla3472}
\pi_k^*(dz_0\wedge\dots\wedge dz_n)|Q=e^{i\theta} (u_0^*\wedge\dots\wedge u_n^*)=:e^{i\theta}\kappa_Q,
\end{equation}
for some function $e^{i\theta}:Q\rightarrow S^1$ (note that since $Q$ is oriented, $\kappa_Q$ is a well-defined $(n+1,0)$-form on $T\tilde X^k|Q$. Then since $H^1(Q_y,\mathbb{C})=0, \ y \in S^1_\epsilon$, $\theta$ lifts to a well-defined function $\theta_y:Q_y\rightarrow \mathbb{R}$, while on $Q$ one has a smooth function $\theta:Q\rightarrow \mathbb{R}/\mathbb{Z}$ satisfying (\ref{bla3472}). Let $X_{f^k} \in \Gamma(T^{(1,0)}\tilde X^k)$ s.t. $df^k(X_{f^k})=1$, then one has for any $y=e^{2\pi it} \in S^1_\epsilon$
\begin{equation}\label{winding-s2}
[s^k_{y(t)}]:= [s^k|_{y(t)\in S^1_\epsilon}]=\frac{1}{c}\int_{Q_{y(t)}}e^{i\theta}i_{X_{f^k}}\kappa_Q\cdot [s^k_x]_{||}(y(t))=\frac{1}{c}\alpha(t)\cdot [s^k_x]_{||}(y(t))=e^{2\pi i \gamma t}\cdot[s^k_x]_{||}(y(t)).
\end{equation}
where $[s^k_x]_{||}\in \Gamma({\bf H}^n(Z^k,\mathbb{C}))$ is the parallel section which coincides at $x \in S^1_\epsilon$ with $s^k|_x$, $c \neq 0$ is determined by 2. in Assumption \ref{ass2} and 
\begin{equation}\label{maslovindex37}
{\rm wind}(\alpha)=\gamma =m(\sum_i \beta_i-\beta)\in \mathbb{Z}.
\end{equation}
\end{lemma}
\begin{proof}
That $Q$ is Lagrangian is immediate from the fact $H_{\Omega^k}$ is defined as the annihilator of the vertical bundle and the fact that, by construction, $H_{\Omega^k}\cap TY_\epsilon \subset TQ$, that $Q$ is well-defined as a closed Lagrangian submanifold of $\tilde X^k$ is implied by (\ref{horizontality}) in Assumption \ref{ass5}. The rest of the proof is identical to that of Lemma \ref{bla456}.
\end{proof}
{\it Remark.} Since, again, the constant $c\neq 0$ will not be of importance in the following, we will set its value to $c=1$ in all subsequent calculations.
\begin{folg}\label{tildetheta}
Let $\sigma_Q=i_H\Omega^k$ be the mean curvature form of $Q$ ($H$ is the mean curvature vector field on $Q$). Fix $x \in S^1_\epsilon$ and choose any fixed $z' \in \tilde Q_x$. Then for any any $y \in S^1_\epsilon$ and $z \in Q_y$ choose a path $c(z',z)$ connecting $z'$ to $z$ in $Q$, thus representing an element $\tilde z$ of the universal covering $\tilde Q$ of $Q$ projecting to $z$. Then define
\begin{equation}\label{covtheta12}
\tilde \theta(\tilde z) = \theta(z') +\int_{c(z',z)}\sigma_Q,
\end{equation}
where $\theta(z')\in [0,2\pi)$. Then $\tilde \theta:\tilde Q \rightarrow \mathbb{R}$ lifts $\theta:Q \rightarrow \mathbb{R}/2\pi\mathbb{Z}$, that is one has a commuting diagram
\begin{equation}\label{betacov35}
\begin{CD}
\tilde Q @>>\pi> Q \\
 @VV \tilde \theta V    @VV \theta V \\
\mathbb{R}  @>>> \mathbb{R}/2\pi\mathbb{Z},
\end{CD}
\end{equation}
\end{folg}
\begin{proof}
The proof follows from the fact that
\[
\sigma_Q=d\theta
\]
on $Q$, that is $H=J\tilde {d\theta}$, where $\tilde{\cdot}$ here denotes metric duality $T^*Q\rightarrow TQ$ and $H$ the mean curvature vector field and is proven in \cite{thomas}, see also \cite{cielie}. 
\end{proof}
\begin{lemma}\label{extension}
There is a smooth {\it extension} $X^k_e$ of $\tilde X^k$ to the unpunctered disk
\[
f^k_e:X_e^k\rightarrow D_\delta 
\]
and an extension $\Omega^k_e$ of $\Omega^k$ to $X^k_e$  so that $(X^k_e,f_e^k, \Omega^k_e)$ defines an exact symplectic fibration which coincides with $(\tilde X^k,f^k, \Omega^k)$ over $D_{[\epsilon,\delta]}$ and such that in a neighbourhood $U_0$ of $0 \in \mathbb{C}$ one has for some $\epsilon>r>0$
\[
X_e^k|_{U_0}=D_r\times M,\quad \Omega^k_e|_{U_0}=\pi_0^*(\beta)+\omega,
\]
where $\pi_0$ is the trivial projection onto $D_\delta$ and $\beta$ is the canonical symplectic form on $D_\kappa$, and  $f^k_e| D_{[0,\epsilon]}$ can be chosen to be flat near the vertical boundary $X_e^k|S^1_\epsilon$. Furthermore, one can extend the complex structure on $\tilde X^k$ to a nearly complex structure on $X_e^k$ compatible  with $f_e^k$ in the sense of (\ref{compatible-j}), such that it is the product complex structure $\pi_0^*(j) \times J$ over $U_0$, for some complex structure $J$ on $M$, such that $X_e^k|_{U_0}$ is Kaehler.
\end {lemma}
\begin{proof}
We define $\tilde X^k_e$ as the union glued along their common (vertical) boundaries
\[
\tilde X^k_e=X_0 \cup X_1 \cup X^k_2\cup \tilde X^k
\]
where for $0<r<s<\epsilon<\delta$
\begin{equation}\label{extension3}
\begin{split}
f_0:X_0&= D_r \times M\rightarrow D_r ,\quad f_1:X_1=\frac{M\times [r,s]\times [0,1]}{(x,t,0)\sim (x,t,1)}\rightarrow D_{[r,s]}\\
f_2^k:X^k_2&=\frac{M\times [s,\epsilon]\times [0,1]}{(x,t,0)\sim (\rho^k_t(x),t,1)}\rightarrow D_{[s,\epsilon]},
\end{split}
\end{equation}
where here, $\rho_{(\cdot)}^k:[s,\epsilon]\times M\rightarrow M$ is an isotopy as in Assumption \ref{ass2} (modulo parametrization), i.e. $\rho^k_s=id_M$, $\rho^k_\epsilon=\rho^k$, $\rho^k_t \in {\rm Symp}(M,\partial M;\omega)$ for any $t \in [s,\epsilon]$. Note that since $M$ is equipped with an exact symplectic form $\omega$ which is exact, i.e. $[\omega,\alpha]=0$ and since $\rho^k_t \in {\rm Symp}(M;\partial M;\omega)$ fixes a neighbourhood of the boundary, $X_2^k$ is equipped with a family of fibrewise symplectic forms $\omega_x$ and contact forms over the fibrewise boundary $\alpha_x$ so that $[\omega_x,\alpha_x]=0$ and so that there is a trivialization of a neighbourhood of the boundary as in (\ref{boundtriv}). This implies (see Gotay et al. \cite{gotay}) there is at least locally on the base a smooth family of forms $\theta_z \in \Lambda^1(T^vX_2^k)_z,  z \in U\subset  D_{[\delta,\epsilon]}$ ($U$ open) so that $d\theta_z|(X_2^k)_z=\omega_z$, for any $x \in D_{[s,\epsilon]}$.Then defining locally
\[
\Theta_U(x)=(f_2^k)^*(\beta)+\theta_z, \ x \in (X_2^k)_z,\ z \in U,
\]
for some $1$-form $\beta$ on the base and setting $\Omega_U=d\Theta_U$ trivializes $X_2^k$ over $U$ by using parallel transport along the annihilator of $T^vX_e^k$. Then for different $U$, the transition mappings are fibrewise symplectomorphisms which fix the boundary, since  $(d\Theta_U)_z=\omega_z, z \in U$. This implies that $X_2^k\rightarrow D_{[r,s]}$ is in fact locally trivial with structure group ${\rm Symp}(M,\partial M,\omega)$. Now we have to show that $X_e^k$ actually carries a closed non-degenerate $2$-form that restricts firewise to the family given by $\omega_z, \ z \in D_\delta$. To prove this assume for the moment, that $X_e^k$ is actually well-defined as a symplectic fibration by glueing the above parts along their common boundaries. That this is the case will be proven below. Assuming this note that $M$ is $(n-1)$-connected i.e. simply connected and that this holds also for the base, $D_\delta$. Then, by Gotay et al. (\cite{gotay}, Theorem 2) since $H^2(M,\mathbb{C})=0$, there is a cohomology class in $H^2(X_e^k,\mathbb{C})$ extending $[\omega_z]$ for any $z \in D_\delta$. By repeating the above procedure over a suitable open covering $\mathcal{U}$ of $D_\delta$ and using a partition of unity one can construct a smooth closed two-form $\Omega$ on $X_e^k$ realizing this class (see again \cite{gotay}). Since the family $\theta_z, z \in D_\delta$ is constant in a neighbourhood of the 'horizontal' boundary of $X_e^k$ w.r.t. its natural trivialization, the symplectic parallel transport is well-defined (see Def. \ref{symplfibr}). So there is a globally defined closed $2$-form $\Omega$ defining a horizontal distribution $H_{\Omega}$ in the sense of Lemma \ref{symplfibr} and restricting fibrewise to $\omega$. Defining
\[
\Omega_2^k=c(f_2^k)^*(\alpha)+\Omega
\]
where $c>0$ big enough and $\alpha$ is the canonical Kaehler form on the base, $X_2^k$ carries the structure of an exact symplectic fibration in our sense.\\
Now what remains to be shown is that the above glueing operations are well-defined, i.e. the symplectic structure on the objects are preserved. Assume first that $X^k_2$ (analogously $\tilde X^k$) is 'flat' along their boundary, that is the boundary of $X_2^k$ is locally symplectomorphic to $Y_s\times [0,\kappa_1] \cup [\epsilon-\kappa_2,\epsilon]\times Y_\epsilon$ for some $\kappa_{1,2}>0$ equipped with the product structures ($Y_r:=X_2^k|S^1_r$). Then the well-definedness of the glueing follows since the monodromies along the boundaries $Y_s$ resp. $Y_\epsilon$ of $X_2^k$ coincide with those of the (corresponding) boundaries of $X_1$ resp. $\tilde X^k$ and the identification is smooth. Now it is easy to see that one can choose a smooth monotone bijection $\tau:[0,1]\rightarrow [0,1]$ s.t. $\tau([0,\epsilon_1))=0,\ \tau((1-\epsilon_2,1])=1$ for some small $\epsilon_1,\epsilon_2 >0$ and that replacing $\rho^k_t, t \in [0,1]$ by $\rho^k_\tau(t), t\in [0,1]$ in  the definiton of $X_2^k$ means to extend a neighbourhood of the boundaries of $X_2^k$ to be 'flat' in the above sense so that the glueing is well-defined. Taking the product symplectic structures on $X_0$ resp. $X_1$ one arrives at the 
assertion. Note that we only showed that the fibrewise symplectic structure defined by $\Omega_2^k$ on $X_2^k$ and the family of horizontal subspaces associated to $\Omega_2^k$ glue with the corresponding strcuture on $\tilde X^k$, which is all that we will need for the subsequent discussion.\\
Concerning the assertion about the almost complex structure on $X_e^k$, first note that one has the diagram
\begin{equation} \label{quotdiag2}
\begin{CD}
 \tilde X^k_2  @>>\simeq> \frac{[s,\epsilon]\times [0,1]\times M}{(t,0,x)\sim (t,1,\rho^k_t(x))} @<<p< 
 [s,\epsilon]\times [0,1]\times M\\
@VV f_2^k V           @VV\pi V                                       @VV\pi_0V  \\
 D_{[s,\epsilon]} @>>id > D_{[s,\epsilon]} @<<\overline p<  [s,\epsilon]\times [0,1],
\end{CD}
\end{equation}
where $p$ and $\overline p$ are the obvious quotient mappings, $\pi_0: [s,\epsilon]\times [0,1]\times M\rightarrow [s,\epsilon]\times [0,1]$ is the trivial projection and $\pi$ is {\it defined} so that the diagram gets commutative. We first define a vertical nearly complex structure on $X_2^k$ that matches the family of complex structures on $Y_\epsilon$ by defining for $(t,\tau,x) \in [s,\epsilon]\times [0,1]\times M$ $J^v\in {\rm End}(T^vX_2^k)$ as 
\begin{equation}\label{vertcomplex}
(J_2^v)_{(t,\tau,x)}:= ((\rho^k)_{\epsilon-\tau(\epsilon-t)}\circ (\rho^k_\epsilon)^{-1})_*(J_{\tau,x})	
\end{equation}
Here $J_{\tau,x} \in {\rm End}(TM)$ is the family of vertical almost complex structures on $[0,1]\times M$ induced by the given family on $Y_\epsilon \simeq p(\{\epsilon\}\times [0,1]\times M)$ which satisfies $J_{0,x}=\rho^k_*(J_{1,x})$. Furthermore we define $J^h\in {\rm End}(T^hX^k_2)$ so that  $(f_2^k)_*(J^h)=j$,  where $j$ is the canonical complex structure of $D_{[s,\epsilon]}$. Then we set as a candiate for $J_{X_2^k}$
\begin{equation}\label{complex}
J_2:=J|_{X_2^k}=J_2^v\oplus J_2^h.
\end{equation}
by construction, this extends to a smooth almost complex structure on $X_2^k\cup \tilde X^k$. To extend this to the whole $X_e^k$, note first that $X_e^k|D_s$ is a trivial fibre bundle, so the above defined family $J^v$, restricted to $Y_s$ is given by a family of almost complex structures $J^v_t \in {\rm End}(TM),\ t \in S^1$ being compatible in the sense of Ass. \ref{compatiblecplx} (this follows from the definition since $\rho_t \in {\rm Symp}(M,\partial M,\omega)$). Now it is well-known (\cite{duff}) that the space of such structures $\mathcal{J}(M,\omega)$ is contractible since it is isomorphic to the space of sections of a bundle over $M$ with contractible fibres $Sp(2n,\mathbb{R})/U(n)$ (i.e. it is simply connected). Note that this fact is a priori formulated for a closed symplectic manifold, but since in a neighbourhood of $\partial M$, $J$ is compatible with $j$ in the sense of Ass. \ref{compatiblecplx}, one can argue as in Seidel \cite{Seidel3} to deduce that the space of in that sense compatible complex structures $\mathcal{J}(M,\omega,j)$ is contractible. So we can choose a path of loops $\tau \mapsto (J_1)_{\tau,t}, t \in S^1,\tau \in [r,s] $ of $j$-compatible complex structures on $M$ so that
\[
(J_1^v)_{r,t}=J_{t_0},\quad (J_1^v)_{s,t}=(J_2^v|_{Y_s})_t,
\]
where $J_2^v$ is as constructed above and $t_0 \in S^1$ is any fixed value. Then one defines $J$ restricted to $X_1$ as $J_1=J_1^v\oplus j$, where $j$ is the complex structure on $D_{[r,s]}$. It is now clear how to extend $J$ to $D_r$ constantly so that the resulting alomost complex structure $J$ is smooth on $X^k_e$ and induces a Kaehler metric over $D_r$ relative to the product symplectic structure $\omega + (f_e^k)^*(\beta)$ (this last claim follows since $(J_{t_0},\omega)$ is Kaehler on $M$).
\end{proof}
{\it Remark.} Since in the following, it will be sufficient to assume that the $2$- form $\Omega_2^k$ on $X_2^k$ as defined above is closed on any $Y_t, \ t\in [s,\epsilon]$, we can give an alternative construction by defining a family of horizontal subspaces $H_{\Omega_2^k}\subset T(M\times [s,\epsilon]\times[0,1])$ by
\begin{equation}\label{hzweik}
H^k_{2}(x,t,\tau)={\rm span}_{X\in H_t,\ Y\in H_\tau}\{(\tau\frac{d}{dt}(\rho^k_t)(x)(X), X,Y)\},
\end{equation}
where for any $(x,t,\tau)$, $H_t,\ H_\tau\subset T(M\times [s,\epsilon]\times [0,1])$ are the subspaces spanned by $((0,x),(1,t),(0,\tau))$ and $((0,x),(0,t),(1,\tau))$, respectively. It is easy to see that $H_2^k$ factorizes to a horizontal subspace on $T(X_2^k)$, again denoted by $H_2^k$ and we can define a fibrewise closed two-form $\tilde \Omega_2^k$ on $X_2^k$ restricting fibrewise to the family $\{\omega_x\}, x\in D_{[s,\epsilon]}$ by pulling back $\alpha$ as defined above to $H_2^k$:
\[
\tilde \Omega_2^k:=c(f_2^k)^*(\alpha)\oplus_{H_2^k}\tilde \omega,
\]
where $c>0$, $\tilde \omega$ is the $2$-form on $T^v(X_2^k)$ induced by the family $\{\omega_x\}$ and the splitting is defined so that $H_2^k$ becomes the symplectic complement of $T^v(X_2^k)$ in $T(X_2^k)$. Note that $\tilde \Omega_2^k$ will in general not glue smoothly with $\Omega^k$ along the vertical boundary of $\tilde X_k$, while using the construction of the last proof, one sees that $H_2^k$ and the family of vertical symplectic forms $\omega_x$ defined by $\Omega_2^k$ glue smoothly along the common boundary of $X_2^k$ and $\tilde X_k$, which will be sufficient for the subsequent construction, in fact we will assume to have chosen $H^k_2$, the associated symplectic form and the compatible nearly complex structure induecd by (\ref{complex}) in the way described here, leading to a subemrsion metric $g= \tilde \Omega_2^k(\cdot,J\cdot)$ on $X_2^k$.\\
Consider the function $\alpha:S^1_\epsilon\rightarrow \mathbb{C}^*$ determined in (\ref{winding-s2}):
\begin{equation}\label{windfunc}
\alpha(y)=\int_{Q_y}e^{i\theta}i_{X_{f^k}}\kappa_Q, \ y\in S^1_\epsilon.
\end{equation}
Since in the following, we will be only interested in the {\it winding number} of this function, we can reduce $\alpha$ to a more simple form as long as this winding number is preserved. For this note at first that the condition (\ref{horizontality}) in Assumption \ref{ass2} implies the following:
\begin{lemma}\label{horsection}
There is a section $u:D_\epsilon\rightarrow X_e^k$ so that $F:={\rm im}(u)$ has boundary in $Q$, and is horizontal in a neighbourhood of its boundary, that means that $\partial F=u(\partial D_\epsilon)\subset Q$ and 
\begin{equation}\label{horizontality2}
Du_z(TD\epsilon)=(T(X_e^k)^h)_u(z), \quad z \in (-r,0]\times \partial  D_\epsilon,
\end{equation}
for some $r>0$. Furthermore, if we choose the horizontal distribution $H_2^k$ as constructed in (\ref{hzweik}), then (\ref{horizontality}) holds for any $z \in D_\epsilon$, that is, $u$ is horizontal.
\end{lemma}
\begin{proof}
Go back to the proof of Lemma \ref{extension} and recall that $X_e^k|D_{[s,\epsilon]}$ for some $s >0$ was defined by 
\[
X^k_2=\frac{M\times [s,\epsilon]\times [0,1]}{(x,t,0)\sim (\rho^k_t(x),1)}\rightarrow D_{[s,\epsilon]},
\] 
where we assumed the structure being 'flat' near the boundary over $S^1_\epsilon$, which is satisfied by assuming that $\rho_t^k=\rho^k$ for $t \in [1-r,\epsilon]$ for some small $s>r>0$. Let now $z_0 \in Q \subset M$ be as in  Assumption \ref{horizontality}, being fixed by $\rho^k_t$ for any $t \in [s,\epsilon]$ (note the change of parametrization). Then defining 
\[
\tilde u: D_\epsilon \hookrightarrow \tilde D= D_{[0,\epsilon]}\times \{z_0\}\subset D_{[0,\epsilon]}\times M
\]
this factorizes to a well-defined map $u:D_\epsilon \rightarrow X_e^k$ with image $ D\subset X_0\cup X_1\cup X_2^k =X_e^k|D_{[0,\epsilon]}$ (using notations from the proof of Lemma \ref{extension}). From the 'flatness' of $X_e^k|D_\epsilon$ near its vertical boundary one concludes that $u((-r,0]\times \partial D_\epsilon)$ is horizontal as required. Furthermore $u(\partial D_\epsilon)\subset Q$ by construction of $Q$ and $u$. The last assertion follows from the obvious fact that if $z_0$ is as in Assumption (\ref{horizontality}), then $\frac{d}{dt}(\rho^k_t)(z_0)=0$.
\end{proof}
We now observe that the pointwise 'phase' $\theta$ on $Q$, lifted to a real-valued function on $\tilde Q$, splits into a sum of a fibrewise constant part $e^{i\vartheta}:S_\epsilon^1\rightarrow S^1$, defined along the disk constructed in the preceeding Lemma having boundary in $Q$ and a 'vertical' part which is determined by the fibrewise restriction of the mean curvature form of $Q$, $\sigma_Q$.
\begin{lemma}\label{tildetheta3}
Let $\sigma_Q$ be the mean curvature form of $Q$. Fix $x \in S^1_\epsilon$ and define $z' \in \tilde Q_x$ by $z'=u(x)$, where $u:D_\epsilon \rightarrow X_e^k$ is as in Lemma \ref{horsection}. Then for any $y \in S^1_\epsilon$ and $z \in Q_y$ choose a path $c(z',z)$ connecting $z'$ to $z$ in $Q$ in the way that 
\begin{equation}\label{pathdecomp}
c(z',z)(t)= c_h(z', z'')*c_v(z'', z),
\end{equation}
where $c_h$ is the path connecting $z'$ to $z'':=u(y)$ in $u(\partial D_\epsilon)=\partial F$ which projects to a path $\tau \mapsto e^{2\pi i\tau}x, \tau \in [0,t] ,0\leq t<1$ s.t. $y=e^{2\pi i t}x$ and $c_v(z'', z)$ is any path in $Q_y$ connecting $z''$ to $z$. Then $c(z', z)$ represents an element $\tilde z$ over $z$ in $\tilde Q$ and one has
\begin{equation}\label{covtheta}
\tilde \theta(\tilde z) = \theta(z')+\int_{c_h(z',z'')} d\vartheta/2 +\int_{c_v(z'',z)}\sigma_Q,
\end{equation}
where $\theta(z') \in [0,2\pi)$, $\tilde \theta$ denotes the lift $\tilde \theta:\tilde Q\rightarrow \mathbb{R}$ defined in (\ref{covtheta12}) and $\tilde Q$ denotes the universal covering of $Q$. Here, $\vartheta:u(S^1_\epsilon) \hookrightarrow \partial F \rightarrow \mathbb{R}/2\pi\mathbb{Z}$ is determined by the horizontal section
\[
u:D_\epsilon\rightarrow X_e^k,\quad u(\partial D_\epsilon)\subset Q,
\]
where $u$ is as constructed in (the proof of) Lemma \ref{horsection} and by a trivializing section $\kappa_F$ of unit length of $\Lambda^{n,0}(T^*X_e^k)|F$ over $u(D_\epsilon)=:F$ by the requirement that along $\partial F$:
\[
\kappa_Q^2=e^{i\vartheta}\kappa_F^2.
\]
\end{lemma}
{\it Remark.} As $F$ is connected and has non-empty boundary, one has $H^2(F,\mathbb{Z})=0$, hence $\Lambda^{n,0}(T^*X_e^k)|F$ is trivial and allows the choice of a trivializing section $\kappa_F$. That the Maslov class
\[
\mu_F=-\frac{1}{2\pi}\int_{\partial F}d\vartheta
\]
is actually independent of the choice $\kappa_F$ follows from Stokes' Theorem since one has for $\kappa_Q^2=e^{i\vartheta'}{\kappa'_F}^2$ that $e^{i\vartheta'}=e^{i\phi} e^{i\vartheta}$ for some function $e^{i\phi}:F \rightarrow S^1$, so $\int_{\partial F}d\vartheta=\int_{\partial F}d\vartheta'$, hence we have a map $\mu_F:H_2(X_e^k, Q)\rightarrow \mathbb{Z}$.
\begin{proof}
We know from Cieliebak et al. \cite{cielie} (eq. (3)) that on $\partial F$
\begin{equation}\label{cieliebla}
2i\sigma_Q=\eta_F+i d\vartheta,
\end{equation}
where $\eta_F$ is a the $1$-form over $u(D_\epsilon)$ defined by $\nabla \kappa^2_F=\eta_F\otimes\kappa^2_F$ along $u$, where $\nabla$ is the Levi-Civita-connection of $X_e^k$. Now we claim that since $u$ is horizontal (using $H_2^k$ as in (\ref{hzweik})) and since in an open neighbourhood $U\cap u(D_\epsilon)$, for some open $U\subset X_e^k$, of its boundary, $X_e^k|U$ is Kaehler, that
\begin{equation}\label{etafvanishes}
\eta_F(v)=0, \ v\in T(\partial F).
\end{equation}
To see this, we define a connection (compare Appendix B.) on $TX_e^k|F$
\[
\nabla^u=\nabla^h \oplus_{H_2^k} \nabla^v,\ {\rm where} \ \nabla_X^vU=[X,U],\ \nabla^v_U V =P^v(\nabla_V^L U),
\]
where $X\in \Gamma(T^hX_e^k|F),\ U,V \in \Gamma(T^vX_e^k|F)$ and $\nabla^h=(f_e^k)^*(\nabla^D)$, is the Levi-Civita connection of $D_\epsilon$ induced by the standard metric, lifted to the horizontal bundle $H_2^k\subset TX_2^k$ defined in (\ref{hzweik}). We use $\nabla^u$ to construct a trivialization $\kappa_F$ of $\Lambda^{n,0}(T^*X_e^k)|F$ by parallel transport along radial rays in $F$ w.r.t. $\nabla^u$. Note that since this connection preserves the splitting $TX_e^k|F=(T^hX_e^k\oplus T^vX_e^k)|F=TF\oplus T^vX_e^k|F$ and since by construction of $H_2^k$, $\mathcal{L}_X\omega=0$ for horizontal vector fields $X$, where $\omega$ denotes the symplectic form on $T^vX_e^k$, we can use parallel transport induced by $\nabla^u$ of a Lagrangian basis $e_0 \in T_zF$, $e_1,\dots,e _n \in T^v_zX_e^k|F$, where $z =(f_e^k)^{-1}(0)\cap F$, that is $\omega(e_i,e_j)=0$, along radial rays, then $u_i=1/\sqrt{2}(I-iJ)e_i, i=0,\dots,n$ trivializes $T^{1,0}X_e^k|F$. Further note that by construction of the family $(J_2^h)_x,\  x\in D_{[s,\epsilon]}$ and by the horizontality of $F$ we have, using the notation from the proof of Lemma \ref{extension} over $U\cap F=u(D_{[\epsilon-\delta,\epsilon]})$ for some $\delta>0$ and for $H \in T(U\cap F)$:
\begin{equation}\label{complexinv}
(\mathcal{L}_{X} J)(H)=\mathcal{L}_X (J_2^v\oplus J_2^h)(H)= \mathcal{L}_X((f_e^k)^*(j))(H)=0,
\end{equation}
where either $X \in \Gamma(T^hX_e^k|U)$ or $X \in \Gamma(T^vX_e^k|U)$. Let now $R \in \Gamma(T(X_e^k)^h)|U)$ be radially outward pointing in $T(F\cap U)$, that is, $(f_e^k)_*(R)=\frac{\partial}{\partial r}$, where $r:D_\epsilon\rightarrow \mathbb{R}^+$ is the radius function. Then, since $X_e^k|D_{[\epsilon-\delta,\epsilon]} \cap \bigcup_{z\in D_{[\epsilon-\delta,\epsilon]}} A^m_z$ ($A_z^m\subset (X_e^k)_z$ as defined in the proof of Lemma \ref{lagrangianbasis}) is Kaehler and writing locally $\kappa_F=u_0^*\wedge\dots\wedge u_n^*$ for some basis $u_i \in T^{1,0}X_e^k|U \cap F$, where we can assume by the horizontality of $F$ that $u_0=1/\sqrt{2}(e_0-iJe_0)$ for some $e_0 \in T(F\cap U)=T^hX_e^k|U$ and $u_i=1/\sqrt{2}(e_0-iJe_i), \ i=1,\dots,n$ for $e_i \in TX_e^k|U$ which satisfy (by construction of $\kappa_F$) $\nabla^u_Re_i=0, \ i=0,\dots,n$. We can then deduce for $i=0,\dots,n$:
\begin{equation}\label{radialvanishing}
\begin{split}
\nabla^L_{JR}u_i&=\nabla^L_{u_i}(JR)+[JR,u_i]=J\nabla^L_{u_i}R-J\mathcal{L}_{u_i}(R)\\
&=J(\nabla^L_Re_i-iJ\nabla^L_Re_i)=J(\nabla^ue_i-iJ\nabla^u_Re_i)=0.
\end{split}
\end{equation}
By the definition of $\kappa_F$, this implies (\ref{etafvanishes}). Here we used that by the horizontality of $F$ and the 'flatness' of $X_2^k$ in a neighbourhood of its common boundary with $\tilde X^k$, we have $\nabla_R^u=\nabla^L_R$ using Proposition 14 of \cite{mazzeo} (the second fundamental form of the fibre in the direction of $R$ and the curvature of $T^hX_e^k|U$ vanishes, here we use the submersion metric on $X_e^k$ constructed in the remark below the proof of Lemma \ref{extension}). Substituting (\ref{etafvanishes}) into (\ref{cieliebla}), noting that
\[
\tilde \theta(\tilde z) = \theta(z')+ \int_{c_h(z',z'')} \sigma_Q +\int_{c_v(z'',z)}\sigma_Q
\]
and substituting (\ref{cieliebla}) into the first integral, we arrive at the assertion.
\end{proof}
Assuming the Assumptions \ref{ass4} resp. \ref{ass5} above, we are then able to prove the following:
\begin{lemma}\label{tildealpha}
Let $F:={\rm im}(u), \ u:D_\epsilon \rightarrow X_e^k$ be the disk with boundary in $Q$ as constructed in Lemma \ref{horsection} and $\tilde \alpha:S^1_\epsilon\rightarrow \mathbb{C}^*$ be defined by
\begin{equation}\label{tildetheta2}
\tilde \alpha(y)=e^{i\vartheta/2}\int_{Q_y}i_{X_{f^k}}\kappa_Q, \ y\in S^1_\epsilon,
\end{equation}
where here, $e^{i\vartheta}:S^1_\epsilon \hookrightarrow \partial F \rightarrow S^1$ is determined as described in the formulation of Lemma \ref{tildetheta3}. Then, with the above notations,  ${\rm wind}(\alpha)={\rm wind}(\tilde \alpha)$, i.e. we have
\begin{equation}\label{diskmaslov}
{\rm wind}(\alpha)=\mu_F/2-k,
\end{equation}
where $\mu_F$ is the Maslov index along $\partial F$ in $Q$ and ${\rm wind}$ is the usual winding number of a mapping $f:S^1\rightarrow \mathbb{C}^*$.
\end{lemma}
{\it Remark.} Since $Q$ is assumed to be oriented, we could have done without squaring, the squares are to meet the convention of Cieliebak's paper (\cite{cielie}).
\begin{proof}
We can proceed exactly as in the proof of Proposition \ref{keylemma} and prove at first
\[
{\rm wind}(\alpha)={\rm wind}(e^{i\theta\circ c})-k,
\]
where $c:[0,1]\rightarrow Q_1$ is an arbitrary smooth path lifting $t \mapsto xe^{2\pi it}$ for a fixed $x\in S^1_\epsilon$ and $Q_1$ is as defined in (\ref{hatq}) resp. (\ref{hatqt}) below. To prove (\ref{diskmaslov}) resp. ${\rm wind}(\alpha)={\rm wind}(\tilde \alpha)$ if (\ref{horizontality}) of Assumption \ref{ass5} is satisfied note that in this case $Q_1$ in the proof of Proposition \ref{keylemma} is trivialized by symplectic parallel transport, coincides with $Q$ as defined in Lemma \ref{bla4562} and we can choose $\gamma_1$ in the former proof to be represented by letting $c:[0,1]\rightarrow Q$ parametrize $\partial F\subset Q$, where $F\subset X_e^k$ is as defined in Lemma \ref{horsection}. Then, by Lemma \ref{tildetheta3} ${\rm wind}(e^{i\theta\circ c})=\mu_F/2$ which concludes the proof.
\end{proof}
We can consequently deduce by the preceding Lemmata the following:
\begin{folg}\label{tildealpha2}
The winding number $\rm{wind}(\alpha)$ of $\alpha:S^1_\epsilon\rightarrow \mathbb{C}^*$ as defined above is given by the winding number of
\[
\tilde \alpha(y)=e^{i\vartheta/2}\int_{Q_y}i_{X_{f^k}}\kappa_Q, \ y\in S^1_\epsilon,
\]
where here, $e^{i\vartheta}: \partial F \rightarrow S^1$ is determined by $F$ as described in Lemma \ref{tildetheta3}. and therefore ${\rm wind}(\alpha)$ solely depends on the Maslov index associated to $F$ in $Q$ and the degree $k$ of the covering $X^k\rightarrow X$.
\end{folg}
Note that from (\ref{winding-s2}) we know that the winding number of the function $\alpha:S^1_\epsilon\rightarrow \mathbb{C}^*$ equals $m(\sum_i\beta_i-\beta)$. We will now find an isotopy of mappings $\tilde \alpha_\tau:S^1 \rightarrow \mathbb{C}^*$, $\tau\in [0,1]$ so that $\alpha_1=\alpha$ and $\alpha_0$ has winding number $-m\beta$, for this need the following family of Lagrangians in $X_e^k$.
\begin{lemma}\label{lag3}
Let $\kappa >0$ be s.t. $X_e^k|D_\kappa\simeq M \times D_\kappa$ is symplectically trivial, where $M$ is a fixed fibre as above. Then with $Q_x$ as in Assumption \ref{ass2} let $Q_\kappa=Q_x\times S^1_\kappa\subset Y_\kappa:=X_e^k|S^1_\kappa$ be the Lagrangian submanifold of $X_e^k$ induced by symplectic parallel transport along $S^1_\kappa$. Let $Q$ be as above. Then there is a smooth $n+2$-dimensional submanifold with boundary $\tilde Q\subset X_e^k$ being a cobordism from $Q_\kappa$ to $Q$ in the sense that $\tilde Q$ fibres into Lagrangian submanifolds of $X_e^k$ such that
\[
\tilde Q=\bigcup_{t\in [\kappa,\epsilon]}Q_t\subset X_e^k\ {\rm and}\ Q_t=\tilde Q \cap Y_t\ {\rm is \  Lagrangian}
\] 
for all $t \in [0,1]$ as a submanifold of $X_e^k$. Furthermore, $Q_\kappa$ is as above and $Q_\epsilon=Q$, i.e.
\[
\partial \tilde Q=Q_\kappa \cup Q_\epsilon.
\]
Finally, each $Q_t$ fibres to a family of Lagrangian submanifolds $Q_{t,y}:=Q_t\cap (X_e^k)_y$ over $S^1_t,\ t \in [\kappa,\epsilon]$.
\end{lemma}
\begin{proof}
Recall again that following the proof of Lemma \ref{extension} $X_e^k|D_{[s,\epsilon]}$ is of the form
\[
X^k_2=\frac{M\times [s,\epsilon]\times [0,1]}{(x,t,0)\sim (\rho^k_t(x),t,1)}\rightarrow D_{[s,\epsilon]},
\] 
where here again we parametrize $\rho^k_{(\cdot)}:[\kappa,\epsilon]\times M\rightarrow M$, $\rho^k_\kappa=Id,\ \rho^k_\epsilon=\rho^k$. Let now $t  \in [\kappa,\epsilon]$ (here $0<\kappa< s$) and with $x\ \in S^1_\epsilon$ being the distinguished point of Assumption \ref{ass2}:
\[
\hat Q= Q_x\times [\kappa,\epsilon]\times [0,1].
\]
This again factorizes due to Assumption \ref{ass2} to a well-defined submanifold $\tilde Q\subset X_e^k$ which, since on any $Y_t, \ t\in [\kappa, \epsilon]$, the horizontal distribution is induced from the trivial one on $M\times [0,1]$ (by Lemma \ref{mappingcylinder}) is invariant under symplectic parallel transport along $S^1_t$, for any $t \in [\kappa,\epsilon]$, which implies its restriction to $Y_t$ is Lagrangian.
\end{proof}
So we will assume in the following that we have a smooth family of Lagrangian submanifolds $Q_t, \ t \in [\kappa,\epsilon]$, so that $Q_\epsilon=Q$ and $Q_\kappa=Q_x \times S^1_\kappa$, where $Q_x$ is the Lagrangian from Ass. \ref{ass2} in $M$.
\begin{prop}\label{alphatau2}
Let $u:D_\epsilon\rightarrow X_e^k, \ {\rm im}\ u=F$ be the horizontal section constructed in Lemma \ref{horsection}. Then $F_t:=F\cap(f_e^k)^{-1}(D_t),\ t \in [\kappa,\epsilon]$ has well defined boundary $\partial F_t=Q_t\cap F$ in $Q_t$. Thus we can associate to each Lagrangian $Q_t, \ t \in[\kappa,\epsilon]$ a function 
\[
e^{i\vartheta_t}:u(S^1_t)=\partial F_t\rightarrow S^1,\quad{\rm by}\quad \eta_{F_t}^2=e^{i\vartheta_t}\eta_{Q_t}^2
\]
where $\eta_{F_t}$ is the trivializing section of $u^*(\Lambda^{(n.0)}(T^*X_e^k))$ chosen above, but restricted to $F_t$, $\eta_{Q_t}$ is the element of $\Omega^{(n+1,0)}(X_e^k)$ over $Q_t$ determined by $Q_t$. Then define 
\[
\tilde \alpha_{(\cdot)}:[\kappa,\epsilon]\times S^1\rightarrow \mathbb{C}^*
\]
by parametrizing each boundary $\partial F_t=u(D_t)\cap Q_t\subset Y_t$ by some function $\tilde y_t:S^1 \rightarrow \partial F_t$ s.t. $f_e^k(\tilde y(s))=y_t(s):= te^{2\pi is}, \ t \in [0,1]$ and writing
\[
\tilde \alpha_t(s)=e^{i\vartheta_t/2\circ \tilde y_t(s)}\int_{Q_{t,y_t(s)}} i_{X_{f^k}}\kappa_Q, \ s \in S^1,
\]
for any $t \in [\kappa,\epsilon]$. Then $\tilde \alpha_{(\cdot)}$ is a smooth isotopy of maps $\alpha_{t}:S^1\rightarrow \mathbb{C}^*$ and one has
\begin{equation}\label{windingn}
{\rm wind}(\tilde \alpha_\kappa)=0,\quad {\rm wind}(\tilde \alpha_\epsilon)=m(\sum_{i=0}^{n}\beta_i-\beta),
\end{equation}
i.e. we infer (see Lemma \ref{tildealpha}) that $\mu_F=\mu_{F_\epsilon}=m\beta$.
\end{prop}
{\it Remark.} In view of Lemma \ref{condition}, the equalities (\ref{windingn}) prove Theorem \ref{theorem34}.
\begin{proof}
The assertions are actually (more or less) clear despite of the asserted values of the winding numbers.
For this let $t=\epsilon$. Then we showed that modulo isotopy $\tilde \alpha$ appears in the expression for $[s^k]$ in Lemma \ref{bla456}:
\begin{equation}
[s^k|_{y\in S^1_\epsilon}]=\int_{Q_{y}}e^{i\theta}i_{X_{f_e^k}}\kappa_{Q}\cdot [s^k_x]_{||}(y)=\alpha(y)\cdot [s^k_x]_{||}(y)
\end{equation}
By the definition of $s^k$ (recall $k=m\beta, m\in \mathbb{N}$) one has $\mathcal{P}_c^*s^k|y=e^{2\pi i m(\sum_{i=0}^{n}\beta_i-\beta)t}s^k|y$, where $c:[0,1]\rightarrow S^1_\epsilon,\ c(t)=\epsilon e^{2\pi i t}$. Since $[s^k_x]_{||}(y(t)), y\in S^1_\epsilon$ is invariant under parallel transport we conclude that
\[
\alpha(ye^{2\pi i t})=\alpha(y)e^{2\pi im(\sum_{i=0}^{n}\beta_i-\beta)t},
\]
which gives the assertion for $t=\epsilon$. For $t=\kappa$ note that since we haven chosen $\kappa>0$ so that $X_e^k$ is trivial as a symplectic fibration $X_e^k|D_\kappa=M\times D_\kappa$ with the constant family of complex structures compatible with $\omega$ it is Kaehler, i.e. one has by Cieliebak \cite{cielie}
\[
2i\sigma_{Q_\kappa}=\eta_{F_\kappa}+i d\vartheta_\kappa,
\]
but $\eta_{F_\kappa}|{\partial F}=0$ by the horizontality of $F|M \times D_\kappa$ (see the proof of Lemma \ref{tildetheta3}). Now let $\{e_i\}$ be a local orthonormal basis on $Q_\kappa$ which is invariant under symplectic parallel transport along $\partial F_\kappa$ and adopted to the splitting $TX_e^k=(TX_e^k)^h\oplus (TX_e^k)^v$ (we can achieve the former since the family of vertical complex structures over $X_e^k|D_\kappa$ was chosen constant in \ref{extension}). Since $\nabla^g_{e_0}(Je_0)=Je_0/\kappa$ for $e_0\in T^hX_e^k\cap T(\partial F_\kappa)$ of unit length (rotating counterclockwise) we get for $\sigma_{Q_\kappa}(X^h)$, where $X_h=2\pi k\kappa\cdot e_0$, that is $X_h=c'$, where $c(t)=ye^{2 \pi it}, \ y\in S^1_\kappa, t \in [0,1]$ (setting $g(\cdot,\cdot)=\Omega_e^k(\cdot,J\cdot)$)
\[
g(\sum_{i=0}^n\nabla^g_{e_i}e_i,JX^h)= k - \sum_i g([e_i, JX^h], e_i),
\]
where we assume that the $e_i, i=1,\dots,n$ are vertical. Now since $J|(TX_e^k)^h=(f_e^k)^*(j)$, one has $L_{e_i}J=0$ for vertical $e_i$, on the other hand since $L_{X^h}e_i=0$, we arrive at
\[
\frac{1}{2\pi}\int_{\partial F_\kappa}\sigma_{Q_\kappa} = \frac{1}{2\pi}\int_{\partial F_\kappa}d\vartheta_\kappa/2=k,
\]
It remains to evaluate the parametrized integral 
\[
\int_{Q_{t,{y_t(s)}}}i_{X_{f_e^k}}\kappa_{Q_\kappa}, \ s \in S^1.
\]
Let for any fixed fibre $(Y_\rho)_x,\ x\in S^1_\rho$ ($\rho$ chosen as above) the set $\{e_1,\dots, e_n\}$ be a local orthornomal frame of $T(Y_\rho)^v_x$ and let $e_0 \in TY_\rho^h$ be of length one. Then (the trivial) parallel transport around $S^1_\rho$ in $Y_\rho=S^1_\rho\times M$ gives a basis of $T^{n+1,0}(S^1_\rho\times U)$ for some open set $U\subset M$, $\{u_i=\frac{1}{\sqrt{2}}(e_i-iJe_i)\}$, so that locally
\begin{equation}\label{insertk}
i_{X_{f_e^k}}\kappa_{Q_\kappa}=i_{X_{f_e^k}}(u_0^*\wedge \dots \wedge u_n^*).
\end{equation}
But the function $x\mapsto u_0^*(X_{f_e^k})_x,\ x\in S^1_\kappa$ has clearly winding number $-k=-m\beta$ (see the proof of Lemma \ref{tildealpha}), which proves the Proposition.
\end{proof}

\subsection{Bounding disks and relative $(n,0)$-forms}\label{app3}
In the following, it will be described how the methods of Section \ref{boundingdisks} in Chapter \ref{chapter2sympl} can be modified to prove Theorem \ref{theorem34} {\it without} the Assumption \ref{ass5} (1.) by invoking a weaker form of Lemma \ref{relativen} in Section \ref{relativen0}. To be more precise we will assume in the following (using the notation of Section \ref{boundingdisks}) and setting $Y^k:=(f^k)^{-1}(S^1_\epsilon)$). Recall that by Lemma \ref{milnorsymplectic} $Y^k$ is equipped with a structure of an exact symplectic fibration $(Y^k,\Omega^k)$ with a compatible vertical almost complex structure $J^v$. 
\begin{ass}\label{relativenlocal}
For any $y \in S^1_\epsilon$ there is a neighbourhood $U\subset S^1_\epsilon$ of $y$, a family of vertical compatible almost complex structures $\tilde J^v_{U,x}, \ x \in U$ on $Y^k$ and a non-vanishing section $\mathbf{s}_U^k \in \Gamma(\Lambda^{n,0}(T^*Y^k|U)^v)$ (with respect to $\tilde J^v_{U}$) so that 
\begin{equation}\label{prop56}
\mathbf{s}_U^k|\tilde A_x(m)=s^k|\tilde A_x(m), \ J^v|A_x(m)=\tilde J^v_U|A_x(m) \ x \in U,
\end{equation}
and one has $\nabla_{X^{\Omega^k}}(\mathbf{s}_U^k)=-2\pi ik\mathbf{s}_U^k$ over $\bigcup_{y \in U}B_y(m))$ (compare (\ref{features2}). Furthermore there is a finite open covering $U_\lambda, \lambda \in \{1,\dots,m\}$ of $S^1_\epsilon$ so that the associated pairs $(\mathbf{s}_U^k,\tilde J^v_{U})$ are compatible over $U_{\lambda\lambda'}:=U_\lambda\cap U_{\lambda'}$ in the sense that 
\[
(\mathbf{s}_{U_{\lambda}}^k,\tilde J^v_{U_\lambda})=\Phi_{\lambda\lambda'}^*(\mathbf{s}_{U_{\lambda'}}^k, \tilde J^v_{U_\lambda'}),\
\] 
where $\Phi_{\lambda\lambda'}(z):T^v_zY^k\rightarrow T^v_zY^k,\  z \in (f^k)^{-1}(U_{\lambda\lambda'})$ is the vectorbundle automorphism so that $\tilde J^v_{U_\lambda}=\Phi_{\lambda\lambda'}^*\tilde J^v_{U_\lambda'}$.
\end{ass}
By Lemma \ref{relativen}, this is the case if $f$ is quasihomogeneous and if the Milnor fibration $Y^k$ is canonically trivialized in the sense of Definition \ref{milnorfibr} along its boundary. Note that here, we do not impose any triviality conditions in a boundary neighbourhood on the pairs $(\mathbf{s}_U^k,\tilde J^v_{U})$, as 
opposed to Lemma \ref{relativen}. Although by the conctruction of $J$ in Lemma \ref{milnorsymplectic}, $J$ is constant relative to $\Theta_K: \bigcup_{y \in S^1_\epsilon}B_y(m)\rightarrow B_x(m)\times S^1_\epsilon$, we will assume $J$ to be of more general form in the following, that is $\theta_k^*(J^v)=J_x^v$, where $J_x^v$ depends on $x \in S^1_\delta$. The combination of the weaker Assumption \ref{relativenlocal} and the technique of Section \ref{boundingdisks} is expected to have applications in more general cases than in the quasihomogeneous case resp. where $f$ is a smoothing of a single isolated singularity along the lines which we will describe now (still assuming $f$ quasihomogeneous and $Y^k\subset X_e^k$ to be a symplectic fibration in the sense of Lemma \ref{milnorsymplectic}, but reducing the existence of $(\mathbf{s}^k, \tilde J^v)$ which is proven in Lemma \ref{relativen} to Assumption \ref{relativenlocal}). To begin, analogously to Lemma \ref{horsection} we define an embedded disk $D\subset X_e^k$ with boundary in $Q=:Q_\epsilon$, where the closed $n+1$-manifold $Q\subset Y^k\subset X^k_e$ is defined by factoring $\hat Q$ as in (\ref{hatq}) in Section \ref{relativen0}. $D$ is defined by specifying for any fixed point $z_0 \in Q_x$ a map (the index '$0$' will become clear below)
\begin{equation}\label{tildeu2}
\begin{split}
\tilde u_0&: [0,\epsilon]\times [0,1] \rightarrow \tilde D\\
\tilde u_0(t,\tau)&= \left\{\begin{matrix}\{z_0\}\times \{t\}\times \{\tau\}, \ \tau \in [0,1-\delta ]\\  \{(\rho^k_{\psi(\tau) \frac{t-\kappa}{\epsilon-\kappa}})(z_0)\}\times \{t\}\times \{\tau\}, \ \tau \in [1-\delta,1]\end{matrix}\right.,
\end{split}
\end{equation}
which factorizes resp. extends to a well-defined map $u_0:D_\epsilon \rightarrow D\subset X_e^k$ (here, we parametrized s.t. $\rho^k_0=Id,\ \rho^k_1=\rho^k$ and extended smoothly $\rho^k_{s}=Id$ for $s \leq 0$). As before, we choose a trivialization $\kappa_D$ of $\Lambda^{n+1,0}(T^*X_e^k)|D$, where $D={\rm im }(u_0)$ and define a function $e^{i\vartheta}:\partial D \rightarrow S^1$ by
\begin{equation}\label{maslovequality2}
\kappa^2_{Q_\epsilon}|_{\partial D}=e^{i\vartheta}\kappa^2_{D}|_{\partial D}
\end{equation}
where, as described in Section \ref{relativen0}, $\kappa_{Q_\epsilon}$ coincides over each $z' \in u_0(\{\epsilon\} \times [0, 1])\subset Q_\epsilon$, with the element of $\Lambda^{n+1,0}T^*_{z'}X_e^k$ induced by the Lagrangian subspace 
\begin{equation}\label{lagsubspaces3}
T^h_{z'}Y^k \oplus T_{z'}Q_{y(\tau)}\subset T_{z'} X_e^k, \ y(\tau)=f^k_e(z'),\ \tau \in [0,1].
\end{equation}
We can then prove:
\begin{lemma}\label{blalemma}
The winding number of the function $\alpha:[0,1]/\{0,1\}\rightarrow \mathbb{C}^*$ as defined in (\ref{windingbla}) coincides with the winding number of the function
\begin{equation}\label{blawinding2}
\tilde \alpha(s)=e^{i\vartheta/2\circ \tilde y(s)}\int_{Q_{\epsilon,{y(s)}}} i_{X_{f^k}}\kappa_Q, \ s \in [0,1]/\{0,1\},
\end{equation}
where $\tilde y:S^1\rightarrow \partial D$ parametrizes $D$ so that $f_e^k(\tilde y(s))=y(s):=\epsilon e^{2\pi is}$.
\end{lemma}
\begin{proof}
At first, we can proof analogously as inthe proof of Proposition \ref{keylemma} (note that the intersection of $Q_\epsilon$ with the image of $M\times \{\delta\}\times [1-r,1]$ in $X_e^k$ is not Lagrangian in $X_e^k$), that the winding number of $\hat \alpha:S^1_\epsilon\rightarrow \mathbb{C}^*$ defined by
\begin{equation}\label{tildetheta5}
\hat \alpha(s)=e^{i\theta\circ \tilde y(s)}\int_{Q_{\epsilon,y(s)}} i_{X_{f^k}}\kappa_Q, \ s \in [0,1]/\{0,1\},
\end{equation}
coincides with that of $\alpha$ (namely (\ref{maslovindex36})). Here, $\alpha$ resp. $e^{i\theta}:Q_\epsilon\rightarrow S^1$ are as defined in resp. below (\ref{windingbla}). Now note that if we define a $1$-form $\eta_{Q_\epsilon}$ on $Q_\epsilon$ by $\nabla\kappa^2_{Q_\epsilon}=\eta_{Q_\epsilon}\kappa^2_{Q_\epsilon}$, then we still (as in Section \ref{boundingdisks}) have the equality $\eta_{Q_\epsilon}=2id\theta$ on $Q_\epsilon$, which is immediate by differentiating (\ref{bla987}), furthermore along $C:=u(\{\epsilon\}\times [1-\delta, 1])\subset X_e^k$ we still have $2i\sigma_{Q_\epsilon}=\eta_{Q_\epsilon}$. On the other hand, by differentiating (\ref{maslovequality2}) we have the equality
\begin{equation}\label{masloveq23}
2id\theta|\partial D=\eta_{Q_\epsilon}|\partial D=(id\vartheta+ \eta_D)|\partial D, 
\end{equation}
where $\eta_D \in H^1(D)$ is defined by $\nabla \kappa^2_D=\eta_D\kappa^2_D$, where $\nabla$ is again the Levi-Civita-connection on $X_e^k$. So the assertion of our Lemma reduces to show the equality that if considering $\eta_D$ as an element of $H^1(\partial D,\mathbb{C})$, then
\begin{equation}\label{etavanish56}
\eta_D(\partial D)=0.
\end{equation}
To show that, define in analogy to the proof of Proposition \ref{keylemma} a smooth path $v:[0,1]\rightarrow M$ where $M=Y^k_x$ for some fixed $x \in S^1\epsilon$ so that $v(0)=z_0$, where $z_0$ is as in (\ref{tildeu2}) and $v(1)\in \partial M$. Then define a mapping
\begin{equation}\label{tildeF}
\begin{split}
\tilde F&: [0,1]\times [0,\epsilon]\times [0,1] \rightarrow \tilde D\\
\tilde F(s,t,\tau)&= \left\{\begin{matrix}\{v(s)\}\times \{t\}\times \{\tau\}, \ \tau \in [0,1-\delta ]\\  \{(\rho^k_{\psi(\tau) \frac{t-\kappa}{\epsilon-\kappa}})(v(s))\}\times \{t\}\times \{\tau\}, \ \tau \in [1-\delta,1]\end{matrix}\right.,
\end{split}
\end{equation}
which factorizes to a well-defined map $F:[0,1]\times [0,\epsilon]\times [0,1]/\{0,1\} \rightarrow X_e^k$ whose image is a compact smooth embedded $3$-manifold in $X_e^k$ with corners of codimension two. Define, as in the proof of Proposition \ref{keylemma}, a map $\hat u:[0,1]\times [0,1]/\{0,1\}\rightarrow Y^k$ by setting $\hat u(s,\tau)=F(s,\epsilon,\tau)$ and define $u_s:[0,\epsilon]\times [0,1]/\{0,1\},\ s\in [0,1]$ by $u_s(t,\tau)=F(s,t,\tau)$. So we have ${\rm im}(u_0)=D=:D_0$, $D_s:= {\rm im}(u_s)\subset X_e^k,\  s \in [0,1]$ is a smooth family of embedded disks and we have $\partial ({\rm im} (F))=D_0\cup D_1\cup \hat F$, where $\hat F:=\bigcup_{s\in [0,1]} \partial D_s$. Note that by construction and by the fact that $\rho^k_t|\partial M=Id, t \in [0,1]$ we have $\partial D_1\subset \partial Y^k$ and $TD_1 \subset T^hY^k$, where $T^hY^k\subset TY^k$ is the subbundle given by the closed two-form $\Omega$ on $Y^k$ as defined in Lemma \ref{extension}, so $D_1$ is horizontal with boundary in $\partial Y^k$. Now along $\hat F\subset Y_\epsilon\subset X_e^k$ we can construct a section $\mathbf{Q}\in \Gamma(\mathcal{L}(X^k_e,\Omega^k)|\hat F)$ over $\hat F$, where $\mathcal{L}(X^k_e,\Omega^k)$ denotes the bundle of Lagrangian subspaces of $TX_e^k$ relative to $\Omega^k$, by repeating the definition given in Section \ref{relativen0} in (\ref{lagpoli}) resp. (\ref{lagsubspaces2}). This in turn gives rise to a trivialization $\kappa_{\mathbf{Q}}\in \Gamma(\Lambda^{n+1,0}(T^*Y^k)|\hat F)$ along $\hat F$ so that $\kappa_{\mathbf{Q}}|\partial D_0=\kappa_{Q_\epsilon}|\partial D_0$, where $\kappa_{Q_\epsilon}\in \Gamma(\Lambda^{n+1,0}(T^*Y^k))|Q_\epsilon$ was defined below (\ref{maslovequality2}). On the other hand we can choose a smooth family of trivializations $\kappa_{D_s} \in \Gamma(\Lambda^{n+1,0}(T^*X^k_e)|D_s)$ for any $s \in [0,1]$, so that $\kappa_{D_0}=\kappa_D$ with $\kappa_D$ as in (\ref{maslovequality2}). Note that in these construction, we used the nearly complex structure $J$ on $Y^k$ defined in Lemma \ref{milnorsymplectic}. Now choose an open covering $U_\lambda\subset S^1_\epsilon, \lambda \in G$, where $G$ is a finite set, of $S^1_\epsilon$, where the $U_\lambda$ satisfy the hypothesis of Lemma \ref{relativen} and consider for each $U_\lambda$ the pair $(J_\lambda, \mathbf{s}^k_\lambda)$ of vertical nearly complex structures resp. elements $\mathbf{s}^k_\lambda \in \Gamma(\Lambda^{n,0}(T^*Y^k)^v)|(f^k)^{-1}(U_\lambda)$ that satisfy the conditions in Lemma \ref{relativen}. Choose a family of bundle automorphisms $\Phi_\lambda(x):T_x^vY^k\rightarrow T_x^vY^k,\ x\in (f^k)^{-1}(U_\lambda)$ so that $\Phi_\lambda^*(J)=J_\lambda$ and define for any $s\in [0,1]$ functions $e^{i\vartheta_\lambda}, e^{i\theta_\lambda}: \partial D_s^\lambda:=\partial D_s\cap (f^k)^{-1}(U_\lambda)\rightarrow S^1$ by the requirement that over $\partial D_s^\lambda$ we have
\begin{equation}\label{thetadef5}
e^{i\theta_\lambda}\Phi_\lambda^*i_{X_{f^k}}\kappa_{\mathbf{Q}}|_{\partial D_s^\lambda}=\mathbf{s}^k_\lambda|_{\partial D_s^\lambda},\quad 
\Phi_\lambda^*\kappa^2_{\mathbf{Q}}|_{\partial D_s^\lambda}=e^{i\vartheta_\lambda}\Phi_\lambda^*\kappa^2_{D_s}|_{\partial D_s^\lambda}
\end{equation}
We now claim that the sets $\{2d\theta_\lambda-d\vartheta_\lambda:=2d {\rm log}(e^{i\theta_\lambda})-d{\rm log}(e^{i\vartheta_\lambda})\}_{\lambda \in G}$ give rise to a well-defined closed $1$-form $\eta_{\hat F}:=2id\theta-id\vartheta\in \Omega^1(\hat F)$, where $\hat F=\bigcup_{s\in [0,1]} \partial D_s$. To see this, let $\lambda,\lambda'\in G$ so that $U_{\lambda\lambda'}:=U_\lambda\cap U_{\lambda'}\neq \emptyset$ and define the bundle automorphism 
\[
\Phi_{\lambda\lambda'}(x)=\Phi_{\lambda'}\circ\Phi_\lambda^{-1}(x):T_x^vY^k\rightarrow T_x^vY^k,\ x \in (f^k)^{-1}(U_{\lambda\lambda'}).
\]
Then, by the definition of the $\mathbf{s}^k_\lambda$ in the proof of Lemma \ref{relativen} resp. its local version Ass. \ref{relativenlocal}, we have $\mathbf{s}^k_{\lambda'}=\Phi_{\lambda\lambda'}^*\mathbf{s}^k_{\lambda}$ over $U_{\lambda\lambda'}$. Then applying $\Phi_{\lambda\lambda'}^*$ for any $s \in [0,1]$ to both equations in (\ref{thetadef5}) substitutes $\lambda$ by $\lambda'$ in all involved elements of $\Gamma(\Lambda^{n,0}(T^*Y^k)^v)|D_s\cap (f^k)^{-1}(U_{\lambda\lambda'})$, which proves the claim. By (\ref{masloveq23}) we have $\eta_{\hat F}|\partial D_0=2id\theta-id\vartheta|\partial D_0=\eta_D$ while by the second equation in (\ref{relativenlocal}) we have $\Phi_\lambda(z)= Id$ if $z\in \bigcup_{x\in S^1_\epsilon}B_x(m)$ and $\lambda \in G$, i.e. we have $\Phi_\lambda(z)=Id$ for any $z \in \partial D_1\subset \partial Y^k$. Thus it follows by covariantly differentiating (\ref{thetadef5}) along $\partial D_1$ and substracting the results that 
\[
\eta_{\hat F}|\partial D_1= (2id\theta-id\vartheta)|\partial D_1=\eta_{D_1}|\partial D_1.
\]
Consequently we have $0=\int_{\hat F}d\eta_{\hat F}=\int_{\partial D_0}\eta_D-\int_{\partial D_1}\eta_{D_1}$. But since the disk $D_1\subset X_e^k$ is {\it horizontal}, we can conclude as in the proof of Lemma \ref{tildetheta3}, that $\eta_{D_1}\equiv 0$ as an element of $\Omega^1(\partial D_1)$, but this proves (\ref{etavanish56}), which proves the Lemma.
\end{proof}
To finish the proof of Theorem \ref{theorem34}, we proceed along the lines of Section \ref{boundingdisks} and construct as in the proof of Lemma \ref{lag3} a smooth $n+2$-dimensional submanifold of $X_e^k$ with boundary $\tilde Q\subset X_e^k$ being a cobordism from $Q_\kappa=Q_x\times S_\kappa^1$ to $Q_\epsilon$ (using the notation of Lemma (\ref{lag3}) by setting for $t \in [\kappa,\epsilon]$ (here $0<\kappa< s$ as in Lemma \ref{lag3}) and for the  function $\psi:[1-\delta,1]\rightarrow [0,1]$ as chosen in (\ref{tildeu2}) resp. in Section \ref{relativen0} that is zero in some neighbourhood of $1-\delta$ and equal to $1$ in a neighbourhood of $1$,
\begin{equation}\label{hatq2}
\hat Q= Q_x\times [\kappa,\epsilon]\times [0,1-\delta]\cup \bigcup_{\tau \in [1-\delta,1], t \in [\kappa,\epsilon]} (\rho^k_{\psi(\tau) (t-\kappa)/(\epsilon-\kappa))})(Q_x)\times \{t\}\times \{\tau\}
\end{equation}
where we assumed that $\rho^k_{(\cdot)}:[0,1]\times M\rightarrow M$ is parametrized so that $\rho^k_t=id$ for $t \in [0,s]$, $s$ small, and $\rho^k_t=\rho^k$ in some neighbourhood of $t=1$. Then $\hat Q\subset M\times [\kappa,\epsilon]\times [0,1]$ factorizes to a well-defined submanifold $\tilde Q\subset X_e^k$ whose intersection $Q_t$ with $Y^k_t=(f^k_e)^{-1}(S_t^1),  t \in [\kappa,\epsilon]$ is fibrewise Lagrangian, we denote $Q_t:=\tilde Q\cap Y_t$. As above we choose a trivialization $\kappa_D$ of $\Lambda^{n+1,0}(T^*X_e^k)|D$, where $D={\rm im }(u_0)$ and noting that the intersection $D\cap Q_t$ is transversal we define a function $e^{i\vartheta_t}: D\cap Q_t  \rightarrow S^1, \ t \in [\kappa,\epsilon]$ by
\begin{equation}\label{maslovequality3}
\kappa^2_{Q_t}|_{D \cap Q_t}=e^{i\vartheta_t}\kappa^2_{D}|_{D \cap Q_t}
\end{equation}
where, analogous to the above, $\kappa_{Q_t}$ coincides over each $z' \in u_0(\{t\} \times [0, 1])\subset Q_t$, with the element of $\Lambda^{n+1,0}T^*_{z'}X_e^k$ induced by the Lagrangian subspace 
\[
T^h_{z'}Y^k \oplus T_{z'}Q_{t,y(\tau)}\subset T_{z'} X_e^k, \ y(\tau)=f^k_e(z'),\ \tau \in [0,1].
\]
Then $e^{i\vartheta_1}: D\cap Q_t  \rightarrow S^1$ coincides with $e^{i\vartheta}:\partial D\rightarrow S^1$ in Lemma \ref{blalemma} and as in Section \ref{boundingdisks} resp. the proof of Proposition \ref{alphatau2} one shows that ${\rm wind}(e^{i\vartheta_0})=k$, which implies that ${\rm wind}(\tilde \alpha)=0$ which gives a contradiction to (\ref{maslovindex37}) as in the proof of the Propositions \ref{keylemma} resp. \ref{alphatau2}, proving the Theorem.\\
{\it Remark.} Note that if we add to Assumption \ref{relativenlocal} a triviality-condition along a boundary-neighbourhood, that is there is a diffeomorphism $\Theta: \bigcup_{y\in S^1_\delta}B_y(m)\rightarrow B_x(m)\times S^1_\epsilon$ for some fixed $x \in S^1_\delta$ so that $(\mathbf{s}^k_{U_\lambda}, \tilde J^v_{U_\lambda})=({\rm pr}_1\circ(\Theta))^*(\mathbf{s}_x^k, \tilde J^k_x)$ for some $\mathbf{s}_x^k \in  \Gamma(\Lambda^{n,0}(T^*B_x(m))^v)$ and some vertical complex structure $\tilde J^v_x$ on $TB_x(m)$, then the proof of Lemma \ref{blalemma} shows that the family $d\theta_\lambda,\ \lambda \in G$ (see the discussion below (\ref{thetadef5})) can be extended to a closed $1$-form on $\hat F$ and one can proceed along the lines of the proof of Proposition \ref{keylemma} in Section \ref{relativen0}, that is, we do not need 'bounding disks'. The triviality condition along the boundary of $Y^k$ is needed to show that this extension takes the precise value $-2\pi k$ along $\partial D_1$ (notation of this section), whereas by using bounding disks, it suffices to observe that there is a closed extension of $\eta_{\hat F}$ whose restriction to $\partial D_1$ is derived from a given trivialization of $\Lambda^{n,0}(T^*X_e^k)$ over $D_1$ and that $D_1$ is horizontal.

\subsection{General polynomials, Lagrangians and bounding disks}\label{generalpol}
In this subsection, we will briefly discuss how to use the 'bounding disk'-method of Section \ref{boundingdisks} culminating in Proposition \ref{alphatau2} to formulate necessary conditions for the existence of certain Lagrangian submanifolds in (symplectically trivial) coverings of one-parameter-smoothings of an isolated singularity at $0$ of a general polynomial, $f:\mathbb{C}^{n+1}\rightarrow \mathbb{C}$. For this, let $f:\hat X\rightarrow D^*_\delta$ be the one-parameter-smoothing of $f^{-1}(0)\cap B_\epsilon$ for some small $0 < \epsilon$ as constructed in (\ref{milnor1}) (replacing $B_1$ by $B_\epsilon$) and suppose we have chosen some basis $(s_1,\dots,s_\mu)$ spanning the sheaf of modules $\mathcal{H}^n(f_*\Omega^\cdot_{\hat X/D_\delta})$ over $\mathcal{O}_{D_\delta}$ so that the associated basis $g_1,\dots g_\mu \in \mathcal{O}_{\mathbb{C}^{n+1},0}/\hat M(f)$, where $\hat M(f)$ is a certain submodule (see \cite{loo}) of $\mathcal{O}_{\mathbb{C}^{n+1},0}$ so that $\mathcal{O}_{\mathbb{C}^{n+1},0}/\hat M(f)¸\simeq \omega_{f,0}/d\Omega^{n-1}_{X/D_\delta,0}$ (c.f. section \ref{relcohom}), is chosen so that $g_1=1$. As in the quasihomogeneous case, let $\hat s \in \Gamma(\mathcal{H}^n(f_*\Omega^\cdot_{\hat X/D_\delta}))$ be associated to $g_1$ and satisfying (\ref{features}) and consider the exact symplectic fibration $f:X\rightarrow D_\delta$ as introduced in (\ref{milnorfibr}). As before, let $s_x=\Psi_x^*\hat s_x, \ x\in D^*_\delta$ using the smooth family of fibrewise diffeomorphisms $\Psi_x:X_x\rightarrow \hat X_x$ as defined in (\ref{th:classical-milnor-fibre}). Then $s \in \Gamma({\bf H}^n(Z,\mathbb{C}))$ is a section of the $n$-th cohomology bundle of the fibres, using notation as introduced below (\ref{features}). We now assume that if $\rho \in {\rm Symp}(M,\partial M,\omega)$ is the symplectic monodromy of $X\cap f^{-1}(S^1_\delta)=:Y$ induced by the kernel of $\Omega|TY$ and $M:=Y_x$ for some fixed $x\in S^1_\delta$, that $[\rho^k]=Id$ in $\pi_0({\rm Symp}(M,\partial M,\omega))$ for some $k \in \mathbb{N}$. Then we can form, as in Lemma \ref{extension}, a symplectic extension $X^k_e$ of the symplectic covering $\pi_k:X^k\rightarrow X$ of $X$ to $D_\delta$ and define $s^k:=\pi_k^*(s)\in {\bf H}^n(Z^k,\mathbb{C})$ using notation as above. We now set $Y^k:=\pi_k^{-1}(Y)$ and make the following 
\begin{ass}\label{ass3}
For a fixed $x\in S^1_\delta$, let $Q_x\subset M$ be the Lagrangian submanifold so that $\int_M[s_x]\wedge PD[Q_x]=c \neq 0\in H^n(M,\mathbb{C})$. We assume that 
\begin{enumerate}
\item $[s_x] \in {\rm im}\ i:H^n(M,\partial M,\mathbb{C})\rightarrow H^n(M,\mathbb{C})$
\item There exists a (time-dependant) Hamiltonian flow $\Phi_H: [0,1]\times M\rightarrow M$ 
so that $\rho^k=\Phi_H(1)$ and a function $\alpha: [0,1]/\{0,1\}\rightarrow \mathbb{C}^*$
satisfying $\alpha(t) \cdot s^k_{x}=(\Phi_H(t)^{-1})^*\mathcal{P}_{x,y(t)}^*s^k_{y(t)}$, for all $t \in [0,1]$, where $\mathcal{P}_{x,y(t)}:Y^k_x\rightarrow Y^k_{y(t)}$ denotes symplectic parallel transport in $Y^k$ along  $y(t)=xe^{2\pi it}$.
\item There is a neighbourhood $\mathcal{U}$ of $Y^k$ in $X_e^k$ so that $\mathcal{U}\subset X^k_e$ is 'flat' in the sense of the proof of Lemma \ref{extension} and Kaehler.
\end{enumerate}
\end{ass}
Asssume now that $\rho^k(Q_x)\subset Q_x$, so that there is as in Lemma \ref{bla456} a closed Lagrangian submanifold $Q\subset Y^k$ in $X^k_e$ which fibres over $S^1_\delta$ to a family of fibrewise Lagrangians $Q_y\subset Y_y^k, y\in S^1_\delta$. Assume furthermore that there is a horizontal section $u:D_\delta \rightarrow X^k_e$ so that $u(D)=F$ and $\partial F\subset Q$ as constructed in Lemma \ref{horsection} (where we used (\ref{horizontality}) in Assumption \ref{ass2}, in the following we will simply assume the existence of a {\it horizontal} disk with boundary in $Q$). Then by the horizontality of $u$ and by the flatness and Kaehler property of $X^k_e$ in $\mathcal{U}$ we can deduce as in the proof of Lemma \ref{tildetheta3} that w.r.t. to some trivialization $\kappa_F\in \Lambda^{n,0}(T^*X^k_e)|F$, $\eta_F|\partial F=0$, where $\eta_F\in \Omega^1(F,\mathbb{C})$ is defined as below (\ref{cieliebla}). Furthermore it follows as in Lemma \ref{alphatau2}, that $e^{i\vartheta}:\partial F\rightarrow S^1$ defined by 
\[
\kappa_Q^2=e^{i\vartheta}\kappa_F^2.
\]
where $\kappa_Q\in \Lambda^{n,0}(T^*X^k_e)|Q$ is induced by $TQ\subset TX^k_e$ as before, satisfies
\[
\mu(F):={\rm wind}(e^{i\vartheta})=k.
\]
Finally, by combining this with Corollary \ref{tildealpha2} we can argue as in the proof of Lemma \ref{alphatau2}, that ${\rm wind}(\alpha)=\mu(F)-k=0$, so to summarize we arrive at the following Corollary:
\begin{folg}\label{generallag}
Assume that $[\rho^k]=Id$ in $\pi_0({\rm Symp}(M,\partial M,\omega))$ for some $k\in \mathbb{N}$, where $\rho$ is the symplectic monodromy of the Milnor bundle $Y\subset X$ of a polynomial $f:(X,Y)\rightarrow (D_\delta, S^1_\delta)$ with isolated singularity at $0$. Then, if there is a Lagrangian submanifold $Q\subset Y^k\subset X^k_e$ as described above so that with $Q_y=Q\cap Y^k_y$ the first condition in Assumption \ref{ass3} is satisfied for some (hence any) $y\in S^1_\delta$ and such that there is horizontal disk $F\subset X^k_e$ with boundary in $Q$, then the Maslov-Index $\mu(F)$ of $F$ with boundary in $Q$ equals $k$. If furthermore there is such a disk and in addition also 2. and 3. of Assumption \ref{ass3} are satisfied, then the winding number of $s^k$ w.r.t. to $Q$ is zero, that is
\begin{equation}\label{masdiskeq}
{\rm wind}(\alpha)={\rm wind}({\rm ev}(s_{y})(Q_y))_{y\in S^1_\delta}=\mu(F)-k=0.
\end{equation}
\end{folg}
Conversely, if $\rho^k_*=0$ in ${\rm Aut}(H^*(M,\mathbb{C}))$ and assume that $\alpha(t):={\rm ev}(s^k_{y(t)})(\mathcal{P}_{x,y(t)}(Q_x))\neq 0$, for all $t \in [0,1]$ with the notation from above. Then if ${\rm wind}(\alpha)\neq 0$, then it follows that at least one of the above Assumptions are not satisfied. So, either $[\rho^k]\neq Id$ in $\pi_0({\rm Symp}(M,\partial M,\omega))$ (which is the case if $f$ is quasi-homogeneous by Theorem \ref{theorem34}) or if $[\rho^k] = Id$ and 2. in Assumption \ref{ass3} is satisfied (note that we did not assume $\Phi_H$ to be the identity on $\partial M$ for all $t$) there is no Lagrangian submanifold $Q\subset Y^k$ in $X^k$ which fibres over $S^1_\delta$ into fibrewise Lagrangians $Q_y, \ y\in S^1_\delta$ so that 1. and 3. in Assumption \ref{ass3} are satisfied and such that there is a horizontal section $u:D_\delta\rightarrow X^k_e$ in the extension $X^k_e$ of $X^k$ to $D_\delta$ which has boundary in $Q$.\\
{\it Remark.} In general, we have since $\mu(F):H_2(X_e^k, Q)\rightarrow \mathbb{Z}$ and since $H_1(X_e^k, \mathbb{Z})=H_2(X_e^k, \mathbb{Z})$ by the long exact sequence of the tuple $(X_e^k, Q)$, $\mu(F)$ depends only on the class in $H_1(Q,\mathbb{Z})$ defined by $\partial F \subset Q$. Now, while by the proof of Lemma \ref{tildetheta} $\mu(F)$ equals the evaluation of the mean curvature form $\sigma_Q$ of $Q\subset \tilde X^k$ on any {\it horizontal} path realizing the homology class of $\partial F$ in $Q$, the result of Corollary \ref{generallag} under the given assumptions gives the value of the Maslov-index of any disk $F$ with boundary in $Q$ so that $\partial F$ generates $H_1(Q,\mathbb{Z})$. So, the existence of a {\it horizontal} disk $F$ with boundary in $Q$ and the full Assumption \ref{ass3} is essential for the equality (\ref{masdiskeq}) that relates the Maslov index along $F$ to the cohomological winding number associated to the section $s \in \Gamma({\bf H}^n(Z,\mathbb{C}))$, while for the equality $\mu(F)=k$ the Assumption \ref{ass3} is not needed.

\section{Relative cohomology}\label{relcohom}
Let $U \subset \mathbb{C}^{n+1}$ be an open set and let 
$f:U\rightarrow \mathbb{C}$ be a holomorphic map so that $x \in
\mathbb{C}^{n+1}$ is an isolated singularity, that is $f$ outside $x$ is a
submersion, assume $f(x)=0$. Let $\epsilon$ and $\delta $ be positive real
numbers and  $S=\{u \in \mathbb{C} \mid |u| < \delta \}$, $X= \{z \in
\mathbb{C}^{n+1} \mid |z| < \epsilon,\ f(z) \in S \}$, $X_0=\{z \in X \mid
f(z)=0\}$ so that with $X' =X - X_0$, $S' =S -{0}$ one gets a locally trivial
$C^\infty$-fibration $f: X' \rightarrow S'$. Let $(\Omega^\cdot_{X'},d)$ be the
sheaf complex of holomorphic differential forms on $X'$, then with
$\Omega^i_{X'/S'}=\Omega^i_{X'} /df \wedge \Omega^{i-1}_{X'}$ we get the sheaf
complex of relative differential forms $(\Omega^\cdot_{X'/S'},d)$ on $X'$. By
the Lemma of Poincare and the regularity of $f|X'$ one has a resolution of
$f^{-1}\mathcal{O}_{S'}$ in the category of $(f^{-1}\mathcal{O}_{S'})$-modules
by
\begin{equation}\label{res}
0 \rightarrow f^{-1}\mathcal{O}_{S'} \rightarrow \Omega^0_{X'/S'} \rightarrow
\Omega^1_{X'/S'} \rightarrow \dots \ .
\end{equation}
Here, $f^{-1}\mathcal{O}_{S'}$ is the topological preimage of the sheaf
$\mathcal{O}'$ of holomorphic functions on $S'$. On the other hand, we observe
that the vector spaces $H^i(X_u,\mathbb{C})$, where $X_u$ are the fibres of
$f:X' \rightarrow S'$, are the fibres of the etale space of the sheaf $R^i 
f_{*}\mathbb{C}_{X'}$, where for an abelian sheaf $\mathcal{F}$ on $X$ and a mapping
$f:X \rightarrow S$ $R^if_*\mathcal{F}$ is the sheaf on $S$ associated to $V
\subset S, V \mapsto
H^p(f^{-1}(V),\mathcal{F})$ ($R^if_*$ is identical to the right derived functor
of the direct image functor $f_*$ and is calculated by 
injective or $f_*$-acyclic resolutions of $\mathcal{F}$, for details see Hartshorne
\cite{hartshorne} or Grothendieck \cite{ega}).
We have the following isomorphism, refer to Looijenga \cite{loo}. Note that the
sections of the vectorbundle accociated to 
$R^if_{*}\mathbb{C}_{X}$ (with fibres $H^i(X_u,\mathbb{C})$ over $S'$)
constitute the
sheaf $R^i f_{*}\mathbb{C}_{X}\otimes _{\mathbb{C}_{S}}\mathcal{O}_{S}$.
\begin{lemma}\label{sect}
With notation as above, the natural map
\[
(R^i f_{*}\mathbb{C}_{X}) \otimes _{\mathbb{C}_{S}}\mathcal{O}_{S}
\longrightarrow R^i f_{*}(f^{-1}\mathcal{O}_{S})
\]
is an isomorphism.
\end{lemma}
Now consider the complex of direct image sheafs $f_*\Omega^\cdot_{X'/S'}$, this
is a complex of $\mathcal{O}_{S'}$-modules, its cohomology sheafs will be
denoted by $\mathcal{H}^p(f_*\Omega^\cdot_{X'/S'})$ for all $p$. The following
result identifies these with the space of sections in the bundle of fibrewise
cohomology groups, for details we refer to Looijenga \cite{loo}.
\begin{prop}\label{deRham}
In the above situation, the fibrewise de Rham evaluation maps
\[
DR_u: \mathcal{H}^p(f_*\Omega^\cdot_{X'/S'})_u \longrightarrow
H^i(X_u,\mathbb{C})
\]
given by integration over the fibre $f^{-1}(u), u \in S'$ are isomorphisms.
Furthermore, they fit together to define a sheaf isomorphism
\[
DR: \mathcal{H}^p(f_*\Omega^\cdot_{X'/S'})_u \longrightarrow (R^i
f_{*}\mathbb{C}_{X'}) \otimes _{\mathbb{C}_{S'}}\mathcal{O}_{S'}.
\]
\end{prop}
\begin{proof}
We will briefly describe the arguments. Note first, that, taking the canocial
soft resolution for the complex $f_*\Omega^\cdot_{X'/S'}$, one has two spectral
sequences with $E_2$-terms
\[
'E_2^{p,q}  = \mathcal{H}^p(R^qf_*\Omega^\cdot_{X'/S'}), \quad ''E_2^{p,q}=R^p
f_*(\mathcal{H}^q(\Omega^\cdot_{X'/S'})),
\]
both converging to the cohomology of the full complex,
$\mathbb{R}^\cdot(\Omega^\cdot_{X'/S'})$.
Here, $R^qf_*\Omega^\cdot_{X'/S'}$ denotes the complex
$R^qf_*(\Omega^i_{X'/S'})_{i \in \mathbb{Z}}$. Since $f$ is Stein, the first
spectral sequence degenerates which gives
\[
\mathcal{H}^p(f_*\Omega^\cdot_{X'/S'})\simeq
\mathbb{R}^p(\Omega^\cdot_{X'/S'}).
\]
On the other hand, considering the resolution (\ref{res}), we also have
$\mathcal{H}^p(\Omega^\cdot_{X'/S'})=0,\ p> 0$, that is, the second spectral
sequence degenerates, giving
\[
R^pf_*(f^*\mathcal{O}_{S'})\simeq  \mathbb{R}^p(\Omega^\cdot_{X'/S'}).
\]
Putting this together and using Lemma \ref{sect}, we arrive at the assertion.
\end{proof}
Note that in the above, we worked outside the 'critical set', that is, over
$S'$, which implied that $\mathcal{H}^p(\Omega^\cdot_{X'/S'})=0,\ p> 0$.  Now
note that the sheaf complex of relative differential forms is equally well
defined on $X$ over $S$, so for further use we state the following refinement of
Lemma \ref{deRham}, we will only sketch its proof, for details see
Looijenga \cite{loo}, Greuel \cite{greu} or Brieskorn \cite{bries}.
\begin{lemma}\label{germs}
Let $f:X\rightarrow S$ be a good Stein representative of a smoothing of an
isolated singularity as described above. Then, after possibly shrinking $S$
we have $\mathcal{H}^p(f_*\Omega^\cdot_{X/S})=0,\ n >p> 0$ and
$\mathcal{H}^n(f_*\Omega^\cdot_{X/S})$ is a free $\mathcal{O}_S$-module
of rank $\mu$, where $\mu$ is the $n$-th Betti number of a Milnor fibre. The
former is fitting in the exact sequence 
\[
0 \rightarrow R^i f_{*}\mathbb{C}_{X} \otimes
_{\mathbb{C}}\mathcal{O}_{S}\xrightarrow
{\alpha^n}\mathcal{H}^n(f_*\Omega^\cdot_{X/S})
\xrightarrow {\beta^n}f_*\mathcal{H}^n(\Omega^\cdot_{X/S})\rightarrow 0.
\]
whgich implies lemma \ref{deRham}. Furthermore, in $0 \in S$, there is a
canonical isomorphism
\begin{equation}\label{stalk2}
\beta^n:\mathcal{H}^p(f_*\Omega^\cdot_{X/S})_0 \longrightarrow
f_*\mathcal{H}^p(\Omega^\cdot_{X/S})_0=H^p(\Omega^\cdot_{X/S,x})
\end{equation}
for $p > 0$.
\end{lemma}
\begin{proof}
The first thing to prove (\cite{loo}, Prop. 8.5) is the long exact sequence for
$p>0$
\[
.. \rightarrow R^p f_{*}\mathcal{H}^0(\Omega^\cdot_{X/S})\xrightarrow
{\alpha^p}\mathcal{H}^p(f_*\Omega^\cdot_{X/S})
\xrightarrow {\beta^p}f_*\mathcal{H}^p(\Omega^\cdot_{X/S})\rightarrow  R^p
f_{*}\mathcal{H}^0(\Omega^\cdot_{X/S})\rightarrow ..
\]
Lemma \ref{deRham} then follows from
$f_*\mathcal{H}^p(\Omega^\cdot_{X'/S'})=0$,
$\mathcal{H}^0(\Omega^\cdot_{X'/S'})=f^{-1}\mathcal{O}_{S'}$ and Lemma
\ref{sect}.
Then one proves that if $X''\subset X$ so that $f|X''$ is also a Stein
representative then the restriction homomorphism between the corresponding exact
sequences is an isomorphism, taking direct limits, one infers that $\beta^p$ is
an isomorphism. Now in \cite{loo}, Prop 8.20, one proves that
one has an {\it exact} sequence of stalks
\[
0 \rightarrow \mathcal{O}_{S,0} \rightarrow \mathcal{O}_{X,x} \rightarrow \Omega^0_{X/S,x}
\rightarrow
\Omega^1_{X/S,x} \rightarrow \dots \rightarrow \Omega^n_{X/S,x}
\]
and $\Omega^n_{X/S,x}/d\Omega^{n-1}_{X/S,x}$ is free of rank $\mu$ (as a
$\mathcal{O}_{S,0}$-module). Then from (\ref{stalk2}), the exact sequence 
\begin{equation}\label{stalk3}
\mathcal{H}^p(\Omega^\cdot_{X/S,x})\rightarrow
\Omega^n_{X/S,x}/d\Omega^{n-1}_{X/S,x}\xrightarrow {d}\Omega^{n+1}_{X/S,x}
\end{equation}
and the fact that $\mathcal{H}^p(f_*\Omega^\cdot_{X/S})$ is coherent, it already
follows that for sufficiently small $S$, $\mathcal{H}^p(f_*\Omega^\cdot_{X/S})$
is a free $\mathcal{O}_S$-module of rank $\mu$ for $p=n$ and is trivial for
$0<p<n$. Note again that for $S$ small enough one has
$\mathcal{H}^0(\Omega^\cdot_{X/S})=f^{-1}\mathcal{O}_{S}$, so $R^p
f_{*}\mathcal{H}^0(\Omega^\cdot_{X/S})$ may be identified with 
$R^i f_{*}\mathbb{C}_{X} \otimes_{\mathbb{C}}\mathcal{O}_{S}$.
\end{proof}
Note that from the above Lemma and the sequence (\ref{stalk3}) it follows that a
basis spanning the $\mathcal{O}_{S,0}$-module 
$\Omega^n_{X/S,x}/d\Omega^{n-1}_{X/S,x}$ will already span the coherent
$\mathcal{O}_S$-module $\mathcal{H}^p(f_*\Omega^\cdot_{X/S})$, provided $S$ is
small enough. However, $df: \Omega^n_{X/S}\rightarrow \Omega^{n+1}_X\simeq
\mathcal{O}_X$, is only an isomorphism outside $\{x\}$, so if $j:X'\rightarrow
X$ denotes the inclusion, we have in our case (see \cite{loo})
$\omega_f:=j_*j^{-1}\Omega^n_{X/S}\simeq \Omega^{n+1}_X$ and one has the
sequence:
\[
0\rightarrow \Omega^n_{X/S,x}\rightarrow \omega_{f,x}\rightarrow
\omega_{f,x}\otimes\mathcal{O}_{\{x\},x}
\]
i.e. $\omega_f$ and $\Omega^n_{X/S}$ coincide outside of $\{x\}$. One then has
the exact sequence
\begin{equation}\label{dualizing}
0\rightarrow \Omega^n_{X/S,x}/d\Omega^{n-1}_{X/S,x}\rightarrow
\omega_{f,x}/d\Omega^{n-1}_{X/S,x}\rightarrow
\omega_{f,x}\otimes\mathcal{O}_{\{x\},x}
\end{equation}
so $\omega_{f,x}/d\Omega^n_{X/S,x}$ is also a free (\cite{loo}, Prop. 8.20)
$\mathcal{O}_{S,0}$-module of rank $\mu$ (note that $\mathcal{O}_{\{x\},x}\simeq \mathcal{O}_{\mathbb{C}^{n+1},x}/(\frac{\partial f}{\partial z_0},\dots,\frac{\partial f}{\partial
z_n})\mathcal{O}_{\mathbb{C}^{n+1},x}$). Then identifying $\omega_{f,x}$ with
$\mathcal{O}_{\mathbb{C}^{n+1},0}$ by means of $\alpha \mapsto df\wedge \alpha$
there is a correspondence of $d\Omega^{n-1}_{X/S,x}$ with a certain $\mathbb{C}\{f\}$ submodule of
$\mathcal{O}_{\mathbb{C}^{n+1},x}$ which we denote by $\hat M(f)$. For $f$
quasihomogeneous, that is, there are positive integers
$\beta_0,\dots,\beta_n,\beta$ so that $f$ is a $\mathbb{C}$-linear combination
of monomials $z_0^{i_0}\dots z_n^{i_n}$ so that $i_0\beta_0+\dots
+i_n\beta_n=\beta$ one deduces that $\mathcal{O}_{\mathbb{C}^{n+1},x}/M(f)$
coincides with $\mathcal{O}_{\mathbb{C}^{n+1},x}/(\frac{\partial 
f}{\partial z_0},\dots,\frac{\partial f}{\partial
z_n})\mathcal{O}_{\mathbb{C}^{n+1},x}$ and 
a basis fort the latter module can be chosen to consists of monomials
$\alpha_1,\dots,\alpha_r$, so that for every $\alpha_j$ there is a number $d_j$
such that $\alpha_j=z_0^{i_0}\cdot \dots\cdot z_n^{i_n}$ with
$i_0w_0+\dots+i_nw_n=d_j$ where $w_i=\beta_i/\beta$ ($d_j$ will be called the
degreee of $\alpha_j$). Summing up, we have (\cite{loo})
\begin{lemma}\label{poly}
For $f:X\rightarrow S$ quasihomogeneous with $0 \in \mathbb{C}^{n+1}$ an
isolated singularity there are global sections $\phi_1,\dots,\phi_\mu$ of $\mathcal{H}^i(f_*\Omega^\cdot_{X/S})$ 
which represent a basis of $H^n(X_u,\mathbb{C})$ for any $u \in S'$ that can be represented by monomials
$\alpha_1,\dots,\alpha_\mu \in \mathbb{C}[z_0\dots,z_{n}]$ by the correspondence $\phi\mapsto$ [coefficient of $df\wedge \phi_i$]. Here, $\mu$ is the Milnor number of $f$. These monomials project onto a $\mathbb{C}$-basis of $\mathcal{O}_{\mathbb{C}^{n+1},0}/(\frac{\partial f}{\partial z_0},\dots,\frac{\partial f}{\partial z_n})\mathcal{O}_{\mathbb{C}^{n+1},0}$.
\end{lemma}
We finally note that $R^i f_{*}\mathbb{C}_{X'} \otimes_{\mathbb{C}}\mathcal{O}_{S'}$ carries a canonical flat connection, the Gauss-Manin connection. Using the correspondence describes in lemma \ref{germs} one can extend this to sections of the sheaf $\mathcal{H}^p(f_*\Omega^\cdot_{X/S})$. For this, one sets over $S'$ if $\omega \in \mathcal{H}^p(f_*\Omega^\cdot_{X/S})$, so $d\omega=df\wedge \tilde \omega$ for a certain $\tilde \omega \in f_*\Omega^n_{X}$,
\[
\begin{split}
\nabla_\psi\omega:=\mathcal{L}_\psi(\omega)&=i_\psi d\omega\ {\rm mod}(df_*\Omega^{n-1}_X)\\
&=\tilde \omega \ {\rm mod}(df\wedge f_*\Omega^ {n-1}_X+df_*\Omega^{n-1}_X),
\end{split}
\]
where $\psi$ lifts $\frac{\partial}{\partial z}$, so $\nabla_\psi\omega=dz\otimes [\tilde \omega]$. It is then well-known (\cite{loo}) that $\nabla$ extends 'regular-singular' along $S$ and that $\nabla$ maps each of the the modules $ \mathcal{H}^p(\Omega^\cdot_{X/S,x})\subset  \Omega^n_{X/S,x}/d\Omega^{n-1}_{X/S,x}\subset \omega_{f,x}/d\Omega^{n-1}_{X/S,x}$ into the next in the chain of inclusions.\\
{\it Remark.} Instead of working with the sheaf $\omega_{f}$ whose quotient by $d\Omega^{n-1}_{X/S}$ at $x$ fits into the short exact sequence (\ref{dualizing}) and which is isomorphic to $\Omega^n_{X/S}$ outside of $\{x\}$ (this approach goes back to Looijenga \cite{loo}) we will in the following also frequently refer to a more common definition of the Brieskorn lattice $\mathcal{H}''$ which is equivalent to the above for our case of an isolated singularity. $\mathcal{H}''$, understood as a sheaf over $S$, fits into the exact sequence (see \cite{bries})
\[
 0 \rightarrow f_*\Omega^n_{X/S}/d(f_*\Omega^{n-1}_{X/S}) \xrightarrow{df \wedge} \mathcal{H}'':=f_*\Omega^{n+1}_{X}/df\wedge d(f_*\Omega^{n-1}_{X/S}) \rightarrow  f_*\Omega^{n+1}_{X/S}  \rightarrow 0
\]
while its stalk at $s=0$ is isomorphic to $\mathcal{H}''_0=\Omega^{n+1}_{X,x}/df\wedge d(\Omega^{n-1}_{X/S,x})$ and, by the above sequence, $\mathcal{H}''$ coincides with $\mathcal{H}^n(f_*\Omega^\cdot_{X/S})$ outside of $0$.\\
Let now $\omega$ be a section of $\mathcal{H}''$ over a neighbourhood $S\subset \mathbb{C}$ around $s=0$ and consider this as a section of $\mathcal{H}^n(f_*\Omega^\cdot_{X/S})$ on $S'=S\setminus \{0\}$, those sections are called by Varchenko \cite{varchenko2} 'geometric sections'. Consider now over $S'$ the locally constant sheaf $\underline H^n=R^i f_{*}\mathbb{C}_{X'}$ and its dual $\underline H_n={\rm Hom}(\underline H^n, \mathbb{C})$ as the sheaf of homomorphisms from $\underline H^n$ to $\mathbb{C}$. For any $s \in S'$ we have $\underline H_n(s)\simeq H_n(X_s, \mathbb{C})$ and there is a natural isomorphism $T: \underline H^n(\gamma(0))\rightarrow \underline H^n(\gamma(1))$ for any smooth path $\gamma:[0,1]\rightarrow S'$ induced by the fibre bundle structure of $X'$. I.e. we obtain a morphism 
\[
M:\pi_1(S', s)\rightarrow {\rm Aut}(\underline H^n(s))
\]
whose evaluation $M(1)$ at the generator $1 \in \pi_1(S', s)$ we will call the monodromy $M$ of $f$. We can 'sheafify' these topological constructions and the result coincides with he Gauss-Manin connection restricted to $\mathcal{H}^p(f_*\Omega^\cdot_{X'/S'})$ described above. Dualizing the above topological notion of parallel transport to isomorphisms $T^*: \underline H_n(\gamma(0))\rightarrow \underline H_n(\gamma(1))$ for smooth paths $\gamma$ as above, we can consider a covariant constant (multivalued) section $\delta$ of $\underline H^n$ over $S'$. Let $s(\omega)$ be the section of $\underline H^n$ over $S'$ represented by $\omega$, then by a Theorem of Malgrange (\cite{malg2}) one has for the dual pairing of $\delta$ and $s(\omega)$ over $S'$:
\begin{theorem}\label{malgrangeth}
The series 
\begin{equation}
(s(\omega), \delta)(t)= \sum_{\alpha}\sum_{k=0}^n\frac{1}{k!}a_{k,\alpha}t^\alpha ({\rm ln}\ t)^k
\end{equation}
where $\alpha >-1$, $e^{-2\pi i\alpha}$ is an eigenvalue of $M$, converges in each sector $a< {\rm arg} t< b$ if $0<|t|$ is sufficiently small in $S'$.
\end{theorem}
Furthermore, the coefficients $a_{k,\alpha}$ depend linearly on the section $\delta$, which implies (cf. Varchenko \cite{varchenko2}) there is a set of covariantly constant sections $A^\omega_{k,\alpha}(t)$ of $\underline H^n(t)$ over (a eventually smaller) $S'$ so that
\[
s(\omega)(t)=\sum_{\alpha}\sum_{k=0}^n\frac{1}{k!}A^\omega_{k,\alpha}(t)t^\alpha ({\rm ln}\ t)^k
\]
and by \cite{varchenko3} for any $t, k, \alpha$ the $A^\omega_{k,\alpha}(t)$ belong to the generalized eigenspace of $M$ associated to $e^{-2\pi i\alpha}$. Then one calls the weight $\alpha(\omega)$ of $\omega$ the number $\alpha(\omega):=\{{\rm  min}(\alpha)|{\rm at \ least \ one \ of\ the\ sections \ } A^\omega_{0,\alpha}(t),\dots, A^\omega_{n,\alpha}(t)\neq 0\}$. Then the {\it principal part} of $s(\omega)$ is defined as
\[
s_{max}(\omega)(t)=\sum_{k=0}^n\frac{1}{k!}A^\omega_{k,\alpha(\omega)}(t)t^{\alpha(\omega)} ({\rm ln}\ t)^k,
\]
and the prinicipal parts of geometric sections of one weight are linearly independent at all points $t\in S'$ if they are at one point $t$. Then one calls the Hodge filtration of each $\underline H^n(t)$ the sequence of subspaces $\{F^p\}$ in $\underline H^n(t)$ generated by the principal parts of all geometric sections $\omega$ of $f$, evaluated at $t$, so that $\alpha(\omega)\leq n-p$. Note that since $s(f\omega)=ts(\omega)$, we have $F^{p+1}\subset F^{p}$ and the $F^p$ form a subbundle of $\underline H^n$ (cf. \cite{varchenko3}). We can now define the {\it spectrum} of a singularity due to Varchenko \cite{varchenko2}:
\begin{Def}\label{spectrumdef}
Let the principal parts of sections $\omega^p_1, \dots, \omega^p_{j(p)} \in \mathcal{H}''$ be a basis of $F^p/F^{p+1}$, i.e. their weights satisfy $\alpha(\omega^p_j)\in (n-p-1,n-p]$. Then the union of all such weights $\alpha(\omega^p_j)$ for all geometric sections $\omega^p_j$ and $(p,j)$ satisfying the above is called the spectrum of $f$.
\end{Def}
Note that by \cite{varchenko3}, at each point $t \in S'$, $F^p$ is left invariant by the semismple part $M_s$ of $M$. So the spectrum of $f$ is just the union over all $p$ of the set of numbers $n-l_p(\lambda)$ being asscociated to each eigenvalue $\lambda$ of the action of $M_s$ on $F^p/F^{p+1}$ that satisfy $exp(2\pi i l_p(\lambda))=\lambda$ and the normalization condition $p\leq l_p(\lambda)< p+1$. It is an unordered collection of $\mu$ numbers, $\mu$ being the Milnor number of $f$. We now have the following celebrated theorem due to Varchenko \cite{varchenko2}.
\begin{theorem}\label{definvspec}
The spectrum of $f$ having an isolated singularity at the origin does not change under a deformation (depending holomorphically on the deformation parameters) of $f$ leaving its Milnor number unchanged (these deformations we will refer to as $\mu$-constant deformations).
\end{theorem}
Note that the spectrum of an isolated holomorphic singularity $f:\mathbb{C}^{n+1}\rightarrow \mathbb{C}$ is a topological invariant for $n\leq 2$. However, as Saeki shows, that result remains true for $n=3$ if $f$ is quasihomogeneous, moreover we have (cf. \cite{saeki}, \cite{varchenko2}):
\begin{theorem}\label{foureq}
Let $f$ and $g$ be quasihomogeneous polynomials with an isolated singularity at the origin in $\mathbb{C}^{n+1}$ for $n\geq 1$. Then the following four are equivalent:
\begin{enumerate}
\item  $f$ and $g$ are connected by a $\mu$-constant deformation.
\item  $f$ and $g$ are connected by a topologically constant deformation.
\item $f$ and $g$ have the same weights.
\item $f$ and $g$ have the same spectrum.
\end{enumerate}
\end{theorem}

\section{Appendix A}\label{app1}
This Appendix explains briefly some facts about the 'perturbed' Milnor fibration $\hat Y$ as introduced by Seidel (\cite{seidel}). For a quasihomogeneous polynomial $f:\mathbb{C}^{n+1}\rightarrow \mathbb{C}$, define again the usual Milnor fibres $X_u$ by $X_u=f^{-1}(u) \cap B^{2n+2}$ where $B^{2n+2}\subset \mathbb{C}^{n+1}$ is the closed unit ball and $u \in \mathbb{C}\backslash\{0\}$. Furthermore, for a fixed $m\in  \mathbb{N}, m>2$ fix a cutoff function $\psi_m$ with $\psi_m(t^2)=1$ for $t \leq 1-2/m$ and $\psi(t^2)=0$ for $t \geq 1-1/m$. Set $F_u=\{z \in B^{2n+2}|f(z) = \psi_m(|z|^2)u\}$. Then set
\begin{equation}\label{milnorapp}
f: \hat Y\ = \ \bigcup_{|u|=\delta}\ F_u \times \{u\} \longrightarrow \delta S^1,
\end{equation}
Note that for any $n\in \mathbb{N}, n>2$, there is a $\delta>0$ sufficiently small, so that we can choose $m\leq n$ in the above definition and the $F_u$ remain regular. The following Lemma states the fact that there is a basis of $H_n(F;\mathbb{C})$ which consists of Lagrangian spheres, this is used in Section \ref{chapter2sympl}, the proof relies on a two-fold application of Moser's technique and classical results about vanishing cycles of Milnor fibres and will be sketched below, for details see Ebeling \cite{ebeling} and Seidel \cite{Seidel3}, note that the Lemma is valid for any isolated singularity $f:(\mathbb{C}^{n+1},0)\rightarrow (\mathbb{C},0)$ where $f$ is holomorphic.
\begin{lemma}\label{lagrangianbasis}
Let $\mu$ be the Milnor number of $F_u$. Then $(F_u;\omega)$ contains a collection of $\mu$ embedded Lagrangian $n$-spheres $i_j:S^n_j \hookrightarrow F$ so that setting $\delta_j:=[i_j(S^n_j)] \in H_n(F,\mathbb{Z})$, for $j\in \{1,\dots,\mu\}$, the 'vanishing cycles' $(\delta_1,\dots, \delta_\mu)$ form a basis of $H_n(F;\mathbb{Z})$. Furthermore setting $A^m_u:=\{z\in F_u | \psi_m(|z|^2)=1\}$ using the notation above (\ref{milnorapp}) we can find pairs $\delta>0, m>2$ so that $i_j(S^n_j)\subset A^m_u$ for all $j=1,\dots,\mu$.
\end{lemma}
\begin{proof}
Let $g_0=-1,g_1,\dots, g_{\mu-1}$ be representatives of a basis for the $\mathbb{C}$-vectorspace $\mathcal{O}_{\mathbb{C}^{n+1}}/{\rm grad}(f)\mathcal{O}_{\mathbb{C}^{n+1}}$ an consider the miniversal unfolding of $f$ as given by
\[
\begin{split}
F:(\mathbb{C}^{n+1}\times \mathbb{C}^\mu,0)&\rightarrow (\mathbb{C},0)\\
(z,u)&\mapsto f(z)+\sum_{j=0}^{\mu-1}g_j(z)u_j. 
\end{split}
\]
We choose a representative $F:B_\epsilon \times U\rightarrow \mathbb{C}$, where $U$ is of the form $U=\Delta \times T$, where $\Delta\subset \mathbb{C}$ is a disc of radius $\eta_1$ and $T\subset \mathbb{C}^{\mu-1}$ is an open ball of radius $\eta_2$ and $\epsilon$ is chosen so that for any $u \in \Delta \times T$, the set $ \mathcal{X}_u:=\{z \in B_\epsilon|F(z,u)=0\}$ intersects $S_\epsilon=\partial B_\epsilon$ transversally. With this, set $\mathcal{X}:=\{(z,u)\in B_\epsilon\times U|F(z,u)=0\}$ and $U_t=\Delta\times \{t\}$ 
\[
\begin{split}
p:\mathcal{X}\rightarrow \Delta\times T\\
(z,u)\mapsto u.
\end{split}
\]
Let $C$ denote the set of critical points of $p$ and $D=p(C)\subset U$ its image, then it is well-known  (see Ebeling \cite{ebeling}) that $D\subset \Delta\times T$ is an analytic hypersurface and the projection $\pi:\Delta\times T \rightarrow T$, restricted to $D$, is a finite branched covering, set $D_t= D\cap U_t$. Now it is a classical result (see \cite{ebeling}, Lemma 3.9), that the function 
\begin{equation}\label{morsification}
f_{\lambda t}:= F(\cdot,(0,\lambda t))\times \{\lambda t\}:B_\epsilon\rightarrow \Delta\times \{\lambda t\},
\end{equation}
is a Morse function for generic values of $t\in T$ and $\lambda \neq 0$ (namely, for $\mathbb{C}\times \{\lambda t\}$ intersecting $D$ in regular points transversally).
Now consider $F_{\lambda t,u}=\{z \in B_\epsilon|f(z) = \psi_m(|z|^2)(u-\sum_{j=1}^{\mu-1}g_j(z)u_j)\}$ for any $u\in \Delta$ and we choose $m>0$ so that all critical points lie in $A^m_u:=\{z\in F_u | \psi_m(|z|^2)=1\}$, this is possible since the set of critical points of $f_{\lambda t}$ are disjoint from the boundary of $B_\epsilon$.
Now set
\begin{equation}\label{perturbedmors}
f_{\lambda t}: \hat Y_{\lambda t}\ = \ \bigcup_{|u| \in \Delta}\ F_{\lambda t,u} \times \{u\} \longrightarrow \Delta,
\end{equation}
and observe that the smooth fibres $\hat Y_{\lambda t, u}$ are diffeomorphic to the fibres $F_u=\hat Y_{u}, u \neq 0$ of $\hat Y$. We now claim that they are also symplectomorphic.  For this, note first that $H^2(\hat Y_u,\partial \hat Y_u,\mathbb{C})=0$ as can be for instance seen by using the long exact sequence of the pair $(\hat Y_u;\partial \hat Y_u)$. Then fix one $\{\lambda t\}\in T$ as above with $t$ sufficiently small, so that with $0\leq \lambda'<\lambda$ one can choose a path $u$ in $U-D$ so that $u(\lambda') \in U_{\lambda' t} \setminus D_{\lambda 't}$ for any such $\lambda'$. Fix a set of diffeomorphisms $\phi_{\lambda'}:\hat Y_{u(0)}\rightarrow \hat Y_{\lambda' t,u(\lambda')}$ and consider the family of symplectic forms on $\hat Y_{u(0)}$
\[
\omega_{\lambda'}=\phi_{\lambda'}^*\omega_{\hat Y_{\lambda' t, u(\lambda')}}
\]
where the $\omega_{\hat Y_{\lambda_t, u(\lambda')}}$ are the symplectic forms on $\hat Y_{\lambda_t, u(\lambda')}$ induced by restriction. Then, by a version of Moser's argument, the vanishing of $H^2(\hat Y_u,\partial \hat Y_u,\mathbb{C})$ implies that there is a diffeomorphism $\psi_\lambda:\hat Y_{u(0)}\rightarrow \hat Y_{u(0)}$ so that $\psi_\lambda^*\omega_\lambda=\omega_{\hat Y_{u(0)}}$ which was the assertion.\\
We now fix a $\lambda \neq 0$ and a $t$ as above and show that each critical point $p_i, i=1,\dots,\lambda$ of $f_{\lambda t}$ gives rise to an embedded Lagrangian  sphere $S_i$ in a nearby fibre. We will work with the (singular) fibration $Y$ determined by (\ref{morsification}) and then embed the $S_n$ into the corresponding fibres in (\ref{perturbedmors}). Around any $p_i$, that is, on a small (closed) ball $B_{\epsilon_i}$ around $p_i$, $f_{\lambda t}$ can be written as
\begin{equation}\label{morsecoord}
f(z_0,\dots, z_n)=f(p_i)+z_0^2+\dots+z_n^2.
\end{equation}
Let $\eta_3>0$ so that $f(B_{\epsilon_i})\subset \Delta_{\eta_3}$, where $\Delta_{\eta_3}$ is a small (closed) disk around $s_i=f(p_i)$ and set $Y_u:=f^{-1}(u)\cap B_{\epsilon_i}$ for any $u \in\Delta_ {\eta_3}$. Then, assuming $s_i=0$, $Y_{\eta_3}$ has the the description
\[
Y_{\eta_3}=\{x+iy \in \mathbb{C}^{n+1}||x|^2+|y|^2\leq \epsilon_i^2, |x|^2-|y|^2=\eta_3, <x,y>=0\}
\]
writing $<,>$ for the Euclidean scalar product on $\mathbb{R}^{n+1}$. But this set is easily seen (see \cite{ebeling}, Lemma 5.2) to be diffeomorphic to the $1$-disk-bundle of $T^*S^n$, which has the description
\[
DS^n=\{u+iv\in \mathbb{C}^{n+1}||u|=1,|v|\leq1, <u,v>=0\},
\]
denote the diffeomorphism by $\psi_i:DS^n\rightarrow Y_{\eta_3}$. Assuming now $Y_{\eta_3}$ would carry the standard symplectic structure relative to the Morse coordinate system chosen in (\ref{morsecoord}) we could infer that $\psi_i$ is a symplectomorphism, since that fact is well-known (see McDuff/Salamon \cite{duff}), which means $\psi_i$ is a symplectic embedding of the zero section in $DS^n$, hence Lagrangian. This is in general not the case, however, as Seidel (\cite{Seidel3}) shows, there is an {\it exact} isotopy $\omega_t, t\in [0,1]$ of symplectic forms on $B_{\epsilon_i}$ so that for for some $ 0 <\tilde \epsilon_i \leq \epsilon_i$, the isotopy is constant on the boundary of $B_{\epsilon_i}$, and on $B_{\tilde \epsilon_i}$, $\omega_0$ coincides with the given one whereas $\omega_1$ is the standard Kaehler form 
\[
\omega_1=\sum_{i=0}^n z_i\wedge \overline z_i
\]
with respect to the coordinates used in (\ref{morsecoord}) on  $B_{\tilde \epsilon_i}$. Hence there is a self-diffeomorphism $\tilde \psi_i$ of $B_{\epsilon_i}$ so that $\tilde \psi_i^*\omega_1=\omega_0$, hence $\Psi:=\tilde \psi_i^{-1}\circ \psi_i:DS^n\rightarrow Y_{\eta_3}\cap B_{\epsilon_i}$ gives the desired Lagrangian embedding $i_i=\Psi_i|S^n:S^n \hookrightarrow Y_{\eta_3}$ for $i=1,\dots, \mu$.\\
We now consider the $S^n_j, j=1,\dots,\mu$ as embedded into a set of distinct fibres $\hat Y_{u_j}$ in $\hat Y_{\lambda t}$. Choose an arbitrary smooth fibre $\hat Y_{\lambda t,u}$ and use symplectic parallel transport in $\hat Y_{\lambda t}$ (which is well-defined) to parallel transport the $S^n_j$ along a set of 'weakly distinguished' paths (see \cite{ebeling}) into $Y_{\lambda t,u}$, giving a set of embedded Lagrangian spheres $\tilde S^n_j$ which by Theorem 5.6 of \cite{ebeling} and by the fact that $\hat Y_{\lambda t,u}$ is symplecticomorphic to $\hat Y_u$ induces the required basis $\delta_j:=[i_*(\tilde S^n_j)]\in H_n(\hat Y_{u},\mathbb{Z}),\ j=1,\dots,\mu$. Finally, the last assertion was proven in Lemma \ref{conditionsproof}.
\end{proof}

\end{document}